\documentclass[reqno,11pt]{amsart}
\usepackage{amsmath,amssymb,amsthm,graphicx,mathrsfs,bbm,url}
\usepackage{amsthm}
\usepackage{wrapfig}
\usepackage{enumitem}
\usepackage{mathtools}
\usepackage[utf8]{inputenc}
\usepackage[usenames,dvipsnames]{color}
\usepackage[colorlinks=true,linkcolor=Red,citecolor=Green]{hyperref}
\usepackage[super]{nth}
\usepackage[open, openlevel=2, depth=3, atend]{bookmark}
\hypersetup{pdfstartview=XYZ}
\usepackage[font=footnotesize]{caption}
\usepackage{a4wide}
\usepackage{soul}
\setstcolor{red}

\usepackage{epstopdf}
 
\usepackage{hyperref}

\theoremstyle{plain}
\newtheorem{theorem}{Theorem}[section]
\newtheorem*{theorem*}{Theorem}

\newtheorem{lemma}[theorem]{Lemma}
\newtheorem{proposition}[theorem]{Proposition}
\newtheorem{corollary}[theorem]{Corollary}

\newtheorem{conjecture}[theorem]{Conjecture}
\newtheorem{question}[theorem]{Question}
\newtheorem{claim}[theorem]{Claim}

\theoremstyle{definition}
\newtheorem{definition}[theorem]{Definition}
\newtheorem{example}[theorem]{Example}

\theoremstyle{remark}
\newtheorem{remark}[theorem]{Remark}

\numberwithin{equation}{section}

\newcommand{\C}{\mathbb{C}}
\newcommand{\R}{\mathbb{R}}

\newcommand{\Z}{\mathbb{Z}}

\newcommand{\M}{\mathcal{M}}
\newcommand{\N}{\mathbb{N}}

\newcommand{\X}{\mathbf{X}}

\newcommand{\HH}{\mathbb{H}}
\newcommand{\Ss}{\mathbb{S}}

\newcommand{\eps}{\varepsilon}

\newcommand{\mc}{\mathcal}

\newcommand{\la}{\lambda}

\newcommand{\dd}{\mathrm{d}}
\newcommand{\pl}{\partial}
\newcommand{\cjg}{\langle}
\newcommand{\cjd}{\rangle}

\DeclareMathOperator{\Op}{Op}
\DeclareMathOperator{\WF}{WF}

\DeclareMathOperator{\ran}{ran}

\DeclareMathOperator{\Div}{div}

\DeclareMathOperator{\supp}{supp}

\DeclareMathOperator{\comp}{c}

\newcommand{\be}{\begin{equation}}
\newcommand{\ee}{\end{equation}}

\title [Local lens rigidity]{Local lens rigidity for manifolds of Anosov type}

\author{Mihajlo Ceki\'c}
\address{Institut f\"ur Mathematik, Universit\"at Z\"urich, Winterthurerstrasse 190, CH-8057 Z\"urich, Switzerland}
\email{mihajlo.cekic@math.uzh.ch}

\author{Colin Guillarmou}
\address{Université Paris-Saclay, CNRS, Laboratoire de mathématiques d'Orsay, 91405, Orsay, France.}
\email{colin.guillarmou@universite-paris-saclay.fr}

\author{Thibault Lefeuvre}
\address{Université de Paris and Sorbonne Université, CNRS, IMJ-PRG, F-75006 Paris, France.}
\email{tlefeuvre@imj-prg.fr}

\begin{document}

\begin{abstract}
The \emph{lens data} of a Riemannian manifold with boundary is the collection of lengths of geodesics with endpoints on the boundary together with their incoming and outgoing vectors. We show that negatively-curved Riemannian manifolds with strictly convex boundary are \emph{locally lens rigid} in the following sense: if $g_0$ is such a metric, then any metric $g$ sufficiently close to $g_0$ and with same lens data is isometric to $g_0$, up to a boundary-preserving diffeomorphism. More generally, we consider the same problem for a wider class of metrics with strictly convex boundary, called metrics of \emph{Anosov type}. We prove that the same rigidity result holds within that class in dimension $2$ and in any dimension, further assuming that the curvature is non-positive.
\end{abstract}

\maketitle

\section{Introduction}

\subsection{The lens rigidity problem}

Let $(M,g)$ be a smooth compact connected Riemannian manifold with strictly convex boundary (i.e. the second fundamental form is positive on $\partial M$). Let $\mc{M}:=SM$ be the unit tangent bundle of $(M,g)$ and define the incoming (-) and outgoing (+) boundary of $\mc{M}$ as:
\[
\partial_\pm \mc{M} := \left\{ (x,v) \in \mc{M} ~|~ x \in \partial M, \pm g_x(v,\nu(x)) > 0 \right\},
\]
where $\nu$ is the unit outward pointing normal vector to the boundary. For any $(x,v) \in \partial_-\mc{M}$, the maximally extended geodesic $\gamma_{(x,v)}$, with initial condition $\gamma_{(x,v)}(0)=x$, $\dot{\gamma}_{(x,v)}=v$ is defined on a time interval $[0,\ell_g(x,v)]$ where $\ell_g(x,v)\in \R_+\cup \{\infty\}$. 
When $\ell_g(x,v)<\infty$, we define 
\[S_g(x,v) := \big(\gamma_{(x, v)}(\ell_g(x,v)), \dot{\gamma}_{(x, v)}(\ell_g(x,v))\big)\]
to be the outgoing tangent vector at $\pl_+\mc{M}$, see Figure \ref{figure:counter-example}.

\begin{definition}[Lens data]\label{def:lens-data}
The map $S_g:\pl_-\mc{M}\setminus \{\ell_g=\infty\}\to \pl_+\mc{M}$ is called the \emph{scattering map} and the function $\ell_g: \pl_-\mc{M}\setminus \{\ell_g=\infty\} \to \mathbb{R}_+$ the \emph{length map}. The pair $(\ell_g,S_g)$ is the \emph{lens data} of the Riemannian manifold $(M,g)$.
\end{definition}

The lens data encodes the boundary data one can measure on the geodesic flow
 from ``outside of the manifold''. A natural inverse problem that arises from tomography consists in determining the 
 geometry, namely, the Riemannian metric $g$ inside $M$, from the measurement of the lens data  $(\ell_g,S_g)$. In geophysics, this is related to recovering the speed of propagation of waves inside a domain such as the Earth, for instance, see \cite{Paternain-Salo-Uhlmann-14-1}. When two metrics $g$ and $g'$ agree on $\pl M$, it makes sense to say that they have the same lens data as there is a natural identification
between the boundary of their respective unit tangent bundles via the unit disk bundle of the boundary, see Section \S\ref{sssection:unit} for further details. The \emph{lens rigidity problem} is concerned with the following question: 

\begin{question}
Assume that $(M,g)$ and $(M',g')$ are two Riemannian metrics with strictly convex boundary such that there exists an isometry $I \in \mathrm{Diff}(\partial M,\partial M')$ with $I^*(g'|_{T \partial M'}) = g|_{T \partial M}$. Does the following implication
\[
(\ell_g,S_g)= I^*(\ell_{g'},S_{g'}) \Longrightarrow \exists \psi \in {\rm Diffeo}(M, M'),\,  \psi|_{\pl M}=I, \quad \psi^*g'=g
\]
hold true?
\end{question}

We say that a manifold $(M,g)$ is \emph{lens rigid} if there is no other Riemannian manifold (up to isometry) having the same lens data as 
$(\ell_g,S_g)$. In the following, in order to simplify the notation, we will assume that $M=M', I = \mathrm{id}$.

\begin{center}
\begin{figure}[htbp!]
\includegraphics[scale=1.5]{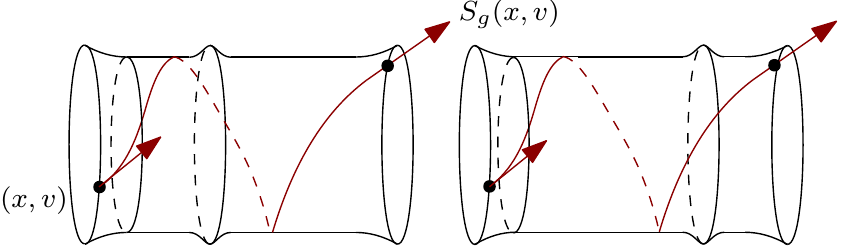}
\caption{A surface with strictly convex boundary which is not lens rigid. Example taken from \cite{Croke-Herreros-16}.}
\label{figure:counter-example}
\end{figure}
\end{center}

There are  simple counter-examples of manifolds for which lens rigidity does not hold: considering certain perturbations of the flat cylinder $\Ss^1 \times [0,1]$ (see Figure \ref{figure:counter-example} and \cite{Croke-Herreros-16} where this is further discussed), one can easily obtain non-isometric metrics with same lens data. 
Such cases have \emph{trapped geodesics}, that is some maximally extended geodesics with infinite length, or equivalently $\ell_g(x,v)=\infty$ for some 
$(x,v)\in \pl_-\mc{M}$. It turns out that all existing counter-examples to lens rigidity have trapped geodesics.

\subsection{Lens rigidity for non-trapping manifolds}

Even among manifolds without a trapped set, the lens rigidity problem is still widely open. The closest result in this direction is the recent breakthrough of Stefanov-Uhlmann-Vasy \cite{Stefanov-Uhlmann-Vasy-21}, showing lens rigidity in dimension $n\geq 3$ under the additional assumption that the manifold $(M,g)$ is foliated by 
strictly convex hypersurfaces. This includes all simply connected non-positively curved manifolds with strictly convex boundary.  
In the class of real analytic metrics such that from each $x\in \pl M$ there is a maximal geodesic free of conjugate points, the lens rigidity was proved by Vargo \cite{Vargo-09}. A local lens rigidity result was also proved near analytic metrics by Stefanov-Uhlmann \cite{Stefanov-Uhlmann-09} under certain assumptions on 
the conjugate points.

There is also a subclass of metrics that have attracted a lot of attention since the work of Michel \cite{Michel-81}, namely the class of \emph{simple manifolds}, which are manifolds with strictly convex boundary that have no trapped geodesics and no conjugate points. These manifolds are diffeomorphic to the unit ball in $\R^n$.
In this case, knowing the lens data is equivalent to knowing the restriction  $d_g|_{\pl M\times \pl M}$ of the Riemannian distance function $d_g\in C^0(M\times M)$ to the boundary, also called the \emph{boundary distance}. The lens rigidity problem for this subclass of metrics is also called the \emph{boundary rigidity problem}.
In dimension $n=2$, it was proved by Otal \cite{Otal-90-2} (in negative curvature), Croke \cite{Croke-91} (in non-positive curvature), and Pestov-Uhlmann \cite{Pestov-Uhlmann-05} (in general) that simple surfaces are boundary rigid, and thus lens rigid. 
We also mention the results by Croke-Dairbekov-Sharafutdinov  \cite{Croke-Dairbekov-Sharafutdinov-00} and Stefanov--Uhlmann \cite{Stefanov-Uhlmann-04} for local boundary rigidity results, the work by Gromov \cite{Gromov-83} and Burago--Ivanov \cite{Burago-Ivanov-10} for rigidity results of flat and close to flat simple manifolds, and we finally refer more generally to the review article by Croke \cite{Croke-04} and the recent book of Paternain--Salo--Uhlmann \cite{Paternain-Salo-Uhlmann-book} for an overview of the boundary rigidity problem.

\subsection{Lens rigidity for manifolds with non-empty trapped set} Trapped geodesics appear in most situations since all Riemannian manifolds $(M,g)$ with strictly convex boundary and non-trivial topology, i.e. non-trivial fundamental group, always have trapped geodesics (and they even have closed geodesics in the interior $M^\circ$). As far as manifolds with trapped geodesics are concerned, very little is known on the 
lens rigidity problem. It is not even clear what would be the most general class of manifolds for which lens rigidity could hold and the example above in Figure \ref{figure:counter-example} shows that it seems hopeless to consider general manifolds with both trapped geodesics and conjugate points.

The only available result considering cases with both trapped geodesics and conjugate points seems to be the local rigidity result of Stefanov-Uhlmann \cite{Stefanov-Uhlmann-09}. In dimension $n\geq 3$, under a certain topological assumption, it is proved that if $(M,g_0)$ is real analytic\footnote{Or more generally if a certain localized X-ray transform is injective.}, with strictly convex boundary, and for each $(x,v)\in SM$ there is $w\in v^\perp$ such that the maximally extended geodesic tangent to $w$ at $x$ has finite length (it is not trapped)
and is free of conjugate points, then the following holds: if $g$ is another metric with $\|g-g_0\|_{C^N}$ small enough for some $N \gg 1$ and $(\ell_g,S_g)=(\ell_{g_0},S_{g_0})$, then
$g$ and $g_0$ are isometric via a boundary-preserving diffeomorphism. On the other hand, it is not clear (geometrically speaking) what type of manifolds are contained in this class and there are many interesting geometric cases not contained in it. For example, there exist convex co-compact hyperbolic $3$-manifolds $M:=\Gamma\backslash \mathbb{H}^3$ (with constant sectional curvature $-1$) whose convex core $\mc{C}$ has positive measure and totally geodesic boundary. Thus, cutting the ends of such examples at a finite positive distance of $\mc{C}$, one obtains a metric not satisfying the assumptions of \cite{Stefanov-Uhlmann-09} due to the totally geodesic surfaces bounding $\mc{C}$.

From our point of view, there is a very natural class of metrics with non-trivial trapped set where the lens ridigity problem seems well-posed and interesting from a geometrical point of view. We call elements of this class manifolds of \emph{Anosov type}; it contains as a strict subclass the set of negatively-curved metrics with strictly convex boundary.

\begin{definition}
\label{admissiblemetrics}
A compact Riemannian manifold $(M,g)$ with boundary is of \emph{Anosov type} if:
\begin{enumerate}
\item it has strictly convex boundary,
\item no conjugate points,
\item the trapped set for the geodesic flow $(\varphi_t^g)_{t \in \R}$ on $\mc{M}:=SM$, defined by
\[
K^{g}:=\bigcap_{t\in \R}\varphi_t^{g}(\mc{M}^\circ)\subset \mc{M}^\circ,
\]
 is \emph{hyperbolic} in the following sense. There exists a continuous flow-invariant splitting
\[
\forall y\in K^{g},\quad T_y\mc{M}=\R X_g(y)\oplus E_-(y)\oplus E_+(y),
\]
where $X_g$ is the geodesic vector field, and constants $\nu,C > 0$ such that
\begin{equation}\label{hyperbolicity}
\forall \pm t\geq 0,\quad \forall y \in K^g, \quad \forall v \in E_{\mp}(y), \quad \|d\varphi^g_t(y) v\|\leq Ce^{-\nu |t|} \|v\|,
\end{equation}
for an arbitrary choice of metric $\|\cdot\|$ on $\mc{M}$.
\end{enumerate}
\end{definition}

\begin{example}
The main two examples of manifolds of Anosov type are:
\begin{enumerate}
\item Riemannian manifolds with negative sectional curvature and strictly convex boundary (see \cite[Theorem 3.2.17 and Section 3.9]{Klingenberg-95}),
\item strictly convex subdomains of closed Riemannian manifolds with Anosov geodesic flows.
\end{enumerate}
\end{example}

Manifolds of Anosov type have a trapped set with fractal structure and zero Lebesgue measure. It implies that almost-every point in $\mc{M}$ is reachable from geodesics with endpoints on $\pl \mc{M}$.
This case can be interpreted as an intermediate rigidity problem between the \emph{length spectrum rigidity} of manifolds with Anosov geodesic flows, where one asks if the lengths of closed geodesics determine the metric up to isometry, and the boundary rigidity problem of simple manifolds. 

In the closed case, Vign\'eras \cite{Vigneras-80} exhibited counter-examples to the length spectrum rigidity: in constant negative curvature, there are non-isometric metrics on surfaces with the same length spectrum. The well-posed rigidity problem is rather that of the \emph{marked length spectrum} problem, also known as the Burns-Katok conjecture \cite{Burns-Katok-85}: on a manifold $(M,g)$ with Anosov geodesic flow, each free homotopy class of loops $c$ on $M$ contains a unique geodesic representative $\gamma_c(g)$ whose length is denoted by $L_g(c)$; if $g_1$ and $g_2$ are two such Anosov metrics on $M$ with $L_{g_1}(c)=L_{g_2}(c)$ for all $c$, it is then conjectured that $g_1$ should be isometric to $g_2$. This conjecture was proved in dimension $2$ by Otal \cite{Otal-90} and Croke \cite{Croke-90}, and in all dimensions for pairs of metrics that are close enough in $C^k$ norm for $k \gg 1$ large enough by the last two authors \cite{Guillarmou-Lefeuvre-18} (local rigidity). However, it is still open in general.

Similarly, for manifolds with boundary and non-trivial topology, the same problem of ``marking" of geodesics is a serious difficulty. The first natural question one may consider is the following, known as \emph{marked lens rigidity} or \emph{marked boundary rigidity} problem for Riemannian manifolds of Anosov type. 

\begin{definition}[Marked lens data]
Let $g_1,g_2$ be two metrics of Anosov type on $M$. We say that $g_1$ and $g_2$ have the same \emph{marked lens data} if for each $(x,v)\in \pl_-\mc{M}\setminus \{\ell_g = \infty\}$ one has $(\ell_{g_1}(x,v),S_{g_1}(x,v))=(\ell_{g_2}(x,v),S_{g_2}(x,v))$ and the $g_1$- and $g_2$-geodesics with initial conditions $(x,v)$ are homotopic via a homotopy fixing the endpoints.
\end{definition}

Technically, having same marked lens data is the same as having same boundary distance function on the universal cover $\widetilde{M}$ (which is now a non-compact space).
The following conjecture is somehow similar to the Burns-Katok conjecture in the closed case and to the boundary rigidity problem of negatively curved simple metrics:

\begin{conjecture}[Marked lens rigidity of manifolds of Anosov type]
\label{conjecture2}
Let $M$ be a smooth manifold with boundary and assume that $g_1,g_2$ are two smooth metrics of Anosov type on $M$ in the sense of Definition \ref{admissiblemetrics}, such that $g_1|_{T(\partial M)} = g_2|_{T(\partial M)}$. If $g_1$ and $g_2$ have same marked lens data, then there exists a smooth diffeomorphism $\psi$, homotopic to the identity and equal to the identity on the boundary $\partial M$, such that $\psi^*g_2=g_1$.
\end{conjecture}

In dimension $2$, Conjecture \ref{conjecture2} was recently solved by the third author with Erchenko in \cite{Erchenko-Lefeuvre-23} (an earlier result had also been obtained by the second author together with Mazzuchelli in \cite{Guillarmou-Mazzucchelli-18} for negatively-curved surfaces using the method of Otal \cite{Otal-90}). In higher dimension, the third author \cite{Lefeuvre-19-2} proved Conjecture \ref{conjecture2} for pairs of negatively-curved metrics  $g_1,g_2$ that are close enough in $C^k$ norm for $k \gg 1$ large enough (local marked lens rigidity). The fact that there is no smooth $1$-parameter family $(g_s)_{s \in (-1,1)}$ of non-isometric negatively-curved 
metrics with the same marked lens data\footnote{In this case, having the same marked lens data is equivalent to having the same lens data.} is called \emph{infinitesimal rigidity} and was first proved by the second author \cite{Guillarmou-17-2}.

In this paper, we consider the more difficult problem of lens rigidity in the class of manifolds of Anosov type. 
Since, contrary to the closed case, there are still no counter-examples to lens rigidity, we make the following conjecture of lens rigidity in the class of metrics of Anosov type:

\begin{conjecture}[Lens rigidity of manifolds of Anosov type]
\label{conjecture3}
Let $(M_1,g_1)$, $(M_2,g_2)$ be two smooth Riemannian manifolds of Anosov type such that 
$(\pl M_1,g_1|_{\pl M_1})=(\pl M_2,g_2|_{\pl M_1})$. If $(\ell_{g_1},S_{g_1})=(\ell_{g_2},S_{g_2})$, then there exists a smooth diffeomorphism $\psi$, equal to the identity on the boundary, such that $\psi^*g_2=g_1$.
\end{conjecture}

There are already partial answers to Conjecture \ref{conjecture3}:
\begin{enumerate}
\item In dimension $2$, Croke and Herreros  \cite{Croke-Herreros-16} proved that negatively-curved cylinders with strictly convex boundary are lens rigid,
\item  In dimension $2$, the second author shows in \cite{Guillarmou-17-2} that the scattering map $S_g$ determines $(M,g)$ up to conformal diffeomorphism fixing the boundary. Recovering the conformal factor of the metric is still an open question.
\item In dimension $n\geq 3$, Stefanov-Uhlmann-Vasy \cite{Stefanov-Uhlmann-Vasy-21} prove that for general metrics with strictly convex boundary, the lens data determines the metric in a neighborhood of $\pl M$; applying this result in the setting of negatively curved manifold, one can recover 
the metric outside the convex core of the manifold (which contains the projection of the trapped set).
\item In \cite{Bonthonneau-Guillarmou-Jezequel-22}, Bonthonneau, Jézéquel, and the second author recently proved Conjecture \ref{conjecture3} under the extra assumption that $(M_1,g_1), (M_2,g_2)$ are real analytic, but only using the equality $S_{g_1}=S_{g_2}$ of the scattering maps.
\end{enumerate}

Our first result in this article is the following local rigidity result answering Conjecture \ref{conjecture3} for metrics close to each other.
\begin{theorem}
\label{theorem:main}
Let $(M,g_0)$ be a Riemannian manifold of Anosov type. Assume that either $\dim M=2$ or that the curvature of $g_0$ is non-positive. Then there exists $N \gg 1, \delta > 0$ such that the following holds: for any smooth metric $g$ on $M$ such that $\|g-g_0\|_{C^N} < \delta$, if $(\ell_g,S_g) = (\ell_{g_0}, S_{g_0}) $, then there exists a smooth diffeomorphism $\psi : M \to M$ such that $\psi|_{\partial M} ={\rm id}$ and $\psi^*g=g_0$.
\end{theorem}

More generally, Theorem \ref{theorem:main} holds under the general assumption that $g_0$ is of Anosov type and that its \emph{X-ray transform operator} $I_2^{g_0}$ on divergence-free symmetric $2$-tensors is injective, see \eqref{I2} for a definition of $I_2^{g_0}$ and \S\ref{sssection:x-ray} where this is further discussed. The fact that $I_2^{g_0}$ is injective on divergence-free tensors was proved in 
\cite{Guillarmou-17-2} in non-positive curvature and in general on Anosov surfaces by \cite{Lefeuvre-19-1} (without any assumption on the curvature). It was also proved in \cite{Bonthonneau-Guillarmou-Jezequel-22} that $I_2^{g_0}$ is injective for real-analytic metrics $g_0$ which implies that generic smooth metrics of Anosov type have an injective X-ray transform operator $I_2^{g_0}$; generic injectivity of $I_2^{g_0}$ follows from the work of the first and third authors \cite{Cekic-Lefeuvre-21-2} as well, admitting also Theorem \ref{axiomAsmooth} below. As a corollary of Theorem \ref{theorem:main}, we obtain:

\begin{corollary}\label{cor:negcurv}
Let $(M,g_0)$ be a negatively-curved Riemannian manifold with strictly convex boundary. Then, there exists $N \gg 1, \delta > 0$ such that the following holds: for any smooth metric $g$ on $M$ such that $\|g-g_0\|_{C^N} < \delta$, if $(\ell_g,S_g) = (\ell_{g_0}, S_{g_0})$, then there exists a smooth diffeomorphism $\psi : M \to M$ such that $\psi|_{\partial M} ={\rm id}$ and $\psi^*g=g_0$.
\end{corollary}

We observe that Corollary \ref{cor:negcurv} and Theorem \ref{theorem:main} are not a consequence of \cite{Stefanov-Uhlmann-09} (nor of \cite{Stefanov-Uhlmann-Vasy-21}) mentioned above since: 1) our result contains the case of surfaces (dimension $n=2$); 2) the assumption on the trapped set in \cite{Stefanov-Uhlmann-09} does not cover all hyperbolic trapped sets (typically, the example $M = \Gamma \setminus \HH^3$ mentioned above is not covered when the boundary of the convex core $\mc{C}$ is totally geodesic), whereas we do not make any specific assumption on the topology, and neither do we assume that $g_0$ is analytic or that it has an injective localized X-ray transform. Theorem \ref{theorem:main} is also clearly stronger than the marked local rigidity result of the third author \cite{Lefeuvre-19-2}, since we are now able to remove the \emph{marking} assumption on the lens data.\\

Let us finally mention that there are interesting and related results for Euclidean billiards: Noakes--Stoyanov \cite{Noakes-Stoyanov-15} show that the lens data for the billiard flow on $\R^n \setminus \mc{O}$ (where $\mc{O}$ is a collection of strictly convex domains) is rigid, and 
De Simoi--Kaloshin--Leguil \cite{Desimoi-Kaloshin-Leguil} prove that the lengths of the marked periodic orbits generically determine the obstacles 
under a $\Z^2\times \Z^2$ symmetry assumption. 

\subsection{Removing the marking assumption. Idea of proof} The removal of the marking  
assumption is not simply a technical artefact: it is rather a crucial aspect in our work. Indeed, without the marking assumption, one can no longer use the fact that the geodesic flows of $g$ and $g_0$ are conjugate with a conjugacy preserving the Liouville measure. This conjugacy was a fundamental aspect of both proofs of \cite{Guillarmou-Mazzucchelli-18,Lefeuvre-19-2}.  
In the proof of Theorem \ref{theorem:main}, one has to rely on 
a completely different argument, which is the linearisation of the pair $(\ell_g,S_g)$. Nevertheless, since $g$ has a big set of trapped geodesics (typically a fractal set), this 
creates many singularities for $(\ell_g,S_g)$ and its linearization.
The analysis one has to perform is then quite involved. One needs to combine several different key tools, in particular: 
\begin{enumerate}
\item the proof of the $C^2$-regularity with respect to $g$ of the operator $\mc{S}_g: C^\infty(\pl_+\mc{M})\to \mc{D}'(\pl_-\mc{M})$ defined by $\mc{S}_gf:=f\circ S_g$, 
\item the exponential decay in $t\to \infty$ of the volume of points $(x,v)\in \mc{M}=SM$ that remain trapped for time $t$.
\end{enumerate}
The first item is obtained by reproving certain results of \cite{Dyatlov-Guillarmou-16} on the resolvent of an Axiom A vector field $X$, but now with an explicit control of the dependence with respect to the vector field $X$. In particular, as a byproduct of this analysis we show the following result that could prove useful for other applications such as Fried's conjecture for manifolds with boundary, in the spirit of \cite{Dang-Guillarmou-Riviere-Shen-20}:  
 
 \begin{theorem}\label{axiomAsmooth}
Let $\mc{M}$ be a smooth manifold with boundary and let $X_0$ be a smooth vector field so that  $\pl \mc{M}$ is strictly convex for the flow of $X_0$. Assume that the trapped set $K^{X_0}:=\cap_{t\in \R}\varphi_t^{X_0}(\mc{M}^\circ)$ of the flow  $(\varphi_t^{X_0})_{t \in \R}$ of $X_0$ is hyperbolic. Then, there exist $\delta>0$, $N\gg1$, such that for all $X\in C^\infty(\mc{M},T\mc{M})$ with $\|X-X_0\|_{C^N} < \delta$, the following holds:
\begin{enumerate}
\item the resolvent $R^{X}(z):=(-X+z)^{-1}: L^2(\mc{M})\to L^2(\mc{M})$, initially defined in the half-plane 
 $\left\{z \in \mathbb{C} \mid \Re(z)\gg 1\right\}$, extends meromorphically to $\C$ as a bounded operator $R^{X}(z): C^\infty_{\comp}(\mc{M}^\circ)\to \mc{D}'(\mc{M}^\circ)$,
\item if $z_0 \in \mathbb{C}$ is not a pole of $R^{X_0}(z)$, then the map
\[
C^\infty(\mc{M},T\mc{M}) \ni X\mapsto R^{X}(z_0) \in \mc{L}(C^\infty_{\comp}(\mc{M}^\circ),\mc{D}'(\mc{M}^\circ)),
\]
is $C^2$-regular\footnote{Even though we only need $C^2$, our proof actually shows it is $C^k$ for all $k\in \mathbb{N}$.} with respect to $X$.
 \end{enumerate} 
 \end{theorem}

Here, we denote by $\mc{L}(A, B)$ the space of continuous linear maps between functional spaces $A$ and $B$. The space $\mc{L}(C^\infty_{\comp}(\mc{M}^\circ),\mc{D}'(\mc{M}^\circ))$ can be naturally identified with $\mc{D}'(\mc{M}^\circ \times \mc{M}^\circ)$ via the Schwartz kernel theorem; the space $\mc{D}'(\mc{M}^\circ \times \mc{M}^\circ)$ is equipped with the standard topology on distributions. In fact, we prove the result above in anisotropic Sobolev spaces, and refer to Theorem \ref{theorem:meromorphic} for a more detailed statement.
We show that the scattering operator $\mc{S}_g$ has a Schwartz kernel that can be written as a restriction of the Schwartz kernel of 
$R^{X_{g}}(0)$ on $\pl_-\mc{M}\times \pl_+\mc{M}$, implying that the map $g\mapsto \mc{S}_g$ is $C^2$-regular as operators acting on some appropriate Sobolev spaces.  \\

The strategy of the proof then goes as follows. First of all, we put the metric $g$ in solenoidal gauge (with respect to $g_0$), namely we find a first diffeomorphism $\psi \in \mathrm{Diff}(M)$ such that 
$\psi|_{\partial M} = {\rm id}$ and $g' = \psi^*g$ is divergence-free with respect to $g_0$, see Lemma \ref{lemma:solenoidal-gauge}.
Secondly, letting
\[
I_2^{g_0} : C^\infty(M, \otimes_S^2T^*M)\to L^\infty_{\rm loc}(\pl_-\mc{M}\setminus \{\ell_{g_0} = \infty\})
\]
be the X-ray transform on symmetric $2$-tensors with respect to $g_0$, defined as
\begin{equation}\label{I2}
I_2^{g_0}h(x, v):= \int_{0}^{\ell_{g_0}(x, v)}h_{\gamma(t)}(\dot{\gamma}(t), \dot{\gamma}(t))dt, \quad \textrm{ if }\varphi_t^{g_0}(x, v)=(\gamma(t), \dot{\gamma}(t))\in \mc{M},
\end{equation}  
we show in Section \S\ref{ssection:key} the following key estimate: there are $C, \mu>0$ such that, if $(\ell_{g_0},S_{g_0})=(\ell_g,S_g)$ and $\|g'-g_0\|_{C^N}<\delta$ for some small $\delta>0$, then
\begin{equation}
\label{equation:key-intro}
\|I_2^{g_0}(g'-g_0)\|_{H^{-6}(\partial_-\mc{M})} \leq C \|g'-g_0\|_{C^N(M, \otimes_S^2T^*M)}^{1+\mu}.
\end{equation}
The proof of this estimate is involved. It is based on some complex interpolation argument using the holomorphic map $\C \ni z\mapsto e^{-z\ell_{g_0}}I_2^{g_0}(g'-g_0)$ and the $C^2$-smoothness of the scattering map 
$g \mapsto \mc{S}_g$ as a continuous map from $C^\infty(\partial_+\mc{M})$ to $H^{-6}(\partial_-\mc{M})$. This is established in Section \S\ref{section:scattering}. It also relies on some volume estimates on the set of geodesics trapped for time $t\to \infty$ that follow from \cite{Guillarmou-17-2}.

Finally, slightly extending $(M, g_0)$ to some $(M_e, g_{0e})$, using the mapping properties of the adjoint $(I_2^{g_{0e}})^*$, interpolation arguments, and \eqref{equation:key-intro}, one obtains for $h := g' - g_0$:
\begin{equation}\label{eq:step}
	\|h\|_{L^2} \leq C  \|\Pi_2^{g_{0e}} E_0 h\|_{H^1} \leq C\|h\|_{C^N}^{1 + \mu},
\end{equation}
where $E_0$ is the zero extension operator to $M_e$, $\Pi_2^{g_{0e}} = (I_2^{g_{0e}})^* I_2^{g_{0e}}$ is the normal operator, and the estimate on the left is an elliptic estimate proved in Proposition \ref{proposition:elliptic-estimate}. It is left to interpolate $C^N$ between $L^2$ and $C^{N'}$ in \eqref{eq:step}, where $N' \gg N$, to get for some $0 < \mu ' < \mu$:
\[
\|h\|_{L^2} \leq C \|h\|_{L^2} \|h\|_{C^{N'}}^{\mu'} \leq C \|h\|_{L^2} \|g-g_0\|_{C^{N'}}^{\mu'}.
\]
For $\|g-g_0\|_{C^{N'}}$ small enough, this readily implies that $g' = \phi^*g = g_0$, concluding the proof.\\

\noindent \textbf{Acknowledgement:} 
We thank the anonymous referees for their careful readings and helpful comments that improved the paper.
This project has received funding from the European Research Council (ERC) under the European Union’s Horizon 2020 research and innovation programme
(grant agreement No. 725967). MC is further supported by an Ambizione grant (project number 201806) from the Swiss National Science Foundation.

\section{Geometric and dynamical preliminaries}

\label{section:preliminaries}

Following \cite[Section 2]{Guillarmou-17-2}, we describe the scattering and length maps in our geometric setting, and relate them to the resolvent of the geodesic flow.

\subsection{Unit tangent bundle and extensions}

\subsubsection{Geometry of the unit tangent bundle}

\label{sssection:unit}

Let $(M,g)$ be a smooth compact oriented Riemannian manifold with strictly convex boundary (in the sense that the second fundamental form is positive) and let $S^gM=\{(x,v)\in TM\,|\, |v|_{g_x}=1\}$ be the unit tangent bundle with projection on the base denoted by $\pi_0:S^gM\to M$. For a point $y=(x,v)\in S^gM$, we shall write $-y:=(x,-v)$. 
Denote by $\varphi^g_t:S^gM\to S^gM$ the geodesic flow at time $t\in \R$ and by $X_g$ its generating vector field. Let $\alpha$ be the canonical Liouville $1$-form on $S^gM$, defined by $\alpha(x, v)(\xi) := g_x(d\pi_0(x, v)\xi, v)$ for any $\xi \in T_{(x, v)}S^gM$, and define $\mu:=\alpha\wedge d\alpha^{n-1}$, the associated Liouville volume form, which we will freely identify with the Liouville measure. It satisfies $\mc{L}_{X_g}\mu=0$, where $\mc{L}_{X_g}$ denotes the Lie derivative along $X_g$.

Recall that we introduced the incoming (-) and outgoing (+) boundaries as 
\[
\pl_\pm S^gM = \{(x,v) \in \partial S^gM \mid \pm g_x(v,\nu)>0\},
\]
where $\nu$ it the outward pointing unit normal to $\partial M$. Using the orthogonal decomposition
\begin{equation}
\label{equation:decomp-intro}
T_{\partial M}M = T(\partial M) \oplus^\bot \R \nu,
\end{equation}
the boundary $\partial_\pm S^gM$ can be naturally identified with the boundary ball
\[
B(\partial M) := \left\{ (x,v) \in TM ~|~ x \in \partial M, v \in T_x(\partial M), |v|_g \leq 1 \right\}
\]
by means of the orthogonal projection onto the first factor in \eqref{equation:decomp-intro}. As a consequence, if $g'$ is any other smooth metric on $M$ such that $g|_{T\partial M} = g'|_{T\partial M}$, the boundaries $\partial_\pm S^gM$ and $\partial_\pm S^{g'}M$ can be naturally identified and it makes sense to say that $(\ell_g,S_g) = (\ell_{g'},S_{g'})$. When this equality holds, we say that the manifolds $(M,g)$ and $(M',g')$ have same \emph{lens data}. 

When we consider a set of metrics $g$, the unit tangent bundles $S^{g}M$ depend on $g$. For convenience, we will thus fix the manifold
\[
\mc{M}:=S^{g_0}M,\]
associated to an arbitrary metric of reference $g_0$. We can always rescale the flow $\varphi_t^{g}$ so that it becomes defined on $\mc{M}$. Indeed, define $\Phi_{g_0 \to g} : S^{g_0}M \to S^gM$ by $\Phi_{g_0 \to g}(x,v) := (x,v/|v|_g)$. Then $\Phi_{g_0 \to g}^{-1} \circ \varphi_t^g \circ \Phi_{g_0 \to g}$ is a flow on $\mc{M}$ which we shall still denote by $\varphi_t^g$, and its vector field will also be denoted by $X_g$ for simplicity. 

We shall always work with metrics $g$ so that $g|_{T\pl M}=g_0|_{T \pl M}$. The boundary of $\mc{M}$ splits into a disjoint union
\begin{equation}\label{eq:splitting}
	\pl \mc{M}=\pl_-\mc{M} \cup \pl_+\mc{M}\cup \pl_0\mc{M},
\end{equation}
where $\pl_\pm \mc{M}:=\{(x,v)\in \pl \mc{M}\, |\, \pm g_x(v,\nu)>0\}$ and $\pl_0 \mc{M}:=\{(x,v)\in \pl \mc{M}\, |\, g_x(v,\nu)=0\}$. Note that the normal $\nu$ depends on $g$, and that the splitting \eqref{eq:splitting} does not depend on the choice of $g=g_0$ on $T\pl M$. This will be important to compare for $g \neq g'$ the length functions $\ell_g$ with $\ell_{g'}$ and the scattering maps $S_g$ with $S_{g'}$ (see Definition \ref{def:lensmap} below).

There is a symplectic form on $\pl_\pm \mc{M}$ obtained by restricting $\iota_{\pl}^*d\alpha$ to $\pl_\pm \mc{M}$, where $\iota_\pl:\pl \mc{M}\to \mc{M}$ is the inclusion map. We denote by 
\[
\mu_\pl :=|\iota_\pl ^*(i_{X_g}\mu)|=|\iota_{\pl}^*(d\alpha)^{n-1}|,
\]
the induced measure on $\pl \mc{M}$, where $i_{X_g}$ denotes the contraction with $X_g$. In what follows we will write $L^p(\pl_\pm \mc{M})$ for the usual $L^p$ space with respect to any smooth Riemannian measure ${\rm dv}_h$ on $\pl \mc{M}$ (for some metric $h$ on $\pl \mc{M}$), while we will write $L^p(\pl_\pm \mc{M},\mu_{\pl})$ when we use the measure $\mu_{\pl}$. We note that $\mu_{\pl}=\omega\, {\rm dv}_h$ where $\omega\in C^\infty(\pl \mc{M})$ is positive outside $\pl_0\mc{M}$ and vanishes to order $1$ at $\pl_0\mc{M}$, thus $L^p(\pl_\pm\mc{M})\hookrightarrow L^p(\pl_\pm \mc{M},\mu_{\pl})$ continuously.

\subsubsection{Extension of the manifold}

\label{sssection:extension}

It will be convenient to consider an embedding of $\mc{M} \hookrightarrow \mc{N}$ into a smooth closed manifold $\mc{N}$. This can be done by considering an embedding $M \hookrightarrow N$, where $N$ is a smooth closed manifold (this is always possible by doubling the manifold $M$ across its boundary for instance, i.e. gluing $M \sqcup M$ along $\partial M$ by means of the identity map), then extending smoothly the metric $g_0$ to $N$ (denoted by $g_{0N}$) and taking $\mc{N} := S^{ g_{0N}}N$. If $g_0$ is of Anosov type (see Definition \ref{admissiblemetrics}), it will be also convenient to have a slightly larger manifold with boundary $M_e$ at our disposal such that $M \hookrightarrow M_e \hookrightarrow N$ and the extension of the metric $g_0$ to $M_e$, which we denote by $g_{0e}$, is of Anosov type, see \cite[Section 2]{Guillarmou-17-2} where this is further discussed. Set $\mc{M}_e := S^{g_{0e}}M_e$. We have the successive embeddings $\mc{M} \hookrightarrow \mc{M}_e \hookrightarrow\mc{N}$. For a metric $g$ close to $g_0$ in $C^N$ norm and such that $g=g_0$ on $T\pl M$, we consider an  extension $g_e$ of Anosov type on $M_e$. The map $g\mapsto g_e$ can be chosen to be smooth and so that $\|g_e-g_{0e}\|_{C^N(M_e,\otimes_S^2T^*M_e)}\leq C_N\|g-g_0\|_{C^N(M,\otimes_S^2T^*M)}$ for all $N\geq 0$ and some constants $C_N>0$, where $\otimes_S^2T^*M$ is the bundle of symmetric $2$-tensors.

\begin{definition}
Let $c \in \mathbb{R}$. We say that a level set $\left\{\rho=c\right\}$ of a function $\rho \in C^\infty(\mc{N})$ is \emph{strictly convex} with respect to a vector field $Y \in C^\infty(\mc{N},T\mc{N})$ if for all $y \in \left\{\rho=c\right\}$, one has:
\[
Y \rho(y) = 0 \implies Y^2\rho(y) < 0.
\]
We say that a smooth submanifold $\mc{H} \subset \mc{N}$ is strictly convex with respect to $Y$ if $\mc{H}$ is in a neighbourhood of $\mc{H}$ given by a level set $\left\{\rho = 0 \right\}$ of some function $\rho$, and this level set is strictly convex with respect to $Y$. This is independent of the choice of $\rho$. \end{definition}
It can be easily checked that $(M, g_0)$ has strictly convex boundary in the Riemannian sense if and only if $\partial \mc{M}$ is strictly convex with respect to the geodesic vector field $X_{g_0}$.

We now consider an arbitrary smooth extension $\widetilde{X}_{g_0}$ of $X_{g_{0e}}|_{\mc{M}_e}$ to $\mc{N}$. Let $\rho \in C^\infty(\mc{N})$ be a global boundary defining function for $\mc{M}$, i.e. such that $\rho > 0$ on the interior of $\mc{M}$, $\partial \mc{M} = \left\{\rho=0\right\}$ and $\rho < 0$ on $\mc{N} \setminus \mc{M}$. Since $X_{g_0}$ does not vanish on $\mc{M} = \left\{\rho \geq 0\right\}$, we can consider $\rho_0 > 0$ small enough such that $\widetilde{X}_{g_0}$ does not vanish in $\left\{\rho > -2 \rho_0 \right\}$. A continuity argument shows that, for all $\rho_0 > 0$ small enough, the level set $\left\{\rho=-\rho_0\right\}$ is strictly convex with respect to $\widetilde{X}_{g_0}$.
We can assume that 
\[\mc{M}_e=\{x\in \mc{N}\, |\, \rho(x)\geq -\rho_0/2\}.\]
In the following, we will consider smooth perturbations $X$ of the vector field $X_{g_0}$ in $\mc{M}$ (small in the $C^N$-topology, for $N \gg 1$ large enough). They will mostly be induced by a metric $g$ close to $g_0$ but it might be better to have in mind a more general picture than just geodesic flows. It will be convenient to extend the vector fields $X_g$ to vector fields $\widetilde{X}_g$ on $\mc{N}$ such that $\widetilde{X}_g = \widetilde{X}_{g_0}$ on the set $\left\{\rho \leq -2\rho_0/3\right\}$ and $\widetilde{X}_g=X_{g_e}$ on $\mc{M}_e$. 
Moreover, it is possible to construct such an extension with, for any $N \in \mathbb{N}$
\[
\|\widetilde{X}_g-\widetilde{X}_{g_0}\|_{C^N(\mc{N},T\mc{N})} \leq C\|X_g-X_{g_0}\|_{C^N(\mc{M},T\mc{M})},
\]
for some constant $C > 0$ (depending only on $\mc{M}$, $\mc{N}$, and $N$). Also observe that strict convexity of the boundary is stable by $C^2$-perturbation of the vector field.

We introduce the smooth function $\psi \in C^\infty(\mc{N})$ with values in $[-1, 1]$ such that:
\begin{itemize}
\item $\psi = \rho + \rho_0$ on the set $\left\{-\rho_0-\rho_0/10 \leq \rho \leq -\rho_0+\rho_0/10\right\}$,
\item $\psi = 1$ on $\mc{M} = \left\{\rho \geq 0\right\}$, and $\psi > 0$ on $\{\rho > -\rho_0\}$,
\item  $\psi = -1$ on $\left\{\rho \leq -2\rho_0\right\}$, and $\psi < 0$ on $\{\rho < - \rho_0\}$.
\end{itemize}
With some abuse of notation, we then denote by $X$ (resp. $X_0$) the vector field on $\mc{N}$ defined by $X := \psi \widetilde{X}_g$ (resp. $X_0 := \psi \widetilde{X}_{g_0}$). This construction ensures that the restriction of $X$ to $\mc{M}$ is the original vector field initially defined on $\mc{M}$ and that $\{\rho\geq -\rho_0 \}$ is preserved by all the flows $(\varphi_t^X)_{t \in \R}$ for all $t\in \R$, and finally that each trajectory leaving $\mc{M}$ never comes back to $\mc{M}$, with the same property for $\mc{M}_e$. See Figure \ref{figure:extension} for a visual summary of this construction.

\begin{center}
\begin{figure}[htbp!]
\includegraphics[scale=0.75]{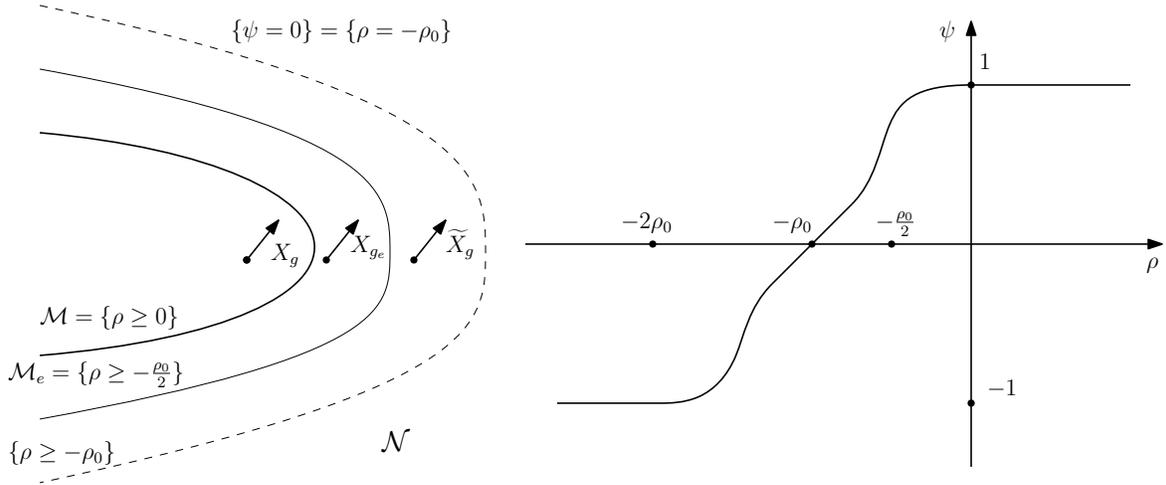}
\caption{One the left: the extension of the vector field $X_g$ from $\mc{M}$ to $X_{g_e}$ on $\mc{M}_e$, and further to $\widetilde{X}_g$ on $\mc{N}$. The vector field $X = \psi \widetilde{X}_g$ is \emph{complete} on the set $\{\rho \geq -\rho_0\}$ and vanishes on $\{\rho = -\rho_0\}$. On the right: the auxiliary function $\psi$ as a function of $\rho$.}
\label{figure:extension}
\end{figure}
\end{center} 
\vspace{-8mm}

\subsection{Scattering and length maps}

For $(x,v)\in \mc{M}$, the escape time $\tau_g(x,v)$ is defined to be the maximal time of existence of the integral curve $(\varphi^{g}_t(x,v))_{t\geq 0}$ in $\mc{M}$:
\[ \tau_g: \mc{M}\to [0,\infty], \quad \tau_g(x,v):=\sup \{t\geq 0 \, | \, \varphi^g_t(x,v)\in \mc{M} \}.\]
The forward (-) and backward (+) trapped sets $\Gamma^g_\pm$ are defined by 
\[
\Gamma^g_\pm :=\{ (x,v)\in \mc{M}\, |\, \tau_g(x,\mp v)=\infty\},
\] 
they are closed sets in $\mc{M}$ and the trapped set is the closed invariant set 
\[
K^g:=\Gamma^g_+\cap \Gamma^g_-=\bigcap_{t\in \R}\varphi_t^g(\mc{M}).
\]
Since $\pl M$ is strictly convex, it is direct to check that $\Gamma_\mp ^g\cap \pl_\pm \mc{M}=\emptyset$ and $K^g\cap \pl \mc{M}=\emptyset$. We now recall the definition (see Definition \ref{def:lens-data}) of the \emph{lens data}:

\begin{definition}[Lens data]
\label{def:lensmap}
The length map $\ell_g:\pl_-\mc{M}\setminus \Gamma^g_-\to \R_+$ and the scattering map $S_g:\pl_-\mc{M}\setminus \Gamma^g_-\to \pl_+\mc{M}\setminus \Gamma^g_+$ are defined by 
\[ \ell_g(x,v):=\tau_g(x,v), \quad S_{g}(x,v):=\varphi^g_{\tau_g(x,v)}(x,v).\]
The pair $(\ell_g,S_g)$ is called the lens data of $(M,g)$.
\end{definition}

When unnecessary, we will drop the index $g$ in the notation. It will be convenient to view the scattering map as acting on functions on $\pl_+\mc{M}$ by pull-back. We define the \emph{scattering operator} as 
\[
\mc{S}_g: C^\infty_{\comp}(\partial_+ \mc{M} \setminus \Gamma^g_+) \to C^\infty_{\comp}(\partial_-\mc{M} \setminus \Gamma^g_-), \qquad \mc{S}_g \omega := \omega \circ S_g.
\]
Under the assumption that $\mu_\pl\left(\left(\Gamma^g_-\cup\Gamma^g_+\right)\cap \pl \mc{M}\right))=0$, it is not difficult to show (see \cite[Lemma 3.4]{Guillarmou-17-2}) that for all $f \in C^\infty_{\comp}(\partial_+ \mc{M} \setminus \Gamma_+)$, one has
\[
\|\mc{S}_gf\|_{L^2(\pl_-\mc{M},\mu_\partial)}=\|f\|_{L^2(\pl_+\mc{M},\mu_\partial)},
\]
and thus $\mc{S}_g$ extends continuously to an isometry $L^2(\pl_+\mc{M},\mu_\partial)\to L^2(\pl_-\mc{M},\mu_\partial)$. The scattering operator $\mc{S}_g$ determines $S_g$ and conversely.

By the implicit function theorem (since $\pl M$ is strictly convex), we also have that 
\[
\tau_g\ \in C^\infty\left(\mc{M}\setminus (\Gamma_-^g \cup \partial_0 \mc{M})\right),\quad  \ell_g\in C^\infty\left(\overline{\pl_-\mc{M}} \setminus \Gamma_-^g\right),
\]
(here $\overline{\partial_-\mc{M}} = \partial_0 \mc{M} \cup \partial_-\mc{M}$), see \cite[Lemmas 4.1.1 and 4.1.2]{Sharafutdinov-94} for further details. Since we shall need the dependence of $\ell_g$ with respect to $g$, we first prove a result outside the trapped sets:

\begin{lemma}
\label{lemma:bound-dell}
Let $(M,g_0)$ be a smooth compact Riemannian manifold with strictly convex boundary and let $p \in \N$. There exists $\eps > 0$ small enough such that the following holds: for all metrics $g \in U_{g_0}$, where
\begin{equation}
\label{equation:neighb}
U_{g_0} := \left\{g \in C^{p+2}(M,\otimes^2_S T^*M) ~|~ \|g-g_0\|_{C^{p+2}} < \eps,\, g|_{T\partial M}=g_0|_{T\partial M}\right\},
\end{equation}
the following map is $C^p$-regular
\[
\ell: V \to \R_+, \quad (g,y)\mapsto \ell_g(y),
\] 
where $V:=\{(g,y)\in U_{g_0}\times \pl_-\mc{M} \, |\, y\notin \Gamma_-^g\}$. Moreover, for all $\chi\in C_{\comp}^\infty(\pl_-\mc{M})$, there exists a constant $C>0$ (depending only on $g_0, p$ and $\chi$) such that for all $j\leq p$ and $h \in C^\infty(M, \otimes_S^2 T^*M)$
\[
\forall (g,y)\in V, \quad |\chi d_y^j \ell_g(y)| \leq Ce^{C\ell_{g}(y)}, \qquad |\chi \partial_g^j \ell_g(y)(\otimes^j h)| \leq C e^{C\ell_g(y)} \|h\|^j_{C^{j+1}}.
\] 
\end{lemma}
\begin{proof}
We shall use the implicit function theorem. Let $\rho$ be the boundary defining function of $\mc{M}$ defined in \S\ref{sssection:extension}. As explained in this paragraph, for $g$ close to $g_0$, we can consider a vector field $X$ on $\mc{N}$ such that $X$ vanishes (to first order) on $\{\rho=-\rho_0\}$. For the sake of simplicity, we still denote by $(\varphi_t^g)_{t \in \R}$ the extended flow on $\mc{N}$, and by $X_g := X$ its generator.

We consider the $C^{p}$-regular map  
\[ 
F: U_{g_0} \times \R_+ \to \R, \quad (g,y,t)\mapsto \rho(\varphi^g_{t}(y)).
\]
The function $\ell_g(y)$ satisfies the implicit equation $F(g,y,\ell_g(y))=0$. Let us take a point $(g_0, y_0) \in V$ and differentiate 
for $(g,y)$ near $(g_0, y_0)$
\[
\pl_t F(g,y,t)= (X_{g}\rho)(\varphi^g_{t}(y)).
\]
Notice that this is non-zero if $y\in \pl_-\mc{M}$ and $\varphi_{t}(y)\in \pl_+\mc{M}$ by strict convexity of $\pl \M$. Thus the implicit function theorem guarantees that there is a neighborhood $U_{g_0}'\subset U_{g_0}$ of $g_0$ 
and $B_{y_0}(\eps')\subset \pl_-\mc{M}$ of $y_0$ such that $(g,y)\mapsto \ell_g(y)$ is a well-defined $C^{p}(U_{g_0}'\times B_{y_0}(\eps'))$ 
function  and 
\[
d_y\ell_g(y)=- \frac{d\rho \big(\varphi^g_{\ell_g(y)}(y)\big) \circ \big(d \varphi^g_{\ell_g(y)}\big)(y)}{(X_g\rho)(\varphi^g_{\ell_g(y)}(y))}.
\]
Notice in particular that this implies that $V$ is an open set. By the Grönwall lemma, there is a constant $C>0$ uniform in $g\in U_{g_0}$ so that for each $(g,y)\in V$ and all $t>0$, where $\|\cdot\|$ denotes an arbitrary fixed metric on $\mc{N}$:
\begin{equation}
\label{equation:gron}
\|d_y\varphi^g_{t}(y)\|\leq Ce^{Ct}.
\end{equation}
The constant $C > 0$ provided by the Grönwall lemma is uniform in the metric $g$ as long as it is $C^3$-close to $g_0$. More generally, \eqref{equation:gron} holds for the $j$-th derivative $d_y^j\varphi^g_t$ with a constant $C > 0$ uniform for $g$ which is $C^{j+2}$-close to $g_0$. On the other hand, we know that $X_g\rho \neq 0$ on $\pl \mc{M}\setminus \pl_0\mc{M}$. So we obtain a constant $C > 0$ such that 
\[
\forall (g,y)\in V, \quad |\chi(y)d_y\ell_g(y)|\leq Ce^{C\ell_g(y)}.
\]
Next, we compute the derivative with respect to $g$, for some $h \in C^\infty(M, \otimes_S^2 T^*M)$: 
\[
	(\pl_g \ell_{g}.h) (y)=-\frac{d\rho\big(\varphi^g_{\ell_g(y)}(y)\big) \circ \big(\pl_g\varphi^g_{\ell_g(y)}.h\big)(y)}{(X_g\rho)(\varphi^g_{\ell_g(y)}(y))}.
\] 
Again, by the Grönwall lemma, we obtain a constant $C>0$ such that for all $t>0$, $(g,y)\in V$:
\begin{equation}
\label{equation:gron2}
\|(\pl_g\varphi^g_{t}.h)(y)\|\leq Ce^{Ct}\|h\|_{C^2},
\end{equation}
which provides the desired estimate for the $C^2$-norm. (The $C^2$-norm of $h$ appears as the vector field $X_g$ involves $1$-derivative of $g$, so that $X_{g+sh}$ is $C^1$ for all $s\in \R$ small). The constant $C > 0$ is uniform for $g$ that is $C^3$-close to $g_0$. More generally, the bound $|\pl^j_g\varphi^g_{t}(\otimes^j h)(y)|\leq Ce^{Ct}\|h\|_{C^{j+1}}^j$ holds with a constant $C > 0$ depending on the $C^{j+2}$-norm of $g$. The case of higher order derivatives works exactly the same way by differentiating as many times as needed the implicit equation defining $\ell_g(y)$ with respect to $(g,y)$, and using that the derivatives of the flow satisfy the bounds $\|D^j\varphi_t^g(y)\|\leq Ce^{Ct}$ (where $D^j = \partial_g^j$ or $d_y^j$), for some uniform $C>0$ with respect to $t > 0$, $y$ and $g\in U_{g_0}$.
\end{proof}

\subsection{Hyperbolic trapped set}

\label{ssection:hyperbolic-trapped-set}

\subsubsection{Axiom A property}

We say that the trapped set is \emph{hyperbolic} if there is a continuous flow-invariant splitting of $T(SM)$ restricted to $K^g$ into three subbundles
\[
\forall y\in K^g, \quad T_{y}\mc{M}=\R X_g(y)\oplus E^g_s(y)\oplus E^g_u(y)
\]
and $C,\nu>0$ such that for all $y \in K^g$ and $t \geq 0$
\begin{equation}
\label{equation:an}
\begin{split}
& v \in E^g_s(y) \implies  \|d\varphi^g_t(y) v\| \leq Ce^{-\nu t} \|v\| ,\\
&v \in E^g_u(y) \implies  \|d\varphi^g_{-t}(y)v\| \leq Ce^{-\nu t} \|v\|.
\end{split}
\end{equation}
There is a continuous extension of the bundle $E_{s}^{g}$ (resp. $E_u^g$) to a bundle $E_{-}^{g}$ (resp. $E_+^g$) over the set $\Gamma_-^{g}$ (resp. $\Gamma_+^g$), on which \eqref{equation:an} is still satisfied, see \cite[Lemma 2.10]{Dyatlov-Guillarmou-16}. For $y \in K^g$, these bundles coincide with $E_{s}^{g}, E_u^g$, namely $E_{s}^{g}(y)=E_-^{g}(y)$, $E_u^{g}(y)=E_+^{g}(y)$. We define $C^k_{\rm hyp}(M, \otimes^2_S T^*M_+)$ to be the set of $C^k$ Riemannian metrics on $M$ with strictly convex boundary and hyperbolic trapped set. For such metrics, the geodesic flow is a typical example of what is known as an \emph{Axiom A flow}. Since these metrics could have conjugate points, this set is larger than the set of metrics of Anosov type.

If $g_0$ is some fixed metric on $M$ and $M_e$ denotes the extension defined in \S\ref{sssection:extension} with $\rho$ a boundary defining function of $\mc{M}$, we can always choose $\rho_0 > 0$ small enough so that for all $|t|\leq \rho_0$ the level set $\{\rho=t\}$ is strictly convex with respect to the extension $g_{0e}$ of $g_0$ to $M_e$. This also holds for any metric $g$ close to $g_0$ in the $C^2$-topology. Recall that we denote by $g_e$ the extension of $g$ from $M$ to $M_e$. 

Observe that if $y\in \pl_\pm \mc{M}$,  $\cup_{\pm t>0}\varphi_t^{g_e}(y)\subset \mc{N}\setminus \mc{M}$. The trapped sets of $(M,g)$ and $(M_e,g_e)$ then coincide and $\Gamma_\pm^{g}=\Gamma_\pm^{g_e}\cap \mc{M}$. Moreover, if $(M,g)$ has no conjugate points, then by taking $\rho_0>0$ small enough $(M_e,g_e)$ does not have conjugate points either (see \cite[Lemma 2.3]{Guillarmou-17-2}).
 
Define the set of points that are trapped for time less than $t\geq 0$ as:
\[
\mc{T}^g(t):=\{ y\in \mc{M}\,| \, \forall s\in (0,t),\, \varphi_s^g(y)\in \mc{M}^\circ \}=\tau_g^{-1}(t,\infty).
\]
It is proved in \cite[Proposition 2.4]{Guillarmou-17-2} that there exist $C_g,Q_g>0$ (depending on the metric $g$), such that for all $t\geq 0$
\begin{equation}\label{eq:volume}
\mu(\mc{T}^g(t)) \leq C_ge^{-Q_gt}.
\end{equation}
(Here $\mu$ is the Liouville measure for the fixed $g_0$.) In particular, $\mu(\Gamma_\pm^g)=0$. The quantity $Q_g$ is called the \emph{escape rate} and is given by $-Q_g=P_g(-J_u^g) < 0$, the topological pressure of minus the unstable Jacobian $J_u^g(y):=\pl_t (\det d\varphi_t^g(y)|_{E_u(y)})|_{t=0}$ of the flow $(\varphi_t^g)_{t \in \R}$. 
Recall that the topological pressure of a H\"older potential $V\in C^\beta(S^{g}M)$ (for  some $\beta>0$) with respect to $g$ can be defined as follows: 
\[
P_g(V):=\lim_{T\to \infty} \frac{1}{T}\log \sum_{\gamma\in \mc{P},T_\gamma\in [T,T+1]} \exp \left({\int_{\gamma} V}\right),
\]
where $\mc{P}$ denotes the set of periodic orbits of the geodesic flow $(\varphi_t^g)_{t \in \R}$, and $T_\gamma$ denotes the period of $\gamma \in \mc{P}$.

The following formula for $f\in L^1(\mc{M})$ is known as \emph{Santaló's formula} (see \cite[Section 2.5]{Guillarmou-17-2}):
\begin{equation}
\label{equation:santalo}
\int_{\mc{M}} f(y) \dd \mu(y) = \int_{\partial_-\mc{M}} \int_0^{+\infty} f(\varphi_t^g(y)) \dd t\, \dd \mu_\partial(y).
\end{equation}
It implies, together with \eqref{eq:volume}, that there is $C_g>0$ so that for all $t>0$:
\begin{equation}\label{escaperate}
\mu_\pl \left(\ell_g^{-1}(t,\infty)\right) \leq C_g e^{-Q_g t}.
\end{equation}

Using Cavalieri's principle, estimates \eqref{eq:volume} and \eqref{escaperate}, it is straightforward to derive the following bounds:
\begin{equation}\label{cavalieri} 
\begin{gathered}
\forall p\in [1,\infty), \quad \tau_g\in L^p(\mc{M}),\,\, \ell_g\in L^p(\pl_-\mc{M}), \\ 
\forall \la\in (0,Q_g),\quad e^{\la \tau_g}\in L^1(\mc{M}), \,\,    e^{\la \ell_g}\in L^1(\pl_-\mc{M}).
\end{gathered}
\end{equation}
Here note that  $\ell_g$ is bounded near $\pl_0\mc{M}$ so that this region is trivial to deal with. 

\subsubsection{Robinson structural stability} In this paragraph, we recall some results about the stability of flows with hyperbolic trapped set, due to Robinson \cite[Theorem C]{Robinson-80}. 
First, the stable and unstable manifolds of a point $y\in K^g$ are defined by 
\[
\begin{split} W_s(y):=& \{ y'\in \mc{M}\, |\, \lim_{t\to +\infty}d(\varphi_t^g(y'),\varphi_t^g(y))\to 0\},  \\
W_u(y):=& \{ y'\in \mc{M}\, |\, \lim_{t\to -\infty}d(\varphi_t^g(y'),\varphi_t^g(y))\to 0\}.
\end{split}
\] 
They are smooth injectively immersed submanifolds. We also set 
$W_u(K^g):=\cup_{y\in K^g}W_u(y)$, $W_s(K^g):=\cup_{y\in K^g}W_s(y)$. It is proved in \cite[Lemma 2.2]{Guillarmou-17-2} that 
\begin{equation}\label{WsGamma_-}
W_s(K^g)=\Gamma_-^g , \qquad W_u(K^g)=\Gamma_+^g.
\end{equation}
The tangent spaces to $W_s(y)$ and $W_u(y)$ are respectively $E_s(y)$ and $E_u(y)$.
The flow satisfies the following \emph{transversality property} for the stable and unstable manifolds: 
for each $y,y'\in K^g$ and $z\in W_s(y)\cap W_u(y') \subset K^g$, we have
\[
T_z(\mc{M}) = T_{z}(W_s(y))\oplus T_{z}(W_u(y')) \oplus \R X_g(z).
\]
Indeed, such $z$ must belong to $K^g$, and the identity of tangent space 
can be rewritten as $E_s(z)\oplus E_u(z) \oplus  \R X_g(z)=T_z(\mc{M})$, which holds since $K^g$ is assumed hyperbolic. 
For a Riemannian manifold with strictly convex boundary and hyperbolic trapped set, the geodesic flow $(\varphi_t^g)_{t \in \R}$ on $\mc{M}$ satisfies that: 
\begin{itemize}
\item the non-wandering set $\Omega\subset K^g$ is hyperbolic, 
\item the stable and unstable manifolds have the transversality property,
\item the boundary is strictly convex with respect to the vector field $X_g$.
\end{itemize}
The following holds:

\begin{proposition}[Robinson \cite{Robinson-80}]
\label{theorem:stability}
Let $(M,g_0)$ be a smooth Riemannian manifold with strictly convex boundary and hyperbolic trapped set $K^{g_0}\subset \mc{M}:=SM$. 
Then, there exists $\eps_0 > 0$ such that for each smooth vector field $X$ on $\mc{M}$ so that  $\|X-X_{g_0}\|_{C^2(\mc{M})}\leq \eps_0$, there is a homeomorphism $h:\mc{M}\to \mc{M}$ and $a\in C^0(U)$, where $U=\big\{(y,t)\in \mc{M}\times \R \,|\, t \in \big[-\tau_{g_0}(-h(y)), \tau_{g_0}(h(y))\big] \big\}$, such that the following holds: for all $y \in \mc{M}$, $t \mapsto a(y, t)$ is strictly increasing in $t$, and satisfies 
 \[
 \varphi_t^{X_{g_0}}(h(y))=h(\varphi_{a(y,t)}^X(y)),
 \]
for all $(y, t) \in \mc{M} \times \mathbb{R}$ such that $\varphi_{a(y,t)}^X(y) \in \mc{M}$. Moreover, for each $\delta > 0$ there exists $\varepsilon > 0$ small enough such that if $\|X-X_{g_0}\|_{C^2(\mc{M})}\leq \eps$, then $d(h(y),y)\leq \delta$ for $y \in \M$, where $d$ denotes a Riemannian distance on $\mc{M}$, that is, $\|h-\mathrm{id}_{\mc{M}}\|_{C^0} \leq\delta$.
\end{proposition}
\begin{proof}
	This is a direct consequence of \cite[Theorems A and C]{Robinson-80}. We note that Robinson's ``quadratic external boundary conditions'' are equivalent to our strict convexity of the boundary, and that the \emph{chain-recurrent set} (see \cite{Robinson-80} for the definition) is contained in the trapped set, which by assumption has a hyperbolic structure with transversal stable and unstable manifolds. Finally, the last statement about the continuity of $h$ is stated in \cite[Theorem A]{Robinson-80}.
\end{proof}

As a consequence we see that for $g$ close enough to $g_0$ in $C^3$ norm, applying this Proposition with $X=X_g$, we get  $K^g=h^{-1}(K^{g_0})$, $h^{-1}(\Gamma_\pm^{g_0})=\Gamma_\pm^g$, and the trapped set varies continuously with respect to the metric.

\subsubsection{Symplectic lift to the cotangent bundle}

\label{sssection:geom-cons}

Recall that we introduced the vector field $X$ on $\mc{N}$ in \S\ref{sssection:extension}. In Section \S\ref{section:scattering}, it will be convenient to work on the cotangent bundle $T^*\mc{N}$ of the extended manifold $\mc{N}$. Denote by $\X$ the symplectic lift of the vector field $X$ to $T^*\mc{N}$. It generates the flow
\begin{equation}\label{eq:symplectic-lift}
\varphi^{\X}_t(y,\xi) = (\varphi^X_t(y), (d\varphi^X_t(y))^{-\top} \xi),
\end{equation}
where ${}^{-\top}$ stands for the inverse transpose. Note that this flow is linear in the second variable and thus induces a flow on the spherical bundle $S^*\mc{N} := (T^*\mc{N}\setminus \{0\})/\R_+$. Let $\pi : S^*\mc{N} \to \mc{N}$ and $\kappa:T^*\mc{N}\to S^*\mc{N}$ be the natural projections, and still write $\pi$ for the projection $T^*\mc{N} \to \mc{N}$. The dual subbundles $(E_{\pm,0}^X)^* \subset T^*\mc{N}$ are defined as the following symplectic orthogonals:
\[
(E_0^X)^*\left(E^X_+ \oplus E^X_-\right) = (E_+^X)^*\left(E^X_+ \oplus E^X_0\right) = (E_-^X)^*\left(E^X_- \oplus E^X_0\right) = \left\{0\right\}.
\]
With some abuse of notation, the spaces $(E_{\pm,0}^X)^*$ will be identified with the projections $\kappa((E_{\pm,0}^X)^*) \subset S^* \mc{N}$. Eventually, we record the following definition to be found useful later:
\begin{equation}
\label{equation:sigma}
\Sigma_\pm := \bigcup_{\|X-X_0\|_{C^2} \leq \delta,\,  \pm t \geq 0} \varphi^X_t(\mc{M}),
\end{equation}
where $\delta>0$ is small enough. Finally, we note that the tails $\Gamma_\pm^{X}$ and the bundles $(E_{\pm,0}^X)^*$ admit an extension to the set $\left\{\rho > -\rho_0\right\}$.

\subsection{Resolvent and X-ray transform}

Since we will work with Sobolev spaces on the manifolds $\mc{M}$ and $\pl_\pm \mc{M}$, let us clarify what this means as these are manifolds with boundary or open manifolds. First, since $\mc{M}$ is a smooth manifold 
with boundary the spaces $H^s(\mc{M})$ are defined intrinsically for $s\geq 0$ (as the restriction of $H^s$-functions defined on $\mc{N}$ for instance). Set $H_0^s(\mc{M})=\overline{C_{\comp}^\infty(\mc{M}^\circ)}$ where the closure is for the $H^s$ norm, and write $H^{-s}(\mc{M}):=(H_0^s(\mc{M}))^*$ for $s>0$, where the upper star denotes the continuous dual. For $\pl_\pm \mc{M}$, write 
$H^s(\pl\mc{M}):=H^s(\overline{\pl_\pm\mc{M}})$ where $\overline{\pl_\pm\mc{M}}:=\pl_\pm \mc{M}\cup \pl_0\mc{M}$ is a smooth manifold with boundary, and $H^{-s}(\pl_\pm \mc{M})=(H_0^s(\pl_\pm\mc{M}))^*$.

Define the resolvent of $X_g$ to be the the family of operators, for $\Re(z) \geq 0$:
\begin{equation}\label{defofRg}
R_g(z): C_{\comp}^\infty(\mc{M}^\circ \setminus \Gamma_-^g)\to C^\infty(\mc{M}), \quad R_g(z)f(y):=-\int_{0}^{\tau_g(y)}e^{-zt}f(\varphi^g_t(y))dt.
\end{equation}
For $z=0$, simply write $R_g := R_g(0)$.
It solves $X_gR_g=\mathbbm{1}$ on $C_{\comp}^\infty(\mc{M}^\circ \setminus \Gamma_-^g)$ with boundary condition $(R_gf)|_{\pl_+\mc{M}}=0$.

Assuming that $(M,g)$ has strictly convex boundary and hyperbolic trapped set, we have by \cite[Propositions 4.2 and 4.4]{Guillarmou-17-2} the following boundedness properties:
\begin{align}
&  \forall p\in [1,\infty), &  R_g: L^\infty(\mc{M})\to L^p(\mc{M}),\label{Rg1}\\
& \forall \alpha\in (0,1), \exists s>0, &   R_g: C_{\comp}^\alpha(\mc{M}^\circ)\to H^s(\mc{M})\label{Rg2},\\
& \forall s > 0, & R_g: H^s(\mc{M})\to H^{-s}(\mc{M})\label{Rg3},
\end{align}
where  $C^\alpha(\mc{M})$ is the H\"older space of order $\alpha$.
Note that if $\eps>0$ is chosen small enough, $U:=\cup_{t\in (-\eps,\eps)} \varphi_t^g(\pl_-\mc{M})$ is a neighborhood of $\pl_- \mc{M}$ in $\mc{M}_e$ which is diffeomorphic to $(-\eps,\eps)\times \pl_-\mc{M}$ by $(t,y)\mapsto \varphi_t^g(y)$, and $\pl_t(\tau_g\circ \varphi_t^g)=-1$ in $U$. Using \eqref{Rg1}, Santaló's formula \eqref{equation:santalo}, and the fact that $\ell_g$ is smooth near $\partial_0 \mc{M}$ in $\pl_- \mc{M}\cup \pl_0\mc{M}$ (see \cite[Lemma 4.1.1]{Sharafutdinov-94}), we consequently obtain
\begin{equation}\label{reg_ell_g} 
\ell_g=-(R_g{\bf 1}_{\mc{M}})|_{\pl_-\mc{M}} \in L^p(\pl_-\mc{M},\mu_{\pl})
\end{equation}
for all $1 \leq p<\infty$. The \emph{X-ray transform} is defined as the operator 
\[
I^g: C_{\comp}^\infty(\mc{M}\setminus \Gamma_-^g)\to C^\infty_{\comp}(\pl_-\mc{M}\setminus \Gamma_-^g),\quad I^gf:=-(R_gf)|_{\pl_-\mc{M}}.
\]
and, by \cite[Lemma 5.1]{Guillarmou-17-2}, it extends as a bounded map for all $p>2$ 
\begin{equation}\label{Igbounded} 
I^g: L^p(\mc{M})\to L^2(\pl_-\mc{M},\mu_{\pl}).
\end{equation}
We now show the following boundedness property:

\begin{lemma}
\label{boundednessIg}
Let $(M, g)$ be a compact Riemannian manifold with strictly convex boundary and hyperbolic trapped set. Then, there exists $s>0$ such that the operator $I^g$ is bounded as a map 
\[
I^g: C^2(\mc{M})\to H^s(\pl_-\mc{M}).
\]
\end{lemma}
\begin{proof}
First of all, if $\chi\in C^\infty(\overline{\pl_-\mc{M}})$ is supported close to $\pl_0\mc{M}$, one can check that $\chi I^gf\in 
C^2(\overline{\pl_-\mc{M}})$ for $f\in C^2(\mc{M})$, see \cite[Lemma 4.1.1]{Sharafutdinov-94}. It thus remains to analyze $\chi I^gf$ when $\chi\in C_{\comp}^\infty(\pl_-\mc{M})$.
Let $\gamma > 0$ be a large enough constant (it will be determined later), $\eps \in \left(0,Q_g/(2\gamma)\right)$, and let $\Delta_h$ be the Riemannian Laplacian associated to an arbitrarily chosen smooth Riemannian metric $h$ on $\overline{\pl_-\mc{M}}$, with Dirichlet condition at $\pl_0\mc{M}$. It is self-adjoint on $H_0^1(\overline{\pl_-\mc{M}})\cap H^2(\overline{\pl_-\mc{M}})$ with respect to the Riemannian volume measure ${\rm dv}_h$. Note that ${\rm dv}_h$ is smoothly equivalent to $\mu_{\pl}$ on each compact set of $\pl_-\mc{M}$ as $\mu_{\pl}$ vanishes to first order on the boundary $\partial_0 \mc{M}$.

For $f\in C^2(\mc{M})$, consider the holomorphic map
\[
\left\{-\eps \leq \Re(z) \leq 1-\eps\right\} \ni z \mapsto u(z) :=(1+\Delta_{h})^{z+\eps}(e^{-z\gamma \ell_g}\chi I^gf)\in \mc{D}'(\pl_-\mc{M}).
\]
We are going to apply the Hadamard three line theorem (see \cite[Theorem 12.8]{Rudin-87}) to the holomorphic family of distributions $u(z)$. From \eqref{Igbounded}, we have $I^gf\in L^2(\pl_-\mc{M},\mu_{\pl})$, but we can also write the pointwise bound 
\begin{equation}\label{eq:pointwise-I^g}
	\forall y\in \pl_-\mc{M}\setminus \Gamma_-,\quad  |I^gf(y)| \leq \|f\|_{L^\infty}\ell_g(y).
\end{equation}
From \eqref{cavalieri}, we get using that $\eps <Q_g/(2\gamma)$, 
\[
\chi e^{\eps \gamma \ell_g}I^gf \in L^2(\pl_-\mc{M},{\rm dv}_h).
\]
Therefore on the line $\{\Re(z)=-\eps\}$ with $0<\eps <Q_g/(2\gamma)$, there exists a constant $C > 0$ independent of $z$, $f$ (but depending on $\chi$) such that
\begin{equation}\label{boundonRes=-eps} 
\|u(z)\|_{L^2}\leq \|(1+\Delta_h)^{i\Im(z)}\|_{L^2\to L^2}\|\chi e^{\eps \gamma \ell_g}I^g(f)\|_{L^2} \leq C\|f\|_{L^\infty}
\end{equation}
where $L^2=L^2(\pl_-\mc{M},{\rm dv}_h)$. Note that we used the spectral theorem for $\Delta_h$ in order to bound 
$\|(1+\Delta_h)^{i\Im(z)}\|_{L^2\to L^2}\leq 1$.

Now, using that $I^gf(y)=\int_0^{\ell_g(y)}f(\varphi_t^g(y))dt$,
we obtain, using Lemma \ref{lemma:bound-dell}, \eqref{equation:gron}, and \eqref{eq:pointwise-I^g}, the pointwise bound on $\pl_-\mc{M}\setminus \Gamma_-^g$
\[ 
|\Delta_h(e^{-z\gamma \ell_g}I^gf)(y)|\leq C(1+|z|^2)\|f\|_{C^2(\mc{M})}e^{(C_0-\gamma \Re(z))\ell_g(y)}
\]
for some uniform constants $C,C_0>0$ (depending only on the metric $g$). We therefore see that for $\Re(z)=1-\eps$, the function $\Delta_h(e^{-\gamma z\ell_g}\chi I^g(f))$ can be extended from $\pl_-\mc{M}\setminus \Gamma_-$ continuously to $\pl_-\mc{M}$ by setting it to be $0$ on $\Gamma_-$ as long as $\gamma(1-\eps)>C_0$. Here, we see that, in order to achieve this, we can choose $\gamma > 2022 C_0$ at the very beginning (the constant $C_0$ only depends on the metric $g$). 

\begin{claim}
\label{claim:above}
The continuous extension by $0$ of $\Delta_h(e^{-z\gamma \ell_g}\chi I^gf)$ on $\Gamma_-^g$ matches with the distributional derivative $\Delta_h(e^{-z\gamma \ell_g}\chi I^gf) \in \mc{D}'(\partial_-\mc{M})$.
\end{claim}

The proof of this claim is postponed below. Then $\Delta_h(e^{-z\gamma \ell_g}\chi I^gf) \in L^2(\partial_-\mc{M})$ and 
on the line $\{\Re(z)=1-\eps\}$ we have 
\begin{equation}
\label{boundonRes=1-eps}
\|u(z)\|_{L^2}\leq \|(1+\Delta_h)^{i\Im(z)}\|_{L^2\to L^2}\|(1+\Delta_h)(e^{-z\gamma \ell_g}\chi I^gf)\|_{L^2}\leq C(1+|z|^2)\|f\|_{C^2}.
\end{equation}
We can then use Hadamard 3-lines interpolation theorem applied to the holomorphic function
\[
\left\{-\eps \leq \Re(z)\leq1-\eps\right\} \ni z \mapsto v(z):=\int_{\pl_-\mc{M}} (1+z)^{-2}u(z) \psi \dd\mathrm{v}_h \in \C,
\]
where $\psi\in C_{\comp}^\infty(\pl_-\mc{M})$ is arbitrary. Note that this is well-defined and holomorphic in the strip $\Re(z)\in [-\eps,1-\eps]$ since we have the bound:
\[
|v(z)|\leq \frac{1}{(1 - \varepsilon)^2}\|\psi\|_{H^{2(\Re(z) + \varepsilon)}}\|e^{\eps \gamma \ell_g}\chi I_gf\|_{L^2}\leq C\|\psi\|_{H^2}\|f\|_{C^2}.
\]

From \eqref{boundonRes=-eps} and \eqref{boundonRes=1-eps}, we deduce the existence of a constant $C>0$, independent of $\psi$, such that for all $z$ with $\Re(z)\in [-\eps,1-\eps]$, one has 
\[
|v(z)|\leq C\|\psi\|_{L^2}.
\]
This shows that $u(z)\in L^2(\pl_-\mc{M})$ for all such $z$ with the bound $|u(z)| \leq C$.
In particular taking $z=0$, we obtain that $(1+\Delta_h)^\eps (\chi I^gf) \in L^2$, thus showing the claimed result. \\

It thus remains to prove Claim \ref{claim:above} above. Denote by $F$ the continuous extension of $\Delta_h(e^{-z\gamma \ell_g}\chi I^g(f))$ by $0$ on $\Gamma^g_-$. We need to show that for each $\psi\in C_{\comp}^\infty(\pl_-\mc{M})$,
\begin{equation}
\label{equation:eq}
\int_{\pl_-\mc{M}} \chi e^{-z\gamma \ell_g}I^g(f)  \Delta_h\psi ~{\rm dv}_h= \int_{\partial_- \mc{M}} 
F  \psi ~ {\rm dv}_h.
\end{equation}
Take $\theta\in C_{\comp}^\infty([0,2))$, equal to $1$ in $[0,1]$. We write the left hand side as 
\[
\begin{split} 
\lim_{T\to \infty}  \int_{\pl_-\mc{M}} \theta\big(\ell_g/T\big) \chi e^{-z\gamma \ell_g}I^g(f) \Delta_h\psi\,  {\rm dv}_h=& \lim_{T\to \infty}
 \int_{\pl_-\mc{M}} \Delta_h\Big(\theta\big(\ell_g/T\big) \chi e^{-z\gamma \ell_g}I^g(f)\Big) \psi \,{\rm dv}_h \\
=&  \lim_{T\to \infty} A_1(T)+A_2(T),
\end{split}\]
where:
\[
A_1(T):= \int_{\pl_-\mc{M}}  \Delta_h\big(\theta\big(\ell_g/T\big)\big)\chi e^{-z\gamma \ell_g}I^g(f)\psi~ {\rm dv}_h +  2\int_{\pl_-\mc{M}} \nabla \big( \theta\big(\ell_g/T\big)\big) \cdot \nabla(\chi e^{-z\gamma \ell_g}I^g(f)) \psi\, {\rm dv}_h,
\]
\[
A_2(T):= \int_{\pl_-\mc{M}} \theta\big(\ell_g/T\big) \Delta_h\big(\chi e^{-z\gamma \ell_g}I^g(f)\big) \psi \, {\rm dv}_h = \int_{\pl_-\mc{M}}  \theta(\ell_g/T) F \psi ~ {\rm dv}_h.
\]
In order to show \eqref{equation:eq}, it thus suffices to show that $A_1(T)\to 0$ as $T\to \infty$. The derivatives  $d^{j}_y( \theta(\ell_g/T))$ of order $j=1,2$ are supported in $\{\ell_g\in [T,2T]\}$, where we can use the pointwise bound of Lemma \ref{lemma:bound-dell}:
\[
|d^j_y( \theta(\ell_g(y)/T))|\leq Ce^{C_0\ell_g(y)} \leq Ce^{2C_0T},
\]
for some uniform $C,C_0 > 0$. Since all terms in the integrand of $A_1$ are multiplied by the weight $|e^{-\gamma z\ell_g(y)}|\leq e^{-\gamma (1-\eps)T}$, we easily see, using Lemma \ref{lemma:bound-dell} once again, that
\[
A_1(T)=\mc{O}((1+|z|)e^{(3C_0-\gamma(1-\eps))T}).
\]
Taking $\gamma>6C_0$ at the beginning and $\eps<1/2$, one obtains that $A_1(T)\to 0$, and this proves our claim.
\end{proof}

Note that, as a corollary of Lemma \ref{boundednessIg}, we obtain that there is $s>0$ such that
\begin{equation}\label{regSobolevell_g}
\ell_g=I^g({\bf 1}_{\mc{M}}) \in H^{s}(\pl_-\mc{M}). 
\end{equation}

\subsection{Scattering operator}

Working with the scattering operator $\mc{S}_g$ has several advantages rather than working directly with $S_g$. The main reason is that its Schwartz kernel can be expressed in terms of restriction of the Schwartz kernel of the resolvent $R_g$ of the geodesic vector field $X_g$. This is the content of Lemma \ref{SchwartzkernelSg} below. This will be important 
so that we can work in a good functional setting in order to apply the Taylor expansion of the lens data with respect to $g$. 
We denote $R_{g_e}$ the resolvent on $\mc{M}_e$ for the extension $g_e$ (for the definition of $g_e$ recall \S \ref{sssection:extension}), which has all the properties of $R_g$. 

\begin{lemma}
\label{SchwartzkernelSg}
Let $(M,g)$ be a compact Riemannian manifold with strictly convex boundary and hyperbolic trapped set.
Let $\iota_{\pl_\pm}:\pl_\pm \mc{M}\to \mc{M}$ be the inclusion map. The restriction $(\iota_{\pl_-} \times \iota_{\pl_+})^*R_{g_e}$ of the Schwartz kernel of the resolvent on $\pl_-\mc{M}\times \pl_+\mc{M}$ makes sense as a distribution, and the Schwartz kernel of $\mc{S}_g$ is given by:
\[
\mc{S}_g(y,y')=-(\iota_{\pl_-} \times \iota_{\pl_+})^*R_{g_e}(y,y'), \quad (y, y') \in \pl_-\mc{M}\times \pl_+\mc{M}.
\] 
\end{lemma}

\begin{proof} First, we define the operator $\mc{E}_g: C_{\comp}^\infty(\pl_+\mc{M})\to H^{s}(\mc{M})$ for $s>0$ as follows. Let 
$\Omega=\{(x,v)\in \pl \mc{M}\,|\, |g_x(\nu,v)|\leq \delta\}$ for $\delta>0$ small and define $\Omega_e=\mc{M}_e\cap \cup_{t\in \R}\varphi_t^{g_e}(\Omega)$ 
the flowout of $\Omega$ by $\varphi^{g_e}_t$, and let  $\psi\in C^\infty(\mc{M}_e,\R_+)$ so that $\psi|_{\Omega_e\cup \pl_-\mc{M}}=0$,  $\psi$ supported in a small neighborhood of $\pl_+\mc{M}\setminus \Omega$ and $X_{g_e}\psi=0$  in $\mc{M}_e\setminus \mc{M}$  and near $\pl_+\mc{M}$. Then set for $\omega\in C_{\comp}^\infty(\pl_+\mc{M})$
\[\mc{E}_g\omega:=\tilde{\omega} \psi- R_{g_e}X_{g_e}(\tilde{\omega}\psi)\in H^{s}(\mc{M}_e)\cap L^p(\mc{M}_e)\cap C^\infty(\mc{M}_e\setminus (\Gamma_-\cup \Gamma_+))\]
for some $s>0$ and all $p<\infty$ using \eqref{Rg1} and \eqref{Rg2}, where $\tilde{\omega}$ is defined on ${\rm supp}(\psi)$ by extending $\omega$ from $\pl_+\mc{M}$ to be constant on the flow lines of $X_{g_e}$. This can be done by using the diffeomorphism
 \[ \Psi_+:\{ (t,y)\in (-\delta/2,\infty)\times (\pl_+\mc{M}\setminus \Omega) \,| \, t\leq \tau_{g_e}(y)\} \ni (t,y) \mapsto \varphi_{t}^{g_e}(y)\in \mc{M}_e\] 
and using that the flow $\varphi_t^{g_e}$ is the translation in $t$ in these coordinates. 
One clearly has that $\mc{E}_g\omega$ is smooth near $\pl_+\mc{M}$ and 
\[ X_{g_e}\mc{E}_g\omega=0, \quad (\mc{E}_g\omega)|_{\pl_+\mc{M}}=\psi|_{\pl_+\mc{M}}\omega.\]
In particular, we see that outside $\Gamma_-$ we have 
\begin{equation}\label{restEgomega} 
(\mc{E}_g\omega)|_{\pl_-\mc{M}\setminus \Gamma_-}=(\mc{S}_g(\omega\psi|_{\pl_+\mc{M}}))|_{\pl_-\mc{M}\setminus \Gamma_-}.
\end{equation}
On the other hand, using the diffeomorphism 
\[\Psi_-:\{(t, y) \in (-\infty,\delta/2)\times (\pl_-\mc{M}\setminus \Omega) \, |\, t\geq -\tau_{g_e}(-y)\} \ni (t,y) \mapsto \varphi_t^{g_e}(y) \in \M_e\]
mapping to a neighborhood of $\pl_-\mc{M}\setminus \Omega$, we see that $\Psi_-^*\mc{E}_g\omega$ is independent of $t$ and can be viewed as a function in $H^s(\pl_-\mc{M})\cap L^p(\pl_-\mc{M})$, i.e. the restriction $(\mc{E}_g\omega)|_{\pl_-\mc{M}}$ makes sense as an $H^s(\pl_-\mc{M})\cap L^p(\pl_-\mc{M})$ function. (This fact can also be proved using H\"ormander pull-back theorem for distributions using wave-front analysis with the fact that $X$ is transverse to $\pl_-\mc{M}$.) 
Since $\mu_\pl(\Gamma_-^g\cap\pl_-\mc{M})=0$, this implies with \eqref{restEgomega} that $(\mc{E}_g\omega)|_{\pl_-\mc{M}}=\mc{S}_g(\omega\psi|_{\pl_+\mc{M}})$. But this is also given by 
$(\mc{E}_g\omega)|_{\pl_-\mc{M}}=-(R_{g_e}X_{g_e}(\tilde{\omega}\psi))|_{\pl_-\mc{M}}$. Since $X_{g_e}R_{g_e}=R_{g_e}X_{g_e}={\rm Id}$ in $C_{\comp}^\infty(\mc{M}_e^\circ)$ (this follows for instance by analytic extension of the identity $R_{g_e}(z)(X_{g_e} - z) = (X_{g_e} - z)R_{g_e}(z) = {\rm Id}$ on $C_c^\infty(M_e^\circ)$ for $\Re(z)\gg 1$), one has $(X_{g_e}R_{g_e})(y,y')=0$ and $(X'_{g_e}R_{g_e})(y,y')$ in the distribution sense for $y$ close to $\pl_-\mc{M}\setminus \Omega$ and 
$y'$ close to $\pl_+\mc{M}\setminus \Omega$, where $X_{g_e}$ (resp. $X'_{g_e}$) 
denotes the action of $X_{g_e}$ on the left (resp. right) variable of $\mc{M}_e\times \mc{M}_e$. This implies as above that the restriction $(\iota_{\pl_-}\times \iota_{\pl_+})^*R_{g_e}$ makes sense and we can apply Green's formula in the right variable:  if $\omega'\in C_{\comp}^\infty(\pl_-\mc{M})$ \[\begin{split}
-\cjg \iota_{\pl_-}^*(R_{g_e}X_{g_e}(\tilde{\omega}\psi)),\omega'\cjd
=& -\int_{\pl_-\mc{M}} \int_{\mc{M}}R_{g_e}(y,y')X_{g_e}(\tilde{\omega}\psi)(y') \omega'(y) d\mu(y')d\mu_{\pl}(y) \\
=&  -\int_{\pl_-\mc{M}}\int_{\pl_+\mc{M}}R_{g_e}(y,y')(\psi\omega)(y')\omega'(y) i_{X_{g_e}}d\mu (y') d\mu_{\pl}(y)
\end{split}\]
where we used $X_{g_e}(\tilde{\omega}\psi) = 0$ on $\M_e \setminus \M$, and that $(X_{g_e}R_{g_e})(y,y')=0$ for the interior term from Green's formula. This means, using that $i_{X_{g_e}}d\mu=d\mu_{\pl}$ at $\pl_+\mc{M}$, that
\[ -\cjg \iota_{\pl_-}^*(R_{g_e}X_{g_e}(\tilde{\omega}\psi)),\omega'\cjd= -\cjg (\iota_{\pl_-}\times \iota_{\pl_+})^*R_{g_e}, \omega'\otimes \psi|_{\pl_+\mc{M}}\omega\cjd .\] 
This shows that $\mc{S}_g(y,y')\psi(y') = -(\iota_{\pl_-}\times \iota_{\pl_+})^*R_g(y,y')\psi(y')$ as a distribution of $(y,y')\in \pl_-\mc{M}\times \pl_+\mc{M}$. Since $\Omega$ can be chosen with $\delta>0$ arbitrarily small, we obtain the result by choosing $\psi=1$ outside a $\delta/4$ neighborhood of $\Omega\cap \pl_+\mc{M}$ in $\pl_+\mc{M}$.
\end{proof}

We will also need the following regularity bound:

\begin{lemma}\label{lemma:upperboundS_g}
Let $g \in C^\infty(M,\otimes^2_S T^*M_+)$ be a metric with strictly convex boundary and hyperbolic trapped set, $\chi\in C_{\comp}^\infty(\pl_-\mc{M})$, $f\in C^\infty(\pl_+\mc{M})$ and $p \in \N$. Then:

\begin{enumerate}
\item There exists $\beta \gg 0$ large enough such that for all $z \in i\R + \beta$, $\chi e^{- z \ell_{g}}\mc{S}_{g}f$ extends by $0$ on $\Gamma_-^{g}$ with an extension belonging to $W^{p+1,\infty}(\partial_- \mc{M})$ and the weak distributional derivative $(1+\Delta_h)^{(p+1)/2}(\chi e^{- z \ell_{g}}\mc{S}_{g}f) \in \mc{D}'(\partial_-\mc{M})$ coincides with the derivative of the $W^{p+1,\infty}(\partial_- \mc{M})$-extension.

\item The map
\[
C^{p+1}(\partial_+\mc{M}) \ni f \mapsto e^{- z \ell_{g}} \mc{S}_{g}f \in W^{p+1,\infty}(\partial_-\mc{M})
\]
is bounded, and there exists a uniform constant $C > 0$ (independent of $z$) such that:
\begin{equation}
\label{boundweightedSgf}
\|(1+z)^{-(p+1)}\chi e^{- z \ell_{g}} \mc{S}_{g}f\|_{W^{{p+1},\infty}(\partial_- \mc{M})} \leq C \|f\|_{C^{p+1}(\partial_+\mc{M})}.
\end{equation}

\item In particular, by the Sobolev embedding $W^{p+1,\infty}(\partial_-\mc{M}) \hookrightarrow C^p(\partial_-\mc{M})$, the function $\chi e^{- z \ell_{g}}\mc{S}_{g}f$ extends to a $C^p$-function with $C^p$-norm bounded by \eqref{boundweightedSgf}.
\end{enumerate}
\end{lemma}

\begin{proof}
The proof is rather similar to that of Lemma \ref{boundednessIg} so we will be more succinct. First, if $f\in C^{p+1}(\pl_+\mc{M})$ and $\Re(z) > 0$, the function $F_z(y):=e^{- z \ell_{g}(y)}(\mc{S}_{g}f)(y)$ is $C^{p+1}$ outside $\Gamma_-^g$ and can be extended by continuity by $0$ on $\Gamma_-^g$. We compute
its derivative on $\pl_-\mc{M}\setminus \Gamma_-^g$: if $Y$ is a smooth vector field on $\pl_-\mc{M}$, then
\[
YF_z(y)=F_z(y)\Big(-zd_y\ell_g(y)Y +df_{S_g(y)}\left(d\varphi^g_{\ell_g(y)}(y)Y+d_y\ell_g(y)(Y)X_g(S_g(y))\right)\Big).
\]
We can use Lemma \ref{lemma:bound-dell} and the fact that $\|d_y\varphi_t^g\|\leq Ce^{C_0|t|}$ for some uniform $C,C_0>0$ with respect to $t$: this gives on 
$\supp(\chi)$ that
\[ 
|YF_z(y)|\leq C(1+|z|)\|Y\|_{C^0}\|f\|_{C^1} e^{(C_0-\beta)\ell_g(y)},
\] 
for some $C,C_0>0$ uniform in $y$. In particular if $\beta>C_0$ we obtain that $|Y (\chi F_z)(y)|\leq C(1+|z|)\|Y\|_{C^0}$ almost everywhere. Now, we claim 
that this function is also equal to the weak distributional derivative $Y(\chi F_z)\in H^{-1}(\pl_-SM)$. As in the proof of Lemma \ref{boundednessIg}, we need to show that for each $\psi\in C_{\comp}^\infty(\pl_-\mc{M})$
\[  
\int_{\pl_-\mc{M}} \chi e^{-z\gamma \ell_g}\mc{S}_g(f) Y(\psi)~ {\rm dv}_h=\lim_{T\to \infty} \int_{\partial_-\mc{M}}\theta\big(\ell_g/T\big)Y(e^{-z\gamma \ell_g}\chi \mc{S}_g(f)) \psi ~{\rm dv}_h
 \]
where $\theta\in C_{\comp}^\infty([0,2))$ is equal to $1$ in $[0,1]$ and $h$ is a smooth metric on $\pl_-\mc{M}$ 
as in the proof of Lemma \ref{boundednessIg}. Since the proof of the equality is exactly the same as in the proof of Lemma \ref{boundednessIg}, we do not repeat the argument. This shows that $\chi F_z\in W^{1,\infty}(\pl_-\mc{M})$ with bound
\[
\| \chi F_z\|_{W^{1,\infty}(\partial_-\mc{M})}\leq C(1+|z|)\|f\|_{C^1},
\]
for some $C$ uniform with respect to $z$. The bound $\|\chi F_z\|_{C^0(\partial_-\mc{M})} \leq C(1+|z|)\|f\|_{C^1}$ also follows immediately by Sobolev embedding.

For higher order derivatives, it suffices to repeat this argument, noting by Lemma \ref{lemma:bound-dell} that there are $C>0,C_0>0$ such that for $j\leq p+1$ we have 
\[
\|d_y^j\ell_g(y)\|\leq Ce^{C_0 |t|}, \quad \|d_y^j\varphi_t^g\|\leq Ce^{C_0t},
\] 
on $(\pl_-\mc{M}\cap \supp(\chi))\setminus \Gamma_-^g$. This means that taking $\beta>0$ large enough depending on $C_0$, the argument explained above works the same way. This proves the claimed result.
\end{proof}

Given $\chi\in C_{\comp}^\infty(\pl_-\mc{M})$, define the function on $\partial_-\mc{M}$:
\[
\mc{L}_g(z) := \chi e^{-z \ell_g} = (z (R_g(z){\bf 1}_{\mc{M}})|_{\pl_-\mc{M}}+1)\chi .
\]
We will need the following regularity property:

\begin{lemma}
\label{expsell}
Let $(M,g_0)$ be a smooth compact Riemannian manifold with hyperbolic trapped set and let $p \in 2\N$. There exists $\eps > 0$ small enough, $\beta \gg 0$ large enough such that the following holds: setting
\begin{equation}
\label{equation:neighb}
U_{g_0} := \left\{g \in C^{p+2}(M,\otimes^2_S T^*M) ~|~ \|g-g_0\|_{C^{p+2}} < \eps, g|_{T \partial M}=g_0|_{T \partial M}\right\},
\end{equation}
as in Lemma \ref{lemma:bound-dell}, we have that for $\Re(z)=\beta$ the map 
\[
\mc{L} : U_{g_0}\times\{\Re(z)=\beta \} \ni (g,z) \mapsto \mc{L}_g(z) = e^{-z\ell_g}\chi \in L^\infty(\pl_-\mc{M})\subset L^2(\pl_-\mc{M}),
\]
is $C^{p-1}$-regular. Moreover, there exists a uniform constant $C > 0$ such that
for all $j \leq p-1$:
\[
\forall h\in C^{p+2}(M,\otimes_S^2T^*M), \quad  \|\partial_g^j \mc{L}_g(z)(\otimes^j h)\|_{L^2}\leq C(1+|z|)^j\|h\|^j_{C^{p+2}}.
\]
\end{lemma}
\begin{proof}
First of all, note by \cite[Proposition 2.1]{Guillarmou-Mazzucchelli-18} that all metrics in a $C^2$-neighborhood of $g_0$ have hyperbolic trapped set and strictly convex boundary. Hence $\eps > 0$ is chosen so that this holds.
Pick an arbitrary $g_0' \in U_{g_0}$ and let $h\in C^{p+2}(M,\otimes_S^2T^*M)$ such that $g_t:=g_0'+th\in U_{g_0}$ for $t\in (-\delta,1+\delta)$ for some $\delta>0$ small. Consider the map 
\[ F: (-\delta,1+\delta)\times \pl_-\mc{M}\times \{\Re(z)=\beta\} \ni (t,y,z) \mapsto \mc{L}_{g_t}(z)(y)=e^{-z\ell_{g_t}(y)}\chi(y)\]
where by convention $e^{-z\ell_{g_t}(y)}:=0$ when $\ell_{g_t}(y)=\infty$. 
Lemma \ref{lemma:bound-dell} implies that  $F$ is $C^p$ in the open set 
\[ \mc{O}:=\{ (t,y,z) \in (-\delta,1+\delta)\times \pl_-\mc{M}\times \{\Re(z)=\beta\} \,|\, y\notin \Gamma_-^{g_t}\},\]
and one can write $\partial_t^{j_1}\pl_y^{j_2}\pl_z^{j_3} \mc{L}_{g_t}(z)(y)=H(t,y,z,h)(\otimes^{j_1}h)$ where 
$H(t,y,z,h)$ is a continuous function on $(-\delta,1+\delta)\times \pl_-\mc{M}\times C^{p+2}(M,\otimes_S^2T^*M)$ with values in $j_1$-multilinear functions on $C^{p+2}(M,\otimes_S^2T^*M)$ and satisfying:
there is $C > 0$  such that for all $j_1+j_2+j_3\leq p$ and all $(t, y, z) \in \mc{O}$
\begin{equation}\label{boundDiffj_1j_2jj_3}
|\partial_t^{j_1}\pl_y^{j_2}\pl_z^{j_3} \mc{L}_{g_t}(z)(y)| \leq C(1+|z|)^{j_1+j_2} e^{(C-\beta) \ell_{g_t}(y)} \|h\|^{j_1}_{C^{p+2}}.
\end{equation}
First, we observe that $F$ is continuous on $(-\delta,1+\delta)\times \pl_-\mc{M}\times \{\Re(z)=\beta\}$. Indeed, if 
$(t_n,y_n)\to (t,y)$ is a sequence such that $\ell_{g_{t_n}}(y_n)\leq T$ for some $T<\infty$, by 
Proposition \ref{theorem:stability} we deduce that the trajectories $\mc{M}\cap \cup_{s\geq 0}
\varphi_s^{g_{t_n}}(y_n)$ converge  to the trajectory $\mc{M}\cap \cup_{s\geq 0}\varphi_{s}^{g_t}(y)$ as $n\to \infty$, and 
therefore $\ell_{g_t}(y)<\infty$, and so the limit point belongs to $\mc{O}$. On the other hand, if there is no such $T$, this also implies that $\ell_{g_{t_n}}(y_n)\to \infty$, and in turn $F(t_n,y_n,z)\to 0$ as $n \to \infty$, and $(t,y,z)$ belongs to the set
\[S:=\cup_{t\in (-\delta,1+\delta)}( \{t\}\times \Gamma_-^{g_t}\times  \{\Re(z)=\beta\}).\]
Since $\ell_{g_{t_n}}(y_n)\to\infty$ if $(t_n, y_n)$ converge to a point in $S$ as $n \to \infty$, we see from \eqref{boundDiffj_1j_2jj_3} that if $\beta\gg 1$ is large enough, the derivative $H(t,y,z,h)$ of $F$ on $\mc{O}$ converges to $0$  when approaching $S$, and can thus be extended from $\mc{O}$ by $0$ as a continuous function
on $(-\delta,1+\delta)\times \pl_-\mc{M}\times \{\Re(z)=\beta\}\times C^{p+2}(M,\otimes_S^2T^*M)$.
Next, we are going to show that $F$ is a $C^{p-1}$ map, with $\partial_t^{j_1}\pl_y^{j_2}\pl_z^{j_3} F(t,y,z)=H(t,y,z,h)(\otimes^{j_1}h)$ with $H$ the continuous extension by $0$ on $S$ just discussed, and that there exists $C>0$ independent of $h,t,y,z$ such that for all 
$(t,y,z)\in (-\delta,1+\delta)\times \pl_-\mc{M}\times \{\Re(z)=\beta\}$, and all  $j_1+j_2+j_3\leq p-1$ 
\begin{equation}\label{estimates_derivatives} 
|\partial_t^{j_1}\pl_y^{j_2}\pl_z^{j_3} F(t,y,z)|\leq C(1+|z|)^{j_1+j_2} \|h\|_{C^{j_1+1}}^{j_1}.
\end{equation}
This would prove that the Gateaux derivatives of order $p-1$ are continuous thus the function $\mc{L}$ is $C^{p-1}$ and with the desired bounds on the derivatives.

We proceed in a way similar to the proof of Claim \ref{claim:above}.
We will show that, for each fixed $h$, the distributional derivatives 
of $F$ of order $j\leq p$ are bounded and coincide with the continuous extension of $H(t,y,z,h)(\otimes^{j_1}h)$ 
from $\mc{O}$ to $W:=(-\delta,1+\delta)\times \pl_-\mc{M}\times \{\Re(z)=\beta\}$.
First we let $\Delta$ be a Laplacian associated to a fixed smooth product metric $\hat{g}:=dt^2+g_-+ds^2$ on 
$(-\delta,1+\delta)\times \pl_-\mc{M}\times \{\beta+is\,|\, s\in \R\}$.
Let  $\psi\in C_{\comp}^\infty( (-\delta,1+\delta) \times \pl_-\mc{M}\times (\beta+i\R))$ and we want to show that for $2j\leq p$
\[
\int_{W} \chi e^{-z\ell_{g_t}}\Delta^j \psi \, {\rm dv}_{\hat{g}}= \int_{\mc{O}}
(\Delta^j F)  \psi ~{\rm dv}_{\hat{g}}.
\]
Take $\theta\in C_{\comp}^\infty([0,2); [0, 1])$, equal to $1$ in $[0,1]$ and write the left hand side as 
\begin{equation}\label{limittoprove}
\lim_{T\to \infty} \int_{W} \theta\Big(\frac{\ell_{g_t}}{T}\Big) \chi e^{-z \ell_g} \Delta^j \psi\,  {\rm dv}_{\hat{g}}=  \lim_{T\to \infty}A_1(T)+A_2(T),
\end{equation}
where
\[
A_1(T):= \sum_{k=1}^{2j} \int_{W}  P_k \Big(\theta\Big(\frac{\ell_{g_t}}{T}\Big)\Big) Q_{2j-k}(\chi e^{-z \ell_{g_t}})\psi~{\rm dv}_{\hat{g}} ,
\]
with $P_k$ and $Q_k$ some differential operator of order $k\geq 1$ in the variable $(t,y,z)$ and such that $P_k(1)=Q_k(1)=0$ and
\[
A_2(T):= \int_{W} \theta\Big(\frac{\ell_{g_t}}{T}\Big) (\Delta^j F) \psi \, {\rm dv}_{\hat{g}} 
\]
In order to show \eqref{limittoprove}, it suffices to show that $A_1(T)\to 0$ as $T\to \infty$. The derivatives  $D^{k}_{t,y,z}( \theta(\ell_{g_t}/T))$ of order $k\in [1, 2j]$ 
are supported in $\{\ell_{g_t}\in [T,2T]\}$, where we can use the pointwise bound of Lemma \ref{lemma:bound-dell}: there exists $C>0$ such that for all $(t,y,z)$ with $\ell_{g_t}(y)\in [T,2T]$ 
 \[
|D^k_{t,y,z}( \theta(\ell_{g_t}(y)/T))|\leq Ce^{C\ell_{g_t}(y)} \leq Ce^{2CT}.
\]
Since all terms in the integrand of $A_1$ are multiplied by the weight $|e^{-\beta \ell_{g_t}(y)}|\leq e^{-\beta T}$, we see using Lemma \ref{lemma:bound-dell} that
\[
A_1(T)=\mc{O}(e^{(4C-\beta)T}).
\]
Thus if $\beta$ is chosen large enough we obtain that $A_1(T)\to 0$ as $T\to \infty$. We thus deduce that $F\in W_{\rm loc}^{p,\infty}(W)$ and by Sobolev embedding that $F\in C^{p-1,\alpha}(W)$ for all $\alpha<1$. Finally, the bound \eqref{estimates_derivatives} follows from \eqref{boundDiffj_1j_2jj_3} by continuity.
\end{proof}

\section{Symmetric tensors and the normal operator}

\label{section:tensors}

\subsection{Symmetric tensors} In this paragraph, we recall standard facts on symmetric tensors on Riemannian manifolds. We refer to \cite{Heil-Moroianu-Semmelmann-16,Guillarmou-17-1,Gouezel-Lefeuvre-21} for further details. 

\subsubsection{Definitions}

Let $(M,g)$ be a smooth connected Riemannian manifold with boundary. Let $m \in \Z_{\geq 0}$. Let $\otimes^m_S T^*M \to M$ be the vector bundle of symmetric tensors over $M$ (for $m=0$ we just take the trivial line bundle $\R\times M\to M$). We will also write $\otimes^2_ST^*M_+\subset \otimes_S^2T^*M$ 
for the open convex subset consisting of positive definite tensors (Riemannian metrics).
Since $\otimes^m_S T^*M$ is a subbundle of the vector bundle $\otimes^m T^*M \to M$ of $m$-tensors over $M$, it inherits the natural metric $g^{\otimes m}$. Define the pullback operator
\[
\pi_m^* : L^2(M,\otimes^m_S T^*M) \to L^2(\mc{M}), ~~~ \pi_m^*f (x,v) := f_x(v^{\otimes m}),
\]
where $M$ is equipped with the Riemannian volume, $\otimes^m_S T^*M$ with the metric $g^{\otimes m}$ and $\mc{M}$ with the Liouville measure $\mu$. We denote by ${\pi_m}_*$ the adjoint of $\pi_m^*$ with respect to these scalar products and volume forms. 

The symmetric covariant derivative
\[
D_g : C^\infty(M,\otimes^m_S T^*M) \to C^\infty(M,\otimes^{m+1}_S T^*M)
\]
is defined as $D_g := \sigma \circ \nabla^g$, where $\nabla^g$ is the Levi-Civita connection induced by $g$ and $\sigma :\otimes^m T^*M \to \otimes^m_S T^*M$ is the symmetrization operator defined as:
\[
\sigma \left( \eta_1 \otimes ... \otimes \eta_m \right) := \dfrac{1}{m!} \sum_{\pi \in \mathfrak{S}_m} \eta_{\pi(1)} \otimes ... \otimes \eta_{\pi(m)},
\]
where $\eta_1, ..., \eta_m \in T^*M$. The operator $D_g$ is of \emph{gradient type}, namely it has injective principal symbol. Moreover, it is injective when $m$ is odd and has kernel given by $\R g^{\otimes m/2}$ for even $m$. It satisfies the relation
\begin{equation}
\label{equation:relation}
X_g \pi_m^*  = \pi_{m+1}^* D_g,
\end{equation}
where we recall that $X_g$ is the geodesic vector field of $g$. We let $D_g^* : C^\infty(M,\otimes^{m+1}_S T^*M) \to C^\infty(M,\otimes^m_S T^*M)$ be the formal adjoint of $D_g$, which is nothing more than the divergence $D_g^*u=-{\rm Tr}(\nabla^g u)$, and ${\rm Tr}(\cdot)$ is the trace operator.

For $m \geq 1$ and $k \geq 0, \alpha \in (0,1)$, there exists a unique decomposition
\begin{equation}
\label{equation:decomp-tens}
C^{k,\alpha}(M,\otimes^m_S T^*M) = D_g\left(C^{k+1,\alpha}_0(M,\otimes^{m-1}_S T^*M)\right) \oplus^\bot \ker D_g^*|_{C^{k,\alpha}(M,\otimes^m_S T^*M)},
\end{equation}
where $C^{k+1,\alpha}_0(M,\otimes^{m-1}_S T^*M)$ denotes the space of tensors of Hölder-Zygmund regularity $k+1+\alpha$, vanishing on the boundary, and the sum is orthogonal with respect to the $L^2$-scalar product. The decomposition \eqref{equation:decomp-tens} also holds in the scale of Sobolev spaces $H^s(M,\otimes^m_S T^*M)$ for $s \geq 0$. We call \emph{potential tensors} the tensors in $\ran D_g$ and \emph{solenoidal tensors} (or divergence free tensors) the ones in $\ker D_g^*$.

\begin{lemma}\label{lemma:potential-solenoidal}
For $m\geq 1$, there exist bounded projections $\pi_{\ker D_g^*}: L^2(M,\otimes^m_ST^*M)\to L^2(M,\otimes^m_ST^*M)\cap \ker D_g^*$ and  
$\pi_{\ran D_g}: L^2(M,\otimes^m_ST^*M)\to L^2(M,\otimes^m_ST^*M)\cap \ran D_g|_{H_0^1}$, 
which are pseudodifferential operator of order $0$ on $M^\circ$. Moreover for all $f \in L^{2}(M,\otimes^m_S T^*M)$, there is a unique $h\in H_0^1(M, \otimes^{m-1}_ST^*M)$ and $f_s\in \ker D_g^*\cap L^2$ such that
$f = D_gh + f_s$, and it is given by 
$\pi_{\ker D_g^*}f=f_s$ and  $\pi_{\ran D_g}f=D_gh$.
\end{lemma}

\begin{proof}
The Dirichlet Laplacian $D_g^*D_g: H^2(M, \otimes_S^mT^*M)\cap H_0^1(M, \otimes_S^mT^*M)\to L^2(M)$ is an elliptic self-adjoint operator which is invertible since there are no symmetric Killing tensors vanishing at $\pl M$ by \cite{Dairbekov-Sharafutdinov-10}.
When restricted to $C_{\comp}^\infty(M^\circ)$, its inverse $(D_g^*D_g)^{-1}:H^{-1}(M, \otimes_S^mT^*M)\to H^1_0(M, \otimes_S^mT^*M)$ is a pseudo-differential operator of order $-2$ on $M^\circ$  by standard elliptic microlocal analysis. We then set:
\[
\pi_{\ran D_g}: = D_g(D_g^*D_g)^{-1}D_g^*, \qquad  \pi_{\ker D_g^*}=: {\rm Id} - \pi_{\ran D_g}.
\] 
By construction, they satisfy the desired properties.
\end{proof}

\subsubsection{X-ray transform of tensors}

\label{sssection:x-ray}

We now further assume that the metric $g$ is of Anosov type in the sense of Definition \ref{admissiblemetrics}. We introduce the X-ray transform of symmetric $m$-tensors.

\begin{definition}
The X-ray transform on the space of symmetric $m$-tensors is defined by $I_m^g := I^g \circ \pi_m^*$, where $I_m^g : C^\infty(M, \otimes_S^m T^*M) \to L^2(\partial_-\mc{M})$.
\end{definition}

It is clear from \eqref{equation:relation} that the following inclusion holds:
\begin{equation}
\label{equation:inclusion}
D_g\left(C^{k+1,\alpha}_0(M,\otimes^{m-1}_S T^*M)\right) \subset \ker I^g_m.
\end{equation}

\begin{definition}
The X-ray transform $I_m^{g}$ is said to be \emph{solenoidal injective} on $C^{k,\alpha}(M,\otimes_S^mT^*M)$ if \eqref{equation:inclusion} is an equality.
\end{definition}

In other words, $I_m^{g}$ is solenoidal injective if it is injective in restriction to solenoidal tensors, i.e. on the second factor of the decomposition \eqref{equation:decomp-tens}. When $(M,g)$ is of Anosov type, solenoidal injectivity of the X-ray transform has been proved so far in the following cases:
\begin{enumerate}
\item In dimension $n \geq 2$, when $g$ is of Anosov type with non-positive sectional curvature, see \cite{Guillarmou-17-2};
\item On all surfaces of Anosov type, see \cite{Lefeuvre-19-1};
\item In dimension $n \geq 2$, on all real analytic manifold of Anosov type, injectivity of $I_2^g$ is proved in \cite{Bonthonneau-Guillarmou-Jezequel-22}.
\end{enumerate}
We conjecture that the following holds:

\begin{conjecture}[Solenoidal injectivity of the X-ray transform on manifolds of Anosov type]
Let $(M,g)$ be a smooth Riemannian manifold of Anosov type in the sense of Definition \ref{admissiblemetrics}. Then $I_m^g$ is solenoidal injective.
\end{conjecture}

Eventually, we conclude this paragraph by the following variational formula which relates the length map and the X-ray transform on $2$-tensors:

\begin{lemma}
\label{lemma:variation}
Let $(M,g_0)$ be a compact Riemannian manifold with strictly convex boundary and hyperbolic trapped set. Let $(x,v) \in \partial_-\mc{M} \setminus \Gamma_-^{g_0}$. Let $(g_t)_{t \in (-1,1)}$ be a smooth family of metrics on $M$ with $g_t|_{t=0}=g_0$ and write $h :=\pl_t g_t|_{t=0}$. Then $t \mapsto \ell_{g_t}(x,v)$ is $C^2$-regular for small $t$ and
\[
\left. \partial_t \ell_{g_t}(x,v)\right|_{t = 0} = \dfrac{1}{2}I_2^{g_0}h (x,v) + \alpha_{S_{g_0}(x,v)}(\partial_t S_{g_t}(x,v)|_{t=0})
\]
where we recall that $\alpha$ is the Liouville $1$-form.
\end{lemma}
\begin{proof} First, we use the fact that for $t$ small enough, $g_t$ must have hyperbolic trapped set by \cite[Proposition 2.1]{Guillarmou-Mazzucchelli-18}. 
Let $c_0(s)$ be a geodesic for $g_0$ parametrised by arc-length, and $s \mapsto c_t(s)$ for $t \in (-1,1)$ be a $C^1$ family of curves for $s \in [0, \ell_{g_0}(c_0)]$. Let $Y(s) := \partial_t c_t(s)|_{t = 0}$ be the vector field along $c_0(s)$ determined by the family $(c_t)_{t \in (-1,1)}$. Denote $\dot{g}:=\pl_tg_t|_{t=0}$, and by $\nabla$ the Levi-Civita derivative defined by $g_0$.
	
	By definition $\ell_{g_t}(c_t) = \int_0^{\ell_{g_0}(c_0)} \sqrt{g_t\big(\partial_s c_t(s), \partial_s c_t(s)\big)} ds$, so differentiating we obtain:
	\begin{align}
	\label{eq:variation}
		\begin{split}
		\partial_t(\ell_{g_t}(c_t))|_{t=0} &= \frac{1}{2} \int_0^{\ell_{g_0}(c_0)} \frac{2g_0\big(\nabla_{\partial_t} \partial_s c_t(s)|_{t = 0}, \partial_s c_0(s)\big) + \dot{g}\big(\partial_s c_0(s), \partial_s c_0(s)\big)}{|\partial_s c_0(s)|_{g_0}}ds\\
		& = \frac{1}{2} \int_0^{\ell_{g_0}(c_0)} \dot{g}\big(\partial_s c_0(s), \partial_s c_0(s)\big) ds  \\
		& \quad + \int_0^{\ell_{g_0}(c_0)} \Big(\partial_s \big(g_0\big(\partial_t c_t(s), \partial_s c_t(s)\big)\big)\Big|_{t=0} - g_0\big(\partial_tc_t(s)|_{t=0}, \underbrace{\nabla_{\partial_s} \partial_s c_0(s)}_{=0}\big)\Big)ds\\
		&= \frac{1}{2} \int_0^{\ell_{g_0}(c_0)} \dot{g}\big(\partial_s c(s), \partial_s c(s)\big) ds + g_0\big(Y(s), \partial_s c_0(s)\big)|_{0}^{\ell_{g_0}(c_0)}.
		\end{split}
	\end{align}
	Here we used that $|\partial_s c_0(s)|_g = 1$ since the parametrisation of $c_0$ is by arc-length, and that $\nabla_{\partial_t} \partial_s = \nabla_{\partial_s} \partial_t$ (this is seen on the pullback bundle $c^*TM$ of the tangent bundle by the family $c$ since the connection is torsion-free and $[\partial_t, \partial_s] = 0$). In the third line, we used the compatibility of $g_0$ with $\nabla$, and the last term is zero since $\nabla_{\partial_s}\partial_s c_0(s) = 0$ is the geodesic equation.
	
	If $(x, v) \in \partial_-\mc{M}\setminus \Gamma_-^{g_0}$, then for $t$ small enough, $(x,v)\notin \Gamma_-^{g_t}$ by Proposition \ref{theorem:stability} and $\ell_{g_t}(x,v)$ is $C^2$ near $t=0$ by Lemma \ref{lemma:bound-dell}. Then, we get from \eqref{eq:variation}:
\begin{align*}
	\partial_t \ell_{g_t}(x, v)|_{t=0} &= \frac{1}{2} I_2^{g_0}(\dot{g})(x, v) + g_0\Big(\underbrace{\partial_{t}\Big(\pi \circ S_{g_t}\big(x, \frac{v}{|v|_{g_t}}\big)\Big)\Big|_{t=0}}_{= d\pi \circ \partial_tS_{g_t}(x, v)|_{t = 0}}, S_{g_0}(x, v)\Big)\\
	&= \frac{1}{2} I_2^{g_0}(\dot{g})(x, v) + \alpha_{S_{g_0}(x, v)}(\partial_t S_{g_t}(x, v)|_{t=0}).\qedhere
\end{align*}
\end{proof}

\subsubsection{Solenoidal gauge} The following lemma asserts that any metric in a neighborhood of a fixed metric $g_0$ can be put in a \emph{solenoidal gauge}.

\begin{lemma}
\label{lemma:solenoidal-gauge}
Let $(M,g_0)$ be a smooth Riemannian manifold with metric of Anosov type  let $k \geq 2, \alpha \in (0,1)$. There exists $C, \delta > 0$ such that the following holds: for all metrics $g$ such that $\|g-g_0\|_{C^{k,\alpha}} < \delta$, there exists a $C^{k+1,\alpha}$-diffeomorphism $\psi$, with $\psi|_{\pl M}={\rm Id}$, such that $\psi^*g$ is divergence-free with respect to $g_0$, namely $D^*_{g_0}(\psi^*g - g_0) = 0$, and $\|\psi^*g-g_0\|_{C^{k,\alpha}} \leq C \|g-g_0\|_{C^{k,\alpha}}$. 
\end{lemma}

\begin{proof}
The proof is contained in \cite[Lemma 2.2]{Croke-Dairbekov-Sharafutdinov-00}.
\end{proof}

\subsection{Normal operator} \label{section:normalop}
Let $(M,g)$ be a smooth Riemannian manifold with metric of Anosov type $g$. The normal operator on $m$-symmetric tensors is defined by
\[
\Pi^{g}_m :=(I_m^g)^* I_m^{g}.
\]
It enjoys strong analytic properties, as proved in \cite{Guillarmou-17-2}:

\begin{proposition}
\label{prop:pim}
The operator $\Pi^g_m \in \Psi^{-1}(M^\circ,\otimes^m_S T^*M^\circ)$ is a pseudodifferential operator of order $-1$ on $M^\circ$. It is elliptic on solenoidal tensors, in the sense that there exists pseudodifferential operator $Q,K_L,K_R$ on $M^{\circ}$ of respective order $1,-\infty,-\infty$ such that:
\[
Q\Pi^g_m = \pi_{\ker D_g^*} + K_L, \qquad \Pi^g_m Q = \pi_{\ker D_g^*} + K_R,
\]
and the equality holds when applied to all distributions $f \in \mc{E}'(M^\circ,\otimes^m_S T^*M^\circ)$ with compact support in $M^\circ$. The operator $Q$ can be taken to be properly supported in $M^\circ$. Moreover, $\Pi_m^g$ is solenoidal injective, i.e. injective in restriction to $\ker D_g^*$, if and only if the X-ray transform $I^g_m$ is solenoidal injective.
\end{proposition}

We now prove an elliptic estimate for the operator $\Pi_m^g$. Recall from \S\ref{ssection:hyperbolic-trapped-set} that $(M_e,g_e) \supset (M,g)$ is a Riemannian extension of the manifold $(M,g)$ which is also of Anosov type in the sense of Definition \ref{admissiblemetrics}. We will denote by 
\[ 
E_0 : L^2(M,\otimes^m_S T^*M) \to L^2(M_e,\otimes^m_S T^*M_e)
\]
the operator of extension by $0$.

\begin{proposition}
\label{proposition:elliptic-estimate}
Let $(M,g)$ be a manifold of Anosov type, and further assume that $I_2^g$ is solenoidal injective. Let 
$(M_e,g_e)$ be an  extension of Anosov type of $(M,g)$. Then, there exists $C > 0$ such that for all $f \in L^2(M,\otimes^2_S T^*M) \cap \ker D_g^*$:
 \[
\|f\|_{L^2(M)} \leq C \|\Pi_2^{g_e} E_0 f\|_{H^1(M_e)}.
\]
\end{proposition}
\begin{proof}
It will be convenient in the proof to consider a second extension of Anosov type  $(M_{ee}, g_{ee}) \supset (M_e,g_e)$ and to work on it.
The argument follows \cite{Stefanov-Uhlmann-04}. The operator $\Pi_2^{g_{ee}}$ is a (non-properly supported) pseudodifferential operator of order $-1$ on ${M_{ee}}^\circ$ which is elliptic on solenoidal tensors. By Proposition \ref{prop:pim}, we can construct a properly supported pseudo-differential operator  $Q \in \Psi^{1}(M_{ee}^\circ, \otimes^2_S T^* M_{ee}^\circ)$ such that
\[
Q \Pi_2^{g_{ee}} = \pi_{\ker D^*_{g_{ee}}} + K,
\]
where $K\in \Psi^{-\infty}(M_{ee}^{\circ})$ is smoothing. We let $\iota : M_e \hookrightarrow M_{ee}$ be the embedding. 
Observe that, taking a cutoff function $\chi \in C^\infty_{\comp}(M_e^\circ)$ with value $1$ in an open neighborhood of $M$, we get:
\[
\begin{split}
\iota^* Q \Pi_2^{g_{ee}} E_0 =&   \iota^* \pi_{\ker D_{g_{ee}}^*} E_0 + \iota^* K E_0 \\
= &  \pi_{\ker D_{g_e}^*}E_0 + \chi (\iota^* \pi_{\ker D_{g_{ee}}^*} - \pi_{\ker D_{g_e}^*}) \chi E_0 +  \iota^* K E_0 \\
& +  (1-\chi) (\iota^* \pi_{\ker D_{g_{ee}}^*} - \pi_{\ker D_{g_e}^*}) E_0 .
\end{split}
\]
By the pseudolocality of pseudodifferential operators (they preserve the singular support of distributions), the term $(\iota^* \pi_{\ker D_{g_{ee}}^*} - \pi_{\ker D_{g_e}^*}) E_0$ maps continuously $L^2$ sections to sections that are smooth outside $M$, and thus
\[
(1-\chi) (\iota^* \pi_{\ker D_{g_{ee}}^*} - \pi_{\ker D_{g_e}^*}) E_0 : L^2(M,\otimes^2_S T^*M) \to L^2(M_e,\otimes^2_S T^*M_e)
\]
is a compact operator. As to the term $\chi (\iota^* \pi_{\ker D_{g_{ee}}^*} - \pi_{\ker D_{g_e}^*}) \chi$, we observe that it has Schwartz kernel supported in the interior of $M_{ee} \times M_{ee}$. It is \emph{a priori} a pseudodifferential operator of order $0$ but its principal symbol vanishes (see Lemma \ref{lemma:potential-solenoidal}) and thus it is a pseudodifferential operator of order $-1$, i.e. it is compact as a map $L^2(M_e) \to L^2(M_e)$. (We now drop the notation of the vector bundle in the functional spaces in order to avoid repetition.) As a consequence, we see that, up to changing the compact remainder:
\begin{equation}
\label{equation:parametrix}
\iota^* Q \Pi_2^{g_{ee}} E_0 = \pi_{\ker D_{g_e}^*}E_0 + K,
\end{equation}
where $K$ is compact as a map $L^2(M) \to L^2(M_e)$.

Given $f \in L^2(M)\cap \ker D_g^*$, by Lemma \ref{lemma:potential-solenoidal} we may write $E_0f = D_gp + h$, where $p \in H^1(M_e, T^*M_e)$ and $p|_{\partial M_e} = 0$, $h = \pi_{\ker D_{g_e}^*} E_0f$. 
Now, using \eqref{equation:parametrix}, there is $C>0$ independent of $f$ such that 
\begin{equation}
\label{equation:f}
\begin{split}
\|f\|_{L^2(M)} & = \|E_0 f\|_{L^2(M_e)}  \\
& \leq \|\pi_{\ker D_{g_e}^*}E_0 f\|_{L^2(M_e)} +  \|({\rm Id}-\pi_{\ker D_{g_e}^*})E_0 f\|_{L^2(M_e)} \\
& \leq \|\iota^* Q \Pi_2^{g_{ee}} E_0f\|_{L^2(M_e)} + \|Kf\|_{L^2(M_e)} + \|D_{g_e}p\|_{L^2(M_e)} \\
& \leq C( \|\Pi_2^{g_{ee}} E_0 f\|_{H^1(M_{ee})} +  \|Kf\|_{L^2(M_e)} + \|D_{g_e}p\|_{L^2(M_e)}).
\end{split}
\end{equation}
It remains now to bound the potential term $D_{g_e}p$. We have
\begin{equation}
\label{equation:dp}
\|D_{g_e}p\|_{L^2(M_e)} \leq \|D_{g_e}p\|_{L^2(\Omega)} + \|D_{g}p\|_{L^2(M)},
\end{equation}
where $\Omega := M_e \setminus M^\circ$. We observe that on $\Omega$, $D_gp = -h = - \pi_{\ker D_{g_e}^*} E_0f$. Hence, using \eqref{equation:parametrix}, we get:
\begin{equation}
\label{equation:dp2}
\|D_{g_e}p\|_{L^2(\Omega)} \leq \|\iota^* Q \Pi_2^{g_{ee}} E_0 f\|_{L^2(\Omega)} + \|Kf\|_{L^2(\Omega)}.
\end{equation}
The boundary $\partial \Omega = \partial M_e \sqcup \partial M$ splits into two components. We define $\nu$ to be the outward pointing unit normal vector to $\partial \Omega$ and $j := p|_{\partial M}$. In $M$, we have $D_g^*f = 0 = D_g^*h + D_g^*D_gp = \Delta_D p$, where $\Delta_D := D_g^*D_g$ is the (symmetric) Laplacian on $1$-forms. Hence, in $M$, $p$ satisfies the elliptic system $\Delta_D p = 0, p|_{\partial M} = j \in H^{1/2}(\partial M, \otimes^2_S T^*M)$ (by the trace Theorem) so by standard elliptic estimates \cite[Chapter 5, Proposition 1.7]{Taylor-11}, we get $\|p\|_{H^1(M)} \lesssim \|j\|_{H^{1/2}(\partial M)}$. Observe that the $H^1$-norm in $M$  can be defined by  $\|p\|_{H^1(M)} := \|p\|_{L^2(M)} + \|D_gp\|_{L^2(M)}$. As a consequence, using the boundedness of the trace map $H^1(\Omega) \to H^{1/2}(\partial \Omega)$, we get (for some $C$ uniform that can change from line to line):
\begin{equation}
\label{equation:dp3}
\begin{split}
\|D_gp\|_{L^2(M)} & \leq C\|p\|_{H^1(M)}  \leq C \|j\|_{H^{1/2}(\partial M)} \leq  C\|p\|_{H^1(\Omega)} \leq  C(\|p\|_{L^2(\Omega)} + \|D_{g_e}p\|_{L^2(\Omega)}) \\
&  \leq C( \|p\|_{L^2(\Omega)} + \|\iota^* Q \Pi_2^{g_{ee}} E_0 f\|_{L^2(\Omega)} + \|Kf\|_{L^2(\Omega)}),
\end{split}
\end{equation}
by \eqref{equation:dp2}. It remains to bound $\|p\|_{L^2(\Omega)}$. Recall that $D_g p = \pi_{\ran D_g} E_0 f$, and by pseudolocality of the pseudodifferential operator $\pi_{\ran D_g}$ (see Lemma \ref{lemma:potential-solenoidal}) we get that $p|_{\Omega}$ belongs to $C^\infty(\Omega, T^*\Omega)$. For any point $(x,v) \in S\Omega$, there is a uniformly bounded time $\tau(x,v)$ (possibly negative) such that $\pi \left(\varphi_{\tau(x,v)}(x,v) \right) \in \partial M_e$ and using that $p$ vanishes on $\partial M_e$, we can thus write using \eqref{equation:relation}
\[
|\pi_1^*p(x,v)| = \Big|\int_0^{\tau(x,v)} (X_{g_e} \pi_1^*p)(\varphi^{g_e}_t(x,v)) \dd t \Big| = \Big|\int_0^{\tau(x,v)} (\pi_2^* D_{g_e}p)(\varphi^{g_e}_t(x,v)) \dd t\Big|.
\]
This equality implies that $\|p\|_{L^2(\Omega)} \leq C\| D_{g_e}p\|_{L^2(\Omega)}$. Hence, combining \eqref{equation:f} with \eqref{equation:dp}, \eqref{equation:dp2} and \eqref{equation:dp3}, we get that for all $f \in L^2(M,\otimes^2_S T^*M) \cap \ker D_{g}^*$, the following inequality holds for some uniform $C>0$
\[
\|f\|_{L^2(M)} \leq C\left( \|\Pi_2^{g_{ee}} E_0 f\|_{H^1(M_{ee})} +  \|Kf\|_{L^2(M_e)}\right),
\]
where $K : L^2(M,\otimes^2_S T^*M) \to L^2(M_{e}, \otimes^2_S T^*M_{e})$ is compact. The solenoidal injectivity of $\Pi^g_2$ on $M$ implies that 
$\Pi_2^{g_{ee}} E_0$ is also solenoidal injective (see \cite[Proof of Lemma 2.3]{Lefeuvre-19-1} for instance) and thus by standard arguments, one can remove the compact remainder $K$ from the previous inequality. Hence there is uniform $C > 0$ such that
\[
\|f\|_{L^2(M)} \leq C \|\Pi_2^{g_{ee}} E_0 f\|_{H^1(M_{ee})}.
\]
The claimed estimate is proved by observing that in the above proof one can replace $(M_{ee}, g_{ee})$ by $(M_e, g_e)$, and $(M_e, g_e)$ by a slightly smaller manifold $(M_{e}', g_e')$ of Anosov type containing $(M, g)$.
\end{proof}

\section{Local lens rigidity, proof of the main result}

In this section, we prove the main Theorem \ref{theorem:main}.

\subsection{Key estimate}
\label{ssection:key}

The goal of this paragraph is to show the following key estimate:

\begin{proposition}
\label{proposition:key}
Let $g_0$ be of Anosov type. There exist $C, \eps, \mu, N > 0$ such that for all smooth metrics $g$ such that $g|_{\pl M}=g_0|_{\pl M}$, 
$\|g-g_0\|_{C^N} < \eps$, and $(\ell_g,S_g) = (\ell_{g_0},S_{g_0})$, then:
\[
\|I_2^{g_0}(g-g_0)\|_{L^2} \leq C \|g-g_0\|_{C^N}^{1+\mu}.
\]
\end{proposition}

In order to prove Proposition \ref{proposition:key}, we are still missing one ingredient, namely, the following $C^2$-regularity of the scattering operator.
\begin{proposition}\label{regularityR}
Let $(M,g_0)$ be a smooth compact Riemannian manifold with strictly convex boundary and hyperbolic trapped set. Let $\chi \in C^\infty_{\comp}(\partial_-\mc{M},[0,1])$ be a smooth cutoff function. 
Then, for each $\omega\in C^\infty(\pl_+\mc{M})$ the map
\[
C^\infty(M, \otimes^2_S T^*M) \ni g\mapsto \chi \mc{S}_g(\omega) \in H^{-6}(\partial_-\mc{M})
\]
is $C^2$-regular near $g_0$. As a consequence, there exists $C,N>0$ large enough and $\delta>0$ such that for all $g\in C^\infty(M, \otimes^2_S T^*M)$ with $\|g-g_0\|_{C^N}\leq \delta$, the following holds:
\begin{equation}\label{Taylororder2S} 
\| \chi \mc{S}_g(\omega) -\chi \mc{S}_{g_0}(\omega) + \chi \pl_{g}\mc{S}_g(\omega)|_{g=g_0}.(g-g_0)\|_{H^{-6}(\pl_-\mc{M})}\leq C\|g-g_0\|^2_{C^N(M,\otimes^2_ST^*M)}. 
\end{equation}
\end{proposition}

Since this result is quite technical, its proof is postponed to Section \S\ref{section:scattering}. In the following, we will write $h := g-g_0$. Using a complex interpolation argument, Proposition \ref{proposition:key} is actually a direct consequence of the following technical lemma, which gives weighted estimates on the $X$-ray transform of $g-g_0$.

\begin{center}
\begin{figure}[htbp!]
\includegraphics[scale=0.75]{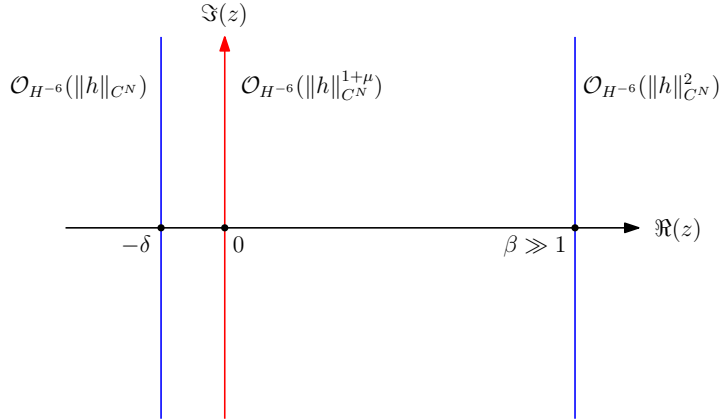}
\caption{Estimates on $f(z) = e^{-z \ell_{g_0}} I_2^{g_0}h$ in \eqref{equation:toprove}. For $z$ on the left blue line we have a ``volume estimate'' of $f(z)$, while for $z$ on the right blue line we have a ``microlocal estimate'' of $f(z)$. For $z$ on the red line we have the interpolation estimate obtained in Proposition \ref{proposition:key}.}

\label{figure:volume-microlocal-interpolation}
\end{figure}
\end{center} 

\begin{lemma}
\label{lemma:key}
There exist $C, \eps, \delta, \beta, N > 0$ such that for all smooth metrics $g$ such that $g|_{\pl M}=g_0|_{\pl M}$, 
$\|g-g_0\|_{C^N} < \eps$, and $(\ell_g,S_g) = (\ell_{g_0},S_{g_0})$, we have for $h=g-g_0$:
\begin{equation}
\label{equation:toprove}
\|(1+z)^{-7} e^{-z \ell_{g_0}} I_2^{g_0}h\|_{H^{-6}(\partial_-\mc{M})} \leq \left\{ \begin{array}{l} C\|h\|_{C^N}, \forall z \in i\R - \delta, \\ C\|h\|^2_{C^N}, \forall z \in i\R + \beta. \end{array} \right.
\end{equation}
\end{lemma}

We now show that Lemma \ref{lemma:key} implies Proposition \ref{proposition:key}. The rest of \S\ref{ssection:key} is devoted to the proof of Lemma \ref{lemma:key}.

\begin{proof}[Proof of Proposition \ref{proposition:key}]
By the Hadamard three line Theorem applied to the function $z\mapsto (1+z)^{-7} e^{-z \ell_{g_0}} I_2^{g_0}(h)$ (which is bounded in $\Re(z)\in [-\delta,\beta]$ with values in $L^2(\pl_-\mc{M}) \subset H^{-6}(\pl_-\mc{M})$), Lemma \ref{lemma:key} implies that
\[
\|I_2^{g_0}h\|_{H^{-6}(\partial_-\mc{M})} \leq C \|h\|_{C^N(M)}^{1+\mu},
\]
for some constants $C,\mu > 0$ independent of $h$ (note that $\mu$ depends on $\delta,\beta$). By Lemma \ref{boundednessIg}, there is $C>0,s>0$ depending on $g_0$ such that (for $N \geq 2$)
\[
\|I_2^{g_0}h\|_{H^{s}(\partial_-\mc{M})} \leq C \|h\|_{C^N(M)}.
\]
  Interpolating $L^2(\pl_-\mc{M})$ between $H^{-6}(\pl_-\mc{M})$ and $H^s(\pl_-\mc{M})$, we deduce that there exists $\mu'>0$ and $C>0$ such that
		\[
		\|I^{g_0}_2h\|_{L^2(\pl_-\mc{M})} \leq C \|h\|_{C^N(M)}^{1+\mu'}. \qedhere
		\]
\end{proof}

We now start with the proof of Lemma \ref{lemma:key}. See Figure \ref{figure:volume-microlocal-interpolation}: on $\{\Re(z) = -\delta\}$ the bound will follow from an estimate on the volume of long trajectories, while the estimate on the line $\{\Re(z) = \beta\}$ may be thought of as a ``microlocal estimate'' since it crucially relies on the Taylor expansion of $g \mapsto \mc{S}_g$ obtained in Proposition \ref{regularityR}. 

The first bound in \eqref{equation:toprove} for $z \in i\R - \delta$ follows directly from the following stronger bound:

\begin{lemma}
\label{lemma:volume}
There exists $\delta > 0$ small enough and $C > 0$ (depending on $\delta$) such that for all $h \in C^0(M,\otimes^2_S T^*M)$:
\[
\|e^{\delta \ell_{g_0}} I_2^{g_0}h\|_{L^2(\pl_-\mc{M})} \leq C\|h\|_{C^0(M)}.
\]
\end{lemma}

\begin{proof}
For $y\notin \Gamma^{g_0}_-$, we have $|I_2^{g_0}h(y)| \leq \|h\|_{C^0} |\ell_{g_0}(y)|$. Thus 
\[\|e^{\delta \ell_{g_0}} I_2^{g_0}h\|_{L^2(\pl_-\mc{M})}\leq \|e^{\delta \ell_{g_0}}\ell_{g_0}\|_{L^2(\pl_-\mc{M})}\|h\|_{C^0}\] 
which gives the result by \eqref{cavalieri}, if $\delta<Q_{g_0}/2$.
\end{proof}

We now study the second bound in \eqref{equation:toprove}. Let $\chi \in C_{\comp}^\infty(\partial_-\mc{M},[0,1])$ be a smooth cutoff function. First of all, near the boundary, we have the following:

\begin{lemma}\label{lemma:microlocal-nearboundary}
There exist $C, \eps > 0$, and $\chi \in C_{\comp}^\infty(\partial_-\mc{M},[0,1])$ such that $1 - \chi^2$ is supported near the boundary of $\partial_-\M$, such that if $\|g-g_0\|_{C^N} < \eps$ and $(\ell_g,S_g) = (\ell_{g_0},S_{g_0})$, then:
\[
\|(1-\chi^2) I_2^{g_0} h\|_{L^\infty(\partial_-\mc{M})} \leq C\|h\|^2_{C^1}.
\]
\end{lemma}
\begin{proof}
This follows from \cite[Section 9]{Stefanov-Uhlmann-04} as we have the following Taylor expansion for $x,x' \in \partial M$ close enough:
\[
d_g(x,x') = d_{g_0}(x,x') + \frac{1}{2} I_2^{g_0}h(x,x') + T_g(x,x'),
\]
with the bound $|T_g(x,x')| \leq C \|h\|^2_{C^1} d_{g_0}(x,x')$, where $C > 0$ is a uniform constant depending only on $g_0$. Since the metrics have same lens data, they also have same boundary distance function for $x,x' \in \partial M$ close enough, that is, $d_g(x,x') = d_{g_0}(x,x')$, which easily implies the claimed estimate when $1 - \chi^2$ is taken to have support near the boundary of $\partial_-\M$ (i.e. close to short geodesics). 
\end{proof}

Using the continuous embeddings $L^\infty(\partial_-\mc{M}) \hookrightarrow L^2(\partial_-\mc{M}) \hookrightarrow H^{-6}(\partial_-\mc{M})$, from Lemma \ref{lemma:microlocal-nearboundary} we deduce that
\begin{equation}
\label{equation:izy}
\|(1+z)^{-7} e^{-z \ell_{g_0}} (1-\chi^2)I_2^{g_0} h\|_{H^{-6}(\partial_-\mc{M})} \leq C\|h\|^2_{C^N},
\end{equation}
for all $z \in i \R + \beta$. It thus remains to prove the following estimate to deduce the second bound of \eqref{equation:toprove}.

\begin{lemma}
\label{lemma:microlocal}
There exist $C, \eps, \beta, N > 0$ such that if $\|g-g_0\|_{C^N} < \eps$ and $(\ell_g,S_g) = (\ell_{g_0},S_{g_0})$, then  for $h:=g-g_0$ and for all $z \in i\R + \beta$:
\[
\|(1+z)^{-7} e^{-z \ell_{g_0}} \chi^2 I_2^{g_0} h\|_{H^{-6}(\partial_-\mc{M})} \leq C \|h\|^2_{C^N}.
\]
\end{lemma}

\begin{proof}
We let $\imath_{\pl_-} : \partial_- \mc{M} \to \mc{M}$ be the inclusion map. For $\beta>0,$ we consider the space
\begin{equation}
\label{equation:ebeta}
E_\beta := C_b^0\left(\beta + i\R, L^2(\partial_-\mc{M})\right),
\end{equation}
where $C^0_b$ denotes the vector space of bounded continuous functions, equipped with the $L^\infty$ norm. It is a Banach space when equipped with the norm:
\[
\|F\|_{E_{\beta}} := \sup_{z \in \beta + i\R}\|F(z)\|_{L^2(\partial_-\mc{M})}.
\]
Then for $z\in \C$ with $\Re(z)=\beta$ large (it will be adjusted later), we define for $U_{g_0}$ the neighborhood of $g_0$ introduced in \eqref{equation:neighb} (with $p = N - 2$)
\begin{equation}
\label{equation:psi}
\mc{F}: U_{g_0} \ni g \mapsto \mc{F}(g)(z):= (1+z)^{-7}\chi^2 \frac{(1-e^{-z\ell_g})}{z} \in E_{\beta},
\end{equation}
where the value at $z=0$ is set to be $\chi^2\ell_g$.

First, the function $\mc{F}$ is $C^2$ by Lemma \ref{expsell} by taking $N\geq 5$. We compute its Taylor expansion in the space 
$E_\beta$: for some $N$ large enough, $g$ close enough to $g_0$, and $h:=g-g_0$
\begin{equation}
\label{equation:taylor-rg}
\begin{split}
\mc{F}(g)(z)-\mc{F}(g_0)(z)=& \frac{\chi^2 e^{-z\ell_{g_0}} }{(1+z)^7} (\partial_g \ell_g)|_{g=g_0}.h + \mc{O}_{L^2(\partial_-\mc{M})}(\|h\|^2_{C^N})\\
=& \frac{\chi^2 e^{-z\ell_{g_0}}}{(1+z)^7}\Big(\frac{1}{2} I_2^{g_0}(h)+\alpha_{S_{g_0}(\cdot)}(\pl_gS_g(\cdot)|_{g=g_0}.h)\Big)+\mc{O}_{L^2(\partial_-\mc{M})}(\|h\|^2_{C^N})
\end{split}
\end{equation}
and the remainder is bounded uniformly in $z$ (by Lemma \ref{expsell} again), where we use Lemma \ref{lemma:variation} in the second line (recall $\alpha$ is the Liouville $1$-form). If $\ell_g=\ell_{g_0}$, we obtain in particular $\mc{F}(g)(z)-\mc{F}(g_0)(z)=0$, thus
\begin{equation}
\label{equation:borne-1}
\sup_{z \in \beta + i\R} \Big\|\frac{\chi^2 e^{-z\ell_{g_0}}}{(1+z)^7}\Big(\frac{1}{2} I_2^{g_0}(h)+\alpha_{S_{g_0}(\cdot)}(\pl_gS_g(\cdot)|_{g=g_0}.h)\Big)\Big\|_{L^2(\pl_-\mc{M})} 
\leq C \|h\|^2_{C^N}.
\end{equation}
Note that, for $\Re(z) = \beta>0$, as a consequence of \eqref{Igbounded} we have $\chi^2 e^{- z \ell_{g_0}} I_2^{g_0}(h) \in L^2(\partial_-\mc{M})$,  thus since by Lemma \ref{expsell} we know that $\pl_g\mc{F}(g)(z)|_{g=g_0}.h \in L^2(\pl_-\mc{M})$ if $\beta$ is large enough, we obtain that
\[
\chi^2 \alpha_{S_{g_0}(\cdot )}\left(\partial_g S_g(\cdot)|_{g=g_0}.h\right)e^{- z \ell_{g_0}} \in L^2(\partial_-\mc{M}).
\]
We now claim the following Lemma, the proof of which is deferred to below.
\begin{lemma}
\label{lemma:h-minus}
There exist $C, \eps, \beta, N > 0$ such that if $\|g-g_0\|_{C^N} < \eps$ and $(\ell_g,S_g) = (\ell_{g_0},S_{g_0})$, then for all $z \in i\R + \beta$ and $h=g-g_0$:
\[
\|(1+z)^{-7}\chi^2 \alpha_{S_{g_0}(\cdot)}\left(\partial_g S_{g_0}(\cdot)|_{g=g_0}.h\right) e^{- z \ell_{g_0}}\|_{H^{-6}(\partial_-\mc{M})}  \leq C \|h\|^2_{C^N}.
\]
\end{lemma}
Using \eqref{equation:borne-1} and Lemma \ref{lemma:h-minus}, we deduce that for all $\Re(z) = \beta$ with $\beta,N>0$ large enough 
\begin{align*}
& \sup_{z \in \beta + i\R} |1+z|^{-7}  \|\chi^2 I_2^{g_0}(h) e^{-z \ell_{g_0}}\|_{H^{-6}(\partial_-\mc{M})} \\
&  \leq \sup_{z \in \beta + i\R} |1+z|^{-7} \|\chi^2 \alpha_{S_{g_0}(\cdot)}\left(\partial_g S_{g_0}(\cdot)|_{g=g_0}.h\right) e^{- z \ell_{g_0}}\|_{H^{-6}(\partial_-\mc{M})}   +  C \|h\|^2_{C^N}\\
& \leq C \|h\|^2_{C^N},
\end{align*}
where the constant $C > 0$ changes from line to line. This concludes the proof of Lemma \ref{lemma:microlocal}.
\end{proof}

 \begin{proof}[Proof of Lemma \ref{lemma:h-minus}]
 Taking a finite cover of $\mc{M} = \cup_i U_i$, a partition of unity $\sum_i \chi_i = \mathbf{1}$ subordinate to that cover, we may write 
\begin{equation}
\label{equation:decomp-alpha}
\alpha = \sum_{i,j} \alpha^{(j)}_i dy ^{(j)}_i,
\end{equation}
where $\alpha^{(j)}_i, y^{(j)}_i \in C^\infty(\mc{M})$ are smooth functions compactly supported inside $U_i$ and thus for $y \not\in \Gamma^{g_0}_-$, we have:
\begin{equation}
\label{equation:relations-alpha}
\begin{split}
\chi^2 \alpha_{S_{g_0}(y)}\left(\partial_g S_{g_0}(y)|_{g=g_0}.h\right) e^{- z \ell_{g_0}(y)} & = \chi^2 \sum_{i,j} \alpha^{(j)}_i(S_{g_0}(y)) \langle d y^{(j)}_i, \partial_g S_g(y)|_{g=g_0}.h\rangle e^{- z \ell_{g_0}(y)} \\
& =  \sum_{i,j}  \chi \mc{S}_{g_0}(\alpha^{(j)}_i) (y) e^{- z \ell_{g_0}(y)} \cdot \chi  \partial_g \mc{S}_g(y^{(j)}_i)(y)|_{g=g_0}.h
\end{split}
\end{equation}
First, taking $\beta > 0$ large enough, we can ensure by Lemma \ref{lemma:upperboundS_g} the existence of a constant $C > 0$ such that for all $z \in i\R + \beta$, and for all $i, j$, one has $\chi \mc{S}_{g_0}^*\alpha^{(j)}_i e^{- z \ell_{g_0}} \in C^6(\partial_-\mc{M})$ with the uniform bound
\begin{equation}
\label{equation:c4}
\|(1+z)^{-7}\chi \mc{S}_{g_0}(\alpha^{(j)}_i) e^{- z \ell_{g_0}}\|_{C^6(\partial_-\mc{M})} \leq C.
\end{equation}

We now let $f \in C^\infty(\mc{M})$ be one of the functions $y_i^{(j)}$ in \eqref{equation:decomp-alpha}. By Proposition \ref{regularityR}, we have 
\[
\chi \mc{S}_gf =  \chi  \mc{S}_{g_0}f + \chi  \partial_g\mc{S}_gf|_{g=g_0}.h + \mc{O}_{H^{-6}(\partial_-\mc{M})}(\|h\|^2_{C^N}).
\]
(The constant in the $\mc{O}$ notation depends on the function $f$, but there are only finitely many functions $y_i^{(j)}$ considered in the end so the constant will be uniform.) Now, using that the scattering relations are the same, i.e. $S_g = S_{g_0}$, we have $\chi \mc{S}_g^* f = \chi  \mc{S}_{g_0}^* f$, where the equality holds in $L^\infty(\partial_-\mc{M})$, hence in $L^2(\partial_-\mc{M}) \subset H^{-6}(\partial_-\mc{M})$. As a consequence, we deduce that:
\begin{equation}
\label{equation:z-bound}
\|\chi  \partial_g\mc{S}_g^* y_i^{(j)} |_{g=g_0}.h\|_{H^{-6}(\partial_-\mc{M})} \leq C \|h\|^2_{C^N}.
\end{equation}

Using both \eqref{equation:c4} and \eqref{equation:z-bound} in \eqref{equation:relations-alpha} and the continuity of the multiplication $C^6(\partial_-\mc{M}) \times H^{-6}(\partial_-\mc{M}) \ni (u,v) \mapsto u v \in H^{-6}(\partial_-\mc{M})$, we deduce that for some $C > 0$
\[
\|(1+z)^{-7}\alpha_{S_{g_0}(\cdot)}\left(\partial_g S_{g_0}(\cdot)|_{g=g_0}.h\right) e^{- z \ell_{g_0}}\|_{H^{-6}(\partial_-\mc{M})}  \leq C \|h\|^2_{C^N}.
\]
This concludes the proof of Lemma \ref{lemma:h-minus}.
\end{proof}

\subsection{End of the proof} We can now complete the proof of Theorem \ref{theorem:main}.

\begin{proof}[Proof of Theorem \ref{theorem:main}]

Assume that $(\ell_g,S_g) = (\ell_{g_0},S_{g_0})$ and $g$ is close enough to $g_0$ in the $C^N$-topology. By Lemma \ref{lemma:solenoidal-gauge}, we can find a diffeomorphism $\psi$ such that $\psi|_{\partial M} = {\rm Id}_{\partial M}$ and $g' := \psi^* g$ is solenoidal with respect to $g_0$. Moreover, $(\ell_{g'},S_{g'}) = (\ell_g,S_g)=(\ell_{g_0},S_{g_0})$. Also note that $\|g'-g_0\|_{C^N} \leq C\|g-g_0\|_{C^N}$ for some uniform $C>0$ (depending on $g_0$).

Writing $h := g'-g_0$, Proposition \ref{proposition:key} implies that:
\begin{equation}\label{eq:I_2^g_0}
	\|I_2^{g_0}h\|_{L^2} \leq C\|h\|_{C^N}^{1 + \mu}.
\end{equation}
Now recall that for any $\eps > 0$, the adjoint $(I_2^{g_{0e}})^*: L^2 \to L^{p(\eps)} \subset H^{-\varepsilon}$ is bounded (here $p(\eps) < 2$ and $p(\eps) \to 2$ as $\eps \to 0$), see \cite[Lemma 5.1 and Equation (5.3)]{Guillarmou-17-2}.
		
By \eqref{eq:I_2^g_0}, and since $\Pi_2^{g_{0e}}$ is of order $-1$ (by Proposition \ref{prop:pim}), and $E_0h$ has regularity $H^{\frac{1}{2} - \varepsilon}$ for any $\varepsilon > 0$, we conclude that for any $\varepsilon > 0$, where $C > 0$ changes from line to line:
\begin{align*}
& \|\Pi_2^{g_{0e}} E_0 h\|_{H^{-\varepsilon}} = \|(I_2^{g_{0e}})^*I_2^{g_{0e}} E_0 h\|_{H^{-\varepsilon}} \leq C \|I_2^{g_{0e}} E_0 h\|_{L^2} \leq C \|I_2^{g_{0}} h\|_{L^2} \leq C\|h\|_{C^N}^{1 + \mu},\\
& \|\Pi_2^{g_{0e}} E_0 h\|_{H^{\frac{3}{2} - \varepsilon}} \leq C \|E_0h\|_{H^{\frac{1}{2} - \varepsilon}} \leq C\|h\|_{C^N}.
\end{align*}
By interpolation in Sobolev spaces, we obtain from these two estimates that, for some (different) $C,\mu > 0$:
\[\|\Pi_2^{g_{0e}} E_0 h\|_{H^1} \leq C \|h\|_{C^N}^{1 + \mu}.\]
Applying the elliptic stability estimate for solenoidal tensors of Proposition \ref{proposition:elliptic-estimate} (using that our assumption implies that $I_2^{g_0}$ is solenoidal injective), we get:
\[\|h\|_{L^2} \leq C  \|\Pi_2^{g_{0e}} E_0 h\|_{H^1} \leq C\|h\|_{C^N}^{1 + \mu}.\]
By interpolation, we then obtain for some (much larger) other integer $N \in \N$:
\[\|h\|_{L^2} \leq C \|h\|_{L^2} \|h\|_{C^{N}}^{\mu} \leq C \|h\|_{L^2} \|g-g_0\|_{C^N}^{\mu}.\]
If $\|g-g_0\|_{C^{N}}< (1/C)^{1/\mu}$, this implies that $h = 0$, namely $g' = \psi^*g = g_0$.
\end{proof}

\section{Smoothness of the scattering operator with respect to the metric}

\label{section:scattering}

The goal of this section is to prove Theorem \ref{axiomAsmooth} and to derive Proposition \ref{regularityR} as a corollary. Theorem \ref{axiomAsmooth} will follow directly from Theorem \ref{theorem:meromorphic} and Lemma \ref{lemma:res} below. The scattering operator $\mc{S}_g$ can be expressed purely in terms of the resolvent $R_{g_e}$ of $X_{g_e}$ thanks to Lemma \ref{SchwartzkernelSg}. 
Thus, in order to analyze the map $g\mapsto \mc{S}_g$, we shall study the regularity of the map $g\mapsto R_{g_e}$ in adequate functional spaces. Since working with $g_e$ or $g$ is equivalent (they share exactly the same properties), we shall consider $R_g$ for the simplicity of notation.
The construction of $R_g$ is done using microlocal methods as in \cite{Dyatlov-Guillarmou-16}, but we need to understand the $g$-dependence in the construction. We fix a metric of Anosov type $g_0$ on $M$ and we denote by $X_0$ its associated geodesic vector field on $\mc{M}$. We will consider the resolvent of $X$ if $X$ is any smooth vector field that is close enough to $X_0$ in $C^2(\mc{M},T\mc{M})$. We refer to \S\ref{sssection:geom-cons}, where the notation for the cotangent bundle is introduced.

\subsection{Construction of the uniform escape function}

\label{ssection:escape}

In this paragraph, we construct a \emph{uniform escape function}, i.e. an escape function\footnote{A function decreasing along the bicharacteristics of the symplectic lift of $X$ to the cotangent bundle.}  for $X_0$ which is also an escape function for all vector fields $X$ that are sufficiently close to $X_0$. We will use an idea of Bonthonneau \cite{Bonthonneau-20} in order to obtain an escape function adapted to all flows $X$ close to $X_0$.
Denote by $S^*\mc{M}:=(T^*\mc{M}\setminus \{0\})/\R^+$ (and similarly $S^*\mc{N}$) the spherical bundle, by $\kappa:T^*\mc{M}\to S^*\mc{M}$ the quotient projection, by $\pi: S^*\mc{N} \to \mc{N}$ the footpoint map, and recall that $\X$ is the generator of the symplectic lift of $\varphi_t$ defined in \eqref{eq:symplectic-lift}. Finally, recall that $\rho_0>0$ is the constant of \S\ref{sssection:extension} used to define the extension $\mc{M}_e$, and that $\widetilde{X}_0$ is some initial extension of the vector field from $\mc{M}$ to $\mc{N}$ (which does not need to vanish at $\{\rho = -\rho_0\}$). 

\begin{proposition}
\label{proposition:escape}
There exists a smooth function $m \in C^\infty(S^*\mc{N},[-1,1])$, invariant by the antipodal map $(x,\xi) \mapsto (x,-\xi)$, and $\delta > 0$ such that for all vector fields $X \in C^\infty(\mc{M},T\mc{M})$ such that $\|X-X_0\|_{C^2(\mc{M},T\mc{M})} \leq \delta$, the following holds:
\begin{enumerate}
\item $m=1$ in a neighborhood of $(E_-^X)^* \cap \pi^{-1}(\mc{M})$, 
\item $m=-1$ in a neighborhood of $(E_+^X)^*\cap \pi^{-1}(\mc{M})$,
\item $\supp(m) \cap \pi^{-1}(\mc{M})$ is contained in a small conic neighborhood of $(E_-^X)^*$ and $(E_+^X)^*$,
\item $\supp(m) \subset \left\{\rho > -2\rho_0\right\}$,
\item $\supp(m) \cap \left\{\rho=-\rho_0\right\} \cap \left\{\widetilde{X}_0 \rho=0\right\} = \emptyset$,
\item $\X m \leq 0$.
\end{enumerate}
\end{proposition}

The fact that $X$ and $X_0$ are $C^2$-close will ensure that the structural stability Proposition \ref{theorem:stability} applies. The function $m$ will be constructed as
\begin{equation}
\label{equation:m}
m = m_- -m_+ + \eta^{-1}(\pi^* \chi_- - \pi^* \chi_+),
\end{equation}
where $m_\pm$ are smooth functions with support near $(E_\pm^X)^*$ and taking value $1$ on $(E_\pm^X)^*$, $\chi_\pm$ are smooth functions with compact support in a slightly larger neighborhood of $\Sigma_\pm$ (defined in \eqref{equation:sigma}), and $\eta > 0$ will be a small parameter chosen small enough in the end. We refer to \S\ref{sssection:geom-cons} where all the previous notation are defined. The proof being rather technical, we advise the reader to have in mind Figure \ref{figure:construction} below, where the various sets and functions of the construction are depicted.

\begin{center}
\begin{figure}[htbp!]
\includegraphics[scale=0.7]{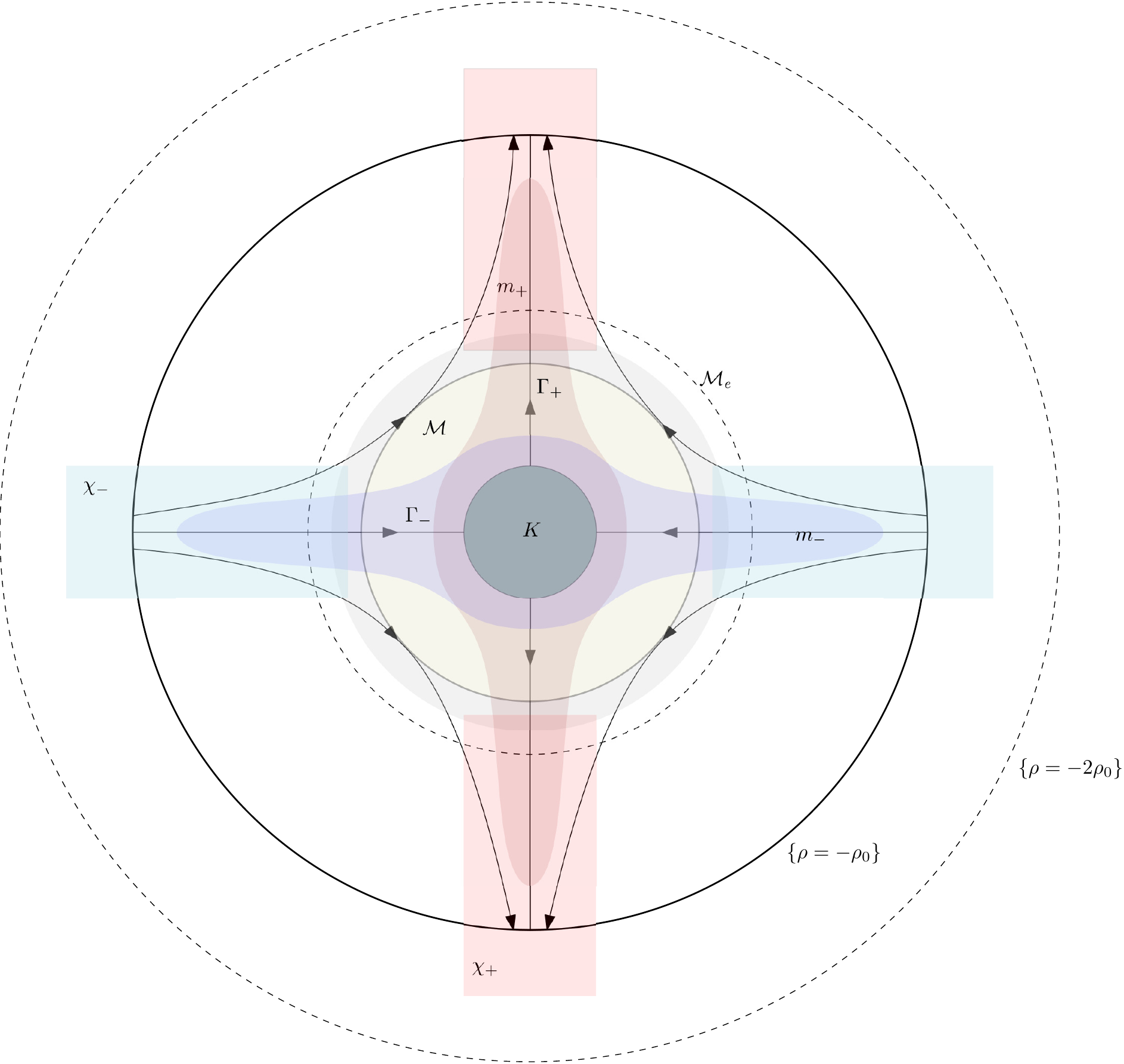}
\caption{A schematic representation of the various sets and functions appearing in Lemmas \ref{lemma:m-plus}, \ref{lemma:chi-plus}. The disks represent (respectively, from the center to the outer disk): the trapped set $K$ of $X_0$, the manifold $\mc{M}$, the set $\left\{q=0\right\}$ (in light gray) defined in \S \ref{ssection:meromorphic}, the extended manifold $\mc{M}_e$, the set $\left\{\rho \geq -\rho_0\right\}$, the set $\left\{\rho \geq -2\rho_0\right\}$. The support of the functions $m_+,\chi_+,m_-,\chi_-$ are depicted respectively in: dark red, light red, dark blue, light blue. The flowlines of $X_0$ are represented in black with arrows indicating the flow direction.}
\label{figure:construction}
\end{figure}
\end{center} 

\begin{remark}
More generally, one could construct a function $m$ taking any positive (resp. negative) constant value near $(E_-^X)^*$ (resp. $(E_+^X)^*$) but this will not be needed.
\end{remark}

\subsubsection{Uniform cone contraction}

We start with some technical lemmas on the contraction of cones in $T^*\mc{M}$. In order to abbreviate notation, we will sometimes write $X \sim X_0$ if $\|X-X_0\|_{C^2} \leq \delta$, where $\delta > 0$ is some small constant, that will be chosen later. In what follows, we will use the notion of conic neighborhoods of conic sets in $T^*\mc{N} \setminus 0$, which may be identified with neighborhoods on the spherical bundle $S^*\mc{N}$. First of all, we have:

\begin{lemma}
\label{lemma:close}
Let $\mc{U}$ be an open neighborhood  of  the trapped set $K^{X_0}$. Then, there exists $\delta > 0$ and $T \geq 0$ such that for all $t \geq T$, and all smooth vector fields $X$ such that $\|X-X_0\|_{C^2(\mc{M},T\mc{M})} < \delta$:
\[
y, \varphi^X_{-t}(y),  \varphi^X_t(y) \in \mc{M}_e \implies y \in \mc{U}.
\]
\end{lemma}

Taking $X \sim X_0$ close enough in the $C^2$-topology, we can ensure that $\mc{U}$ is also an open neighborhood of $\cup_{X \sim X_0} K^X$ by the structural stability Proposition \ref{theorem:stability}.

\begin{proof}
We argue by contradiction. Assume that we can find sequences $(T_j)_{j \geq 1}$, $T_j \to +\infty$, $(X_j)_{j \geq 1}$ such that $X_j \to X_0$ in $C^2(\mc{M},T\mc{M})$, and $(y_j)_{j \geq 1}$ such that $y_j \in \mc{M}_e$, $\varphi^{X_j}_{-T_j}(y_j) \in  \mc{M}_e$ and $\varphi^{X_j}_{T_j}(y_j) \in  \mc{M}_e$ but $y_j \notin \mc{U}$. By compactness of $\mc{M}_e$, we can always assume, up to extraction, that $y_j \to y_\infty \in \mc{M}_e$. But then $y_\infty \in K^{X_0}$, which contradicts $y_\infty \notin \mc{U}$.
\end{proof}

We now show the existence of small conic subsets in $T^*\mc{M}$, independent of the vector field $X$, on which the differential of the flow $(\varphi_t^X)_{t \in \R}$ is exponentially expanding/contracting. This may be compared to \cite[Lemma 2.11]{Dyatlov-Guillarmou-16}.

\begin{lemma}
\label{lemma:expansion}
There exist $\delta > 0$ small enough, constants $C,T,\lambda > 0$ and small open conic neighborhoods $U_\pm$ of $\cup_{X \sim X_0} (E_\pm^X)^*$, such that for all $X$ with $\|X -X_0\|_{C^2} \leq \delta$, the following holds: for all $(y,\xi) \in U_\pm$, for all $t \geq T$ such that $y, \varphi_{\pm t}^X(y) \in \mc{M}_e$,
\[
\begin{split}
\forall s \in [0,t-T], \, e^{\pm s\X}(y,\xi) \in U_\pm, \qquad \forall s \in [0,t],\,   |e^{\pm s\X}(y,\xi)| \geq C e^{\lambda s} |\xi|. 
\end{split}
\] 
\end{lemma}

\begin{proof}
We prove the lemma for the outgoing (+) direction, the proof being similar for the incoming (-) direction. Fix arbitrary small conic neighborhoods $\widetilde{U}_+^{(2)} \Subset \widetilde{U}_+^{(1)}$ of $(E_+^{X_0})^*$. By hyperbolicity, there is a $T_0 > 0$ large enough, such that the following holds: for all $(y,\xi) \in T^*_{\Gamma^{X_0}_+}\mc{M}_e \cap \widetilde{U}_+^{(1)}$ such that $y, \varphi^{X_0}_{T_0}(y) \in \mc{M}_e$, one has:
\[
e^{T_0 \X_0}(y,\xi) \in T^*_{\Gamma^{X_0}_{+}}\mc{M}_e \cap  \widetilde{U}_+^{(2)}, \qquad |e^{T_0\X_0}(y,\xi)| \geq 10 |\xi|.
\]
By continuity, there exist small neighborhoods $U_+^{(j)}$ of $\widetilde{U}_+^{(j)}$ such that the following hold:
\begin{enumerate}
\item The neighborhoods are chosen so that $\pi(U_+^{(1)}) \Subset \pi(U_+^{(2)})$.
\item Letting $W := \pi(U_+^{(1)})$, one has $U_+^{(2)} \cap W \Subset^{\mathrm{fiber}} U_+^{(1)} \cap W$, in the sense that for all $y \in W$, $U_+^{(2)} \cap T^*_y\mc{M}_e \Subset U_+^{(1)} \cap T^*_y\mc{M}_e$.
\item For all $(y,\xi) \in U^{(1)}_+$ such that $y,\varphi^{X_0}_{T_0}(y) \in \mc{M}_e$,
\[
e^{T_0 \X}(y,\xi) \in U_+^{(2)},\quad  |e^{T_0 \X}(y,\xi)| \geq 5 |\xi|.
\]
\item There is a time $T_1 > T_0$ such that: if $y \in \pi(U_+^{(2)}) \setminus \pi(U_+^{(1)})$, then for all $t \geq T_1$, $\varphi_t^{X_0}(y) \notin \mc{M}_e$.
\end{enumerate}

By continuity, this can be achieved so that points (1-4) also hold for all smooth vector fields $X$ such that $\|X-X_0\|_{C^1} \leq \delta$, where $\delta > 0$ is chosen small enough. We will actually choose $\|X-X_0\|_{C^2} \leq \delta$, where $\delta > 0$ is chosen small enough: by the structural stability Proposition \ref{theorem:stability}, we can then ensure that the neighborhoods $U_+^{(j)}$ also contain $(E_+^{X})^*$ for $X \sim X_0$ in the $C^2$-topology.

We set $U_+ := U_+^{(1)}$ and $T := 3T_1$ and we claim that these satisfy the required properties. Take $(y,\xi) \in U_+$ such that $y \in \mc{M}_e, \varphi_t(y) \in \mc{M}_e$ and $t \geq T$. Write $t = k_1 T_1 + r_1$, with $k_1 \in \Z_{\geq 1}, r_1 \in [0,T_1)$, and $(k_1-1)T_1 = k_0 T_0 + r_0$, with $k_0 \in \Z_{\geq 0}, r_0 \in [0,T_0)$, that is
\[
t = k_0T_0 + T_1 + r_1 + r_0.
\]
Note that $T_1 + r_1 + r_0 < 3T_1 = T$.

For all $s \in [0, k_0 T_0]$, one has $\varphi_s^X(y,\xi) \in \pi(U_+^{(1)})$ and $(y,\xi) \in U_+^{(1)}$. Indeed, otherwise, we would get for some $s_\star \in [0,k_0 T_0]$ that $\varphi_{s_\star}^X(y,\xi) \in \pi(U_+^{(2)}) \setminus \pi(U_+^{(1)})$ but then $\varphi_{s_\star+T_1}^X(y) \notin \mc{M}_e$, which contradicts the fact that $\varphi_t^X(y) \in \mc{M}_e$ since
\[
s_\star+T_1 \leq (k_1-1)T_1 + T_1 = kT_1 \leq t.
\]
Then, using the uniform lower bound $|e^{(T_1+r_0+r_1)\X}(y,\xi)| \geq C_0 |\xi|$, we obtain:
\[
|e^{t\X}(y,\xi)| = |e^{(T_1 + r_0+r_1)\X} (e^{T_0 \X})^{k_0}(y,\xi)| \geq C_0 5^{k_0} |\xi| \geq C e^{\lambda t}|\xi|,
\]
for some constant $C > 0$ and $\lambda = \log(5)/T_0$. 
\end{proof}

We now let $V_+$ be a small conic neighborhood of $\cup_{X \sim X_0} (E^X_+)^*$ contained inside $U_+$, i.e. $V_+ \Subset U_+$. It will be convenient to use the following operation on the category of fibered conic subsets: if $V \subset T^*\mc{N}$ is an open conic subset, define the \emph{fiberwise complement} of $V$ as:
\[
\label{equation:fiber-complement}
V^{\complement_{\mathrm{fiber}}} := \left\{ (y,\xi) \in T^*\mc{N} ~|~ y \in \pi(V), \xi \in \overline{V}^{\complement} \cap T^*_y \mc{N} \right\},
\]
where the superscript $\complement$ denotes the set theoretic complement.

\begin{lemma}
\label{lemma:contraction}
There exists $\delta > 0$, $T > 0$, and $V_- := (W_-)^{\complement_{\mathrm{fiber}}}$, where $W_-$ 
is a small conic neighborhood of $\cup_{X \sim X_0} (E_-^{X})^* \oplus (E_0^{X})^*$, such that for all $X$ with $\|X-X_0\|_{C^{2}(\mc{M},T\mc{M})} \leq \delta$, one has $e^{T \X}V_- \Subset V_+$.
\end{lemma}

The same lemma can be proved by reversing the direction of $X$, i.e. by swapping the role of $E_-^*$ and $E_+^*$.

\begin{proof}
We fix an arbitrary open conic set $\widetilde{V}_-$ near $\pi^{-1}(K^{X_0})$ such that $\widetilde{V}_- \cap \left((E^{X_0}_-)^* \oplus (E_0^{X_0})^*\right) = \emptyset$. In restriction to $\pi^{-1}(K^{X_0})$, hyperbolicity ensures the existence of a time $T > 0$ such that
\[
e^{T \X_0}\left(\widetilde{V}_- \cap \pi^{-1}(K^{X_0})\right) \Subset V_+ \cap \pi^{-1}(K^{X_0}).
\]
By continuity, this also holds for an open conic neighborhood $V_-$ by taking $\pi(V_-)$ to be contained inside a small neighborhood of $K^{X_0}$ (whose size depends on $T$) and it also holds uniformly for all vector fields $X$ such that $\|X- X_0\|_{C^2} \leq \delta$, if $\delta > 0$ is taken small enough (depending on $T$) by using the stability result of Proposition \ref{theorem:stability} and choosing $\delta>0$ small enough so that 
$\cup_{X\sim X_0}K^X\subset \pi(V_-)$.
\end{proof}

In order to simplify notation, we will write $\zeta = (y,\xi)$ for a point in $T^*\mc{N}$, and $p_X(x, \xi) := \xi(X)$ for the principal symbol of $-iX$. From Lemmas \ref{lemma:expansion} and \ref{lemma:contraction}, we deduce:

\begin{lemma}
\label{lemma:time}
Let $\Omega$ be a small conic neighborhood of $\cup_{X \sim X_0} \left\{p_X = 0\right\}$ in $T^*\mc{M}_e$. There exist $\delta > 0, T > 0$ such that the following holds: for all $X$ with $\|X-X_0\|_{C^{2}(\mc{M},T\mc{M})} \leq \delta$, $t \geq T$, if $\zeta, e^{t \X}(\zeta) \in \Omega \cap T^*\mc{M}_e \setminus \left\{0\right\}$, then:
\[
\int_0^t \mathbf{1}_{U_+\sqcup~ U_-}(e^{s\X}(\zeta)) \dd s \geq t-T.
\]
\end{lemma}

In other words, the flowline of $\zeta$ spends at least a time $t-T$ in $U_+ \sqcup~ U_-$, where there is some uniform contraction/expansion. 

\begin{proof}
We use the sets $U_\pm,V_\pm$ defined in Lemmas \ref{lemma:expansion} and \ref{lemma:contraction}. Note that by construction, $\pi(V_\pm) \subset \pi(U_\pm)$ and we set $\mc{U} :=  \pi(U_+) \cap \pi(U_-)$.
We introduce the following constants:

\begin{enumerate}
\item Let $T_0 > 0$ be the time provided by Lemma \ref{lemma:close} applied with the open neighborhood $\mc{U}$ of $K^{X_0}$ and such that for all $X$ with $\|X-X_0\|_{C^2} \leq \delta$, for all $t \geq T_0$, $y \in \mc{M}_e$ such that $\varphi^X_t(y) \in \mc{M}_e$, one has
\[
\left\{\varphi_s^X(y) ~|~ s \in [T_0, t-T_0]\right\} \subset \mc{U}.
\]
\item Let $T_1 > 0$ be the time provided by Lemma \ref{lemma:expansion}.
\item Let $T_2 > 0$ be the time provided by Lemma \ref{lemma:contraction} such that $e^{T_2 \X}V_- \Subset V_+$.
\end{enumerate}

Take a point $\zeta \in \Omega  \cap  T^*\mc{M}_e\setminus \left\{0 \right\}$ such that $e^{t \X}(\zeta) \in T^*\mc{M}_e$ for some $t \geq 2 T_0$, that is, $\varphi_s^X(\pi(\zeta)) \in \mc{U}$ for all $s \in [T_0, t-T_0]$. We treat different cases: \\

\emph{Case 1:}  Assume that $e^{T_0 \X}(\zeta) \in U_-$. If $e^{s\X}(\zeta) \in U_-$ for all $s \in [T_0,t-T_0]$, then the claim holds for $\zeta$ and $T = 2T_0$. If not, there is a time $s_\star \in [T_0,t-T_0]$ such that $e^{s_\star \X}(\zeta) \in V_-$ \and $e^{s\X}(\zeta)\in U_-$ if $s\in [T_0,s_\star]$. By Lemma \ref{lemma:contraction}, we then deduce that $\zeta' := e^{(s_\star+T_2)\X}(\zeta) \in V_+ \Subset U_+$. Observe that $\zeta' \in U_+$ and $e^{(t-(s_\star+T_2))\X}(\zeta') \in T^*\mc{M}_e$. If $t-(s_\star+T_2) \geq T_1$, from Lemma \ref{lemma:expansion} we deduce that for all $s \in [T_0,s_\star] \cup [s_\star+T_2, t-T_1]$, we have $e^{s\X}(\zeta) \in U_- \cup U_+$, that is, the flowline of $\zeta$ spends at least $t-(T_0+T_1+T_2)$ in $U_- \cup U_+$. Thus, the claim holds with $T := T_0+T_1+T_2$. If $t-(s_\star+T_2) \leq T_1$, then the flowline of $\zeta$ has spent a time at least $s_\star - T_0 \geq t-(T_0+T_1+T_2)$ in $U_-$ and the claim holds with the same time $T$ defined previously. \\

\emph{Case 2:} Eventually, if $e^{T_0\X}(\zeta) \notin U_-$, then $e^{T_0 \X}(\zeta) \in V_-$ and the claim is also straightforward, following the previous arguments.
\end{proof}

Eventually, we will need:

\begin{lemma}
\label{lemma:technique0}
Let $W_- = W_-' \cap (W_-'')^{\complement_{\mathrm{fiber}}}$, where $W_-'$ and $W_-''$ are conic neighborhoods of $\pi^{-1}(K^{X_0})$ and $(E_+^{X_0})^*$, respectively. Let $W_+$ be a small conic neighborhood of $ (E_+^{X_0})^*$. Then, there exists $T > 0$ such that for all $t \geq T$, $e^{-t\X_0}W_- \cap W_+ = \emptyset$.
\end{lemma}

By \emph{small} for $W_+$, it is understood that $W_+ \cap \left((E_0^{X_0})^* \oplus (E_-^{X_0})^* \right)= \emptyset$.

\begin{proof}
This follows from the fact that there is a uniform time $T>0$ so that, for each $(y,\xi)\in W_-$ 
either $\rho(\varphi_{-t}^{X_0}(y))<0$ for all $t>T$ or $e^{-t{\bf X}_0}(y,\xi)$ belongs to a small conic neighborhood of $(E_0^{X_0})^*\oplus (E_-^{X_0})^*$ 
for all $t>T$, by the same argument as in Lemma \ref{lemma:contraction}.
\end{proof}

\subsubsection{Construction of $m_\pm$.} In this paragraph, we construct the functions $m_\pm$ involved in the expression \eqref{equation:m} of the escape function $m$.  
We introduce a smooth function $m_0\in C^\infty(S^*\mc{N},[0,1])$, invariant by the antipodal map $(x,\xi) \mapsto (x,-\xi)$, such that $m_0 = 1$ in a small neighborhood of $\kappa((E^{X_0}_u)^*)$ over $K^{X_0}$ and $m_0 = 0$ on the complement of a slightly larger neighborhood of $\kappa ((E_u^{X_0})^*)$.
We will need the following lemma:

\begin{lemma}
\label{lemma:technique}
For all $T > 0$ large enough, the following holds:
\[
\left\{ \begin{array}{l} \zeta, e^{T\X_0}(\zeta) \in S^*\mc{M}_e \\ m_0(\zeta) < 1 \end{array} \right. \implies \left(\forall t \in [T,3T], m_0(e^{-t\X_0}(\zeta)) = 0 \right).
\]
\end{lemma}

\begin{proof}
We argue by contradiction. Assume that
\begin{itemize}
\item there exists an increasing sequence of values $(T_j)_{j \in \Z_{\geq 0}}$ such that $T_j \to +\infty$,
\item a sequence of points $(\zeta_j)_{j \in \Z_{\geq 0}}$ such that $\zeta_j, e^{T_j\X_0}(\zeta_j) \in S^*\mc{M}_e$, $m_0(\zeta_j) < 1$,
\item and a sequence of values $(S_j)_{j \in \Z_{\geq 0}}$ such that $S_j \geq T_j$ and $m_0(e^{-S_j \X_0}(\zeta_j)) > 0$.
\end{itemize}
By compactness of $S^*\mc{M}_e$, up to extraction, we can always assume that $\zeta_j \to \zeta_\infty$. Observe that $\zeta_\infty \in \pi^{-1}(K^{X_0})$ as $T_j \to +\infty$: indeed, since $T_j \to \infty$, we have that $\zeta_\infty \in \pi^{-1}(\Gamma_-^{X_0})$; if $\zeta_\infty \in \pi^{-1}(\Gamma_-^{X_0} \setminus K^{X_0})$, the exit time from $\mc{M}$ in the past of $\zeta_\infty$ is finite and since $S_j \to +\infty$, $m_0(e^{-S_j \X_0} \zeta_j) > 0$ and $m_0$ vanishes outside of $\mc{M}$, we would get a contradiction for $j \geq 0$ large enough.

Since $m_0(\zeta_j) < 1$ and $m_0=1$ near $\kappa((E_u^{X_0})^*)$, we can find $V_-$, a small neighborhood of $\pi^{-1}(K^{X_0})$ whose closure is not intersecting $(E_-^{X_0})^*$ and such that $\zeta_\infty \in V_-$. Let $V_+$ be a small neighborhood of $\supp(m_0)$. By Lemma \ref{lemma:technique0}, there is $T > 0$ such that for all $t \geq T$, $e^{-t\X_0}V_- \cap V_+ = \emptyset$. In particular, for $j \geq 0$ large enough, $\zeta_j \in V_-$ and thus $e^{-S_j\X_0}(\zeta_j) \notin V_+$, that is $m_0(e^{-S_j\X_0}(\zeta_j)) = 0$. But this contradicts $m_0(e^{-S_j\X_0}(\zeta_j)) > 0$.
\end{proof}

We then set for $T > 0$ large enough satisfying Lemma \ref{lemma:technique}:
\begin{equation}
\label{equation:m1}
m_1(\zeta) := \dfrac{1}{2T} \int_T^{3T} m_0(e^{-t \X_0}(\zeta)) \dd t.
\end{equation}

\begin{lemma}
\label{lemma:m1}
The function $m_1 \in C^\infty(S^*\mc{N}, [0,1])$ satisfies the following properties:
\begin{enumerate}
\item $m_1 = 1$ near $(E_+^{X_0})^* \cap \pi^{-1}(\mc{M}_e)$;
\item $\supp(m_1) \subset \pi^{-1}(\Sigma_+)$ and $\supp(m_1)$ is contained in a small neighborhood of $(E_+^{X_0})^*$;
\item $\X_0m_1 \geq 0$ on $\pi^{-1}(\mc{M}_e)$;
\item There exist $\eps_0, \delta_0 > 0$ such that: if $\zeta \in \pi^{-1}(\mc{M}_e)$ and $|m_1(\zeta)-1/2| \leq \eps_0$, then $\X_0m_1(\zeta) \geq \delta_0$.
\end{enumerate}
\end{lemma}

\begin{proof}
We prove each point separately.

(1,2) Taking $T > 0$ large enough in \eqref{equation:m1}, the first two items are immediate to check. \\

(3) For $\zeta \in T^*\mc{M}_e$, we have:
\[
\X_0  m_1(\zeta) = \dfrac{1}{2T} \left( m_0(e^{-T\X_0}(\zeta)) - m_0(e^{-3T\X_0}(\zeta)) \right),
\]
and we want to show that $\X_0 m_1 \geq 0$ on $\pi^{-1}(\mc{M}_e)$. Observe that if $m_0(e^{-T\X_0}(\zeta)) = 1$, then the claim $\X_0 m_1(\zeta) \geq 0$ is immediate. We can thus assume that $m_0(e^{-T\X_0}(\zeta)) < 1$. If $e^{-T\X_0}(\zeta) \notin \pi^{-1}(\mc{M}_e)$, then $m_0(e^{-T\X_0}(\zeta)) = 0$ and by convexity, $m_0(e^{-3T\X_0}(\zeta)) = 0$ and $\X_0 m_1(\zeta) = 0$. If $e^{-T\X_0}(\zeta) \in \pi^{-1}(\mc{M}_e)$, we can apply Lemma \ref{lemma:technique} which implies that $m_0(e^{-3T \X_0}(\zeta)) = 0$ and thus we also obtain $\X_0 m_1(\zeta) \geq 0$. \\

(4) In order to show the last item, it suffices to show that on the compact set 
\[
\left\{\X_0 m_1 = 0\right\} \cap \pi^{-1}(\mc{M}_e),
\]
one has $|m_1-1/2| \geq \eps_1$, for some positive $\eps_1 > 0$, that is, the continuous function $|m_1-1/2|$ does not vanish on this set. Let $\zeta \in \pi^{-1}(\mc{M}_e)$ be such that $\X_0 m_1(\zeta) = 0$. Then $m_0(e^{-T\X_0}\zeta) = m_0(e^{-3T\X_0}\zeta)$.

Assume that $m_0(e^{-T\X_0}\zeta) < 1$. If $e^{-T\X_0}\zeta \notin \pi^{-1}(\mc{M}_e)$, then by convexity of $\mc{M}_e$, $e^{-t \X_0}(\zeta) \notin \pi^{-1}(\mc{M}_e)$ for all $t \geq T$ and thus $m_1(\zeta) = 0$, that is $|m_1-1/2| = 1/2 \neq 0$. We can thus assume that $e^{-T \X_0}(\zeta) \in \pi^{-1}(\mc{M}_e)$. By Lemma \ref{lemma:technique}, we get that $m_0(e^{-3T \X_0}(\zeta)) = 0 = m_0(e^{-T \X_0}(\zeta))$. Lemma \ref{lemma:technique} also gives us that $m_0(e^{-t \X}(\zeta)) = 0$ for all $t \in [2T,3T]$. As a consequence:
\[
m_1(\zeta) = \dfrac{1}{2T} \int_T^{3T} m_0(e^{-t\X_0}\zeta) \dd t = \dfrac{1}{2T} \int_T^{2T} m_0(e^{-t \X_0}\zeta) \dd t < 1/2,
\]
so $|m_1(\zeta)-1/2| \neq 0$.

We now assume that $m_0(e^{-T\X_0}(\zeta))=1 = m_0(e^{-3T\X_0}(\zeta))$. We claim that $m_0(e^{-t\X_0}\zeta)=1$ for all $t \in [T,2T]$. Indeed, assume that there exists some $t_0 \in [T,2T]$ such that $\zeta_0 := e^{-t_0\X_0}(\zeta)$ satisfies $m_0(\zeta_0) < 1$. By Lemma \ref{lemma:technique}, since $\zeta_0, e^{T \X_0}(\zeta_0) \in S^*\mc{M}_e$, we obtain that $m_0(e^{-t \X_0}(\zeta_0)) = 0$ for all $t \geq T$. Taking $t_1 := 3T-t_0 \geq T$, we deduce that
\[
m_0(e^{-t_1 \X_0}(\zeta_0)) = 0 = m_0(e^{-(3T-t_0)\X_0} e^{-t_0\X_0}(\zeta)) = m_0(e^{-3T\X_0}(\zeta)),
\]
which is a contradiction. We then deduce that
\[
m_1(\zeta) > \dfrac{1}{2T} \int_T^{2T} m_0(e^{-t \X_0}(\zeta)) \dd t = 1/2,
\]
that is $|m_1(\zeta)-1/2| \neq 0$. This eventually proves the fourth item.
\end{proof}

We now introduce:
\begin{equation}
\label{equation:m+}
m_+ := \chi(m_1) \in C^\infty(S^*\mc{N},[0,1]),
\end{equation}
where $\chi \in C^\infty(\R)$ is a smooth cutoff function such that: $\chi' \geq 0$, $\chi = 0$ on $(-\infty,-1/2-\eps_0]$, and $\chi=1$ on $[1/2+\eps_0,+\infty)$, where $\eps_0 > 0$ is the constant provided by Lemma \ref{lemma:m1}. By construction, this function takes value $1$ near $(E_+^{X_0})^*$. By the same process, one can also construct a function $m_- \in C^\infty(S^*\mc{N},[0,1])$ such that $m_- = 1$ near $(E_-^{X_0})^*$.

\begin{lemma}
\label{lemma:m-plus}
There exists $\delta > 0$ small enough such that for all smooth vector fields $X$ with $\|X-X_0\|_{C^2(\mc{M},T\mc{M})} < \delta$, the functions $m_\pm \in C^\infty(S^*\mc{N}, [0,1])$ satisfy the following properties:
\begin{enumerate}
\item $m_\pm = 1$ near $(E_\pm^{X})^* \cap \pi^{-1}(\mc{M}_e)$;
\item $\supp(m_\pm) \subset \pi^{-1}(\Sigma_\pm)$ and $\supp(m_\pm)$ is contained in a small neighborhood of $(E_\pm^{X})^*$;
\item There exists $\delta_1 > 0$ small such that
\begin{align}
\supp(m_\pm) \subset \pi^{-1}\left(\left\{\rho>-(1-\delta_1)\rho_0\right\} \right), \\
\supp(m_\pm) \cap \pi^{-1}\left( \mc{M}^{\complement} \right) \subset \left\{ \pm \widetilde{X}_0\rho < -\delta_1\right\};
\end{align}
\item $\pm \X m_\pm \geq 0$ on $\pi^{-1}(\mc{M}_e)$.
\end{enumerate}
\end{lemma}

We will argue on $m_+$ as the proof is similar for $m_-$.

\begin{proof}
We prove each item individually.

(1,2,3) are straightforward to check with $\delta_1 > 0$ small enough. The fact that $X$ and $X_0$ are $C^2$-close implies by the structural stability Proposition \ref{theorem:stability} that $\cup_{X \sim X_0} (E_\pm^X)^*$ are contained in a small neighborhood of $(E_\pm^{X_0})^*$ where $m_\pm = 1$. \\

(4) Observe that
\[
\X m_+ = \X m_1~ \chi'(m_1) = \left((\X-\X_0) m_1 + \X_0 m_1\right)\chi'(m_1).
\]
The nonnegative function $\chi'(m_1) \geq 0$ vanishes everywhere, except on the set $\left\{|m_1-1/2| \leq \eps_0\right\}$. Observe that on $\left\{|m_1-1/2| \leq \eps_0\right\}$, we have by Lemma \ref{lemma:m1} that:
\[
(\X-\X_0) m_1 + \X_0 m_1 \geq \delta_0 - \|X-X_0\|_{C^0}\|m_1\|_{C^1} \geq \delta_0/2,
\]
provided $\delta \leq \delta_0/(2\|m_1\|_{C^1})$. As a consequence, we deduce that $\X m_+ \geq 0$ on $\pi^{-1}(\mc{M}_e)$.
\end{proof}

\subsubsection{Construction of the bump functions $\chi_\pm$}  In this paragraph, we construct the bump functions $\chi_\pm$ involved in the expression \eqref{equation:m} of the escape function $m$. 

\begin{lemma}
\label{lemma:chi-plus}
There exist $\delta_1,\delta > 0$ small enough, and cutoff functions $\chi_\pm \in C^\infty(\mc{N},[0,1])$ such that for all smooth vector fields $X$ such that $\|X-X_0\|_{C^1(\mc{M},T\mc{M})} < \delta$, the following holds:
\begin{enumerate}
\item $\supp(\chi_\pm) \subset \left\{-2 \rho_0 < \rho < -\delta_1\right\} \cap \left\{ \pm \widetilde{X}_0 \rho < -\delta_1 \right\}$,
\item $X\chi_\pm \geq 0$,
\item $X \chi_\pm > \delta^3_1\rho_0/2$ on $\left(\left\{-(1-\delta_1) \rho_0 < \rho < 0\right\} \cap \left\{ \pm \widetilde{X}_0 \rho < -\delta_1 \right\}\right) \setminus \mc{M}_e$.
\end{enumerate}
\end{lemma}

\begin{proof}
We only deal with $\chi_+$, the proof being similar for $\chi_-$. First of all, for $j=1,2$, we define functions $\chi_{j} \in C^\infty(\R)$ depending on some parameter $\delta_1 > 0$ which will be chosen small enough in the end. The function $\chi_1 \in C^\infty_{\comp}(\R)$ is defined such that (see Figure \ref{figure:chi1}):
\begin{itemize}
\item $\supp(\chi_1) \subset \left\{-2\rho_0 < \rho < -\delta_1\right\}$;
\item $\chi_1 \geq 0, \chi_1(-\rho_0) = 1, \chi_1'(-\rho_0)=0$;
\item $\chi_1' \geq 0$ on $\left\{-2\rho_0 < \rho < -\rho_0\right\}$, $\chi_1' \leq 0$ on $\left\{-\rho_0 < \rho < -\delta_1\right\}$;
\item $\chi'_1 \leq - \delta_1$ on $\left\{-\rho_0(1-\delta_1) \leq \rho \leq -2\delta_1\right\}$.
\end{itemize}
The function $\chi_2 \in C^\infty(\R)$ is defined such that:
\begin{itemize}
\item $\supp(\chi_2) \subset (-\infty,-\delta_1]$;
\item $\chi_2 \geq 0$;
\item $\chi_2 = 1$ on $(-\infty,-2\delta_1]$.
\end{itemize}

\begin{center}
\begin{figure}[htbp!]
\includegraphics[scale=0.8]{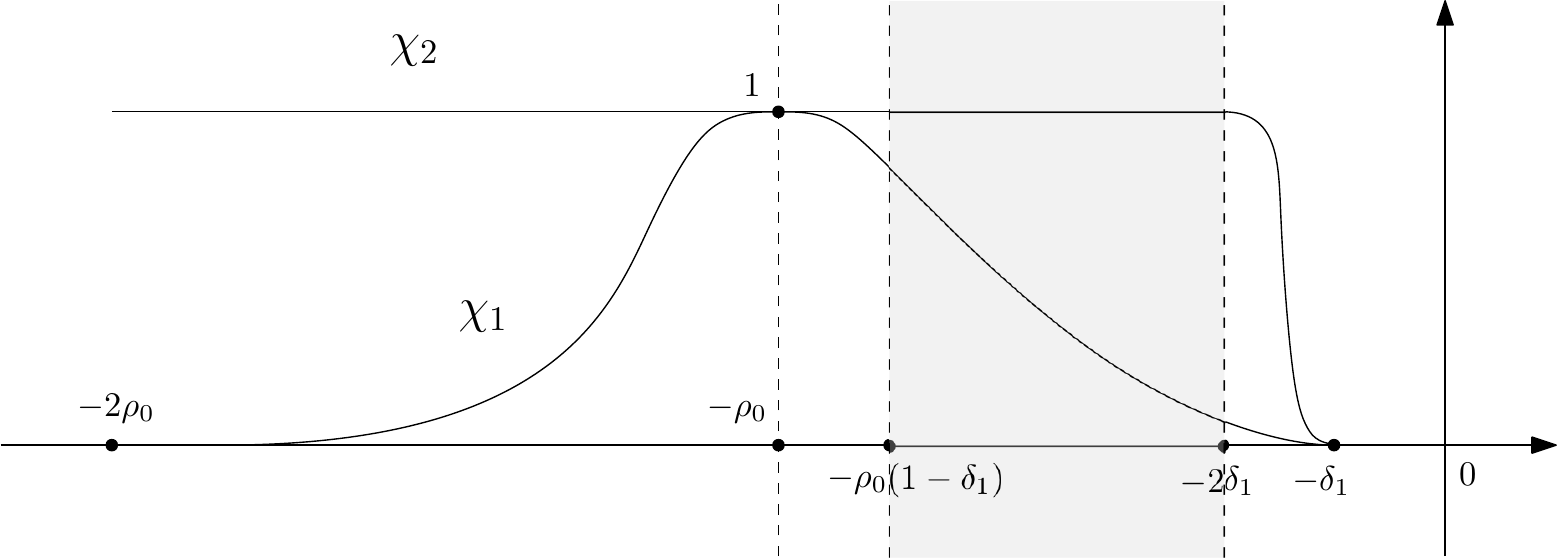}
\caption{The cutoff functions $\chi_1$ and $\chi_2$.}
\label{figure:chi1}
\end{figure}
\end{center}

We then set
\begin{equation}
\label{equation:chi-plus}
\chi_+ := \chi_1(\rho) \chi_2(\widetilde{X}_0 \rho),
\end{equation}
and we claim that it satisfies the required properties. Recall from \S \ref{sssection:geom-cons} that $X = \psi \widetilde{X}$, where $\widetilde{X}$ is some smooth extension of the vector field $X$, initially defined on $\mc{M}$ to the closed manifold $\mc{N}$.

We have:
\begin{align}
\label{equation:derivee}
X \chi_+ & = X \rho \chi_1'(\rho) \chi_2(\widetilde{X}_0\rho) + (X \widetilde{X}_0\rho) \chi_1(\rho) \chi_2'(\widetilde{X}_0\rho) \\
& = \psi \cdot (\widetilde{X}\rho) \chi_1'(\rho) \chi_2(\widetilde{X}_0\rho) + \psi \cdot (\widetilde{X}_0^2\rho) \chi_1(\rho)\chi_2'(\widetilde{X}_0\rho)+ \psi \cdot \left((\widetilde{X}-\widetilde{X}_0)\widetilde{X}_0\rho\right) \chi_1(\rho) \chi_2'(\widetilde{X}_0\rho)\nonumber,
\end{align}
where we writeWe study separately the three terms of \eqref{equation:derivee}. \\

We study the first term in \eqref{equation:derivee}. On $\supp(\chi_2(\widetilde{X}_0\rho))$, one has $\widetilde{X}_0\rho \leq -\delta_1$. Thus, assuming $\|X-X_0\|_{C^0(\mc{M},T\mc{M})} < \delta$ is small enough (depending on $\delta_1$), we obtain that $\widetilde{X}\rho \leq -\delta_1/2$ on $\supp(\chi_2(\widetilde{X}_0\rho))$. As a consequence, we obtain (note that  $\psi\chi_1'\leq 0$):
\[
 \psi \cdot (\widetilde{X}\rho) \chi_1'(\rho) \chi_2(\widetilde{X}_0\rho) \geq -\frac{\delta_1\psi}{2} \chi_1'(\rho) \chi_2(\widetilde{X}_0\rho) \geq 0.
\]

Moreover, on the set $\left(\left\{-(1-\delta_1) \rho_0 < \rho < -2\delta_1\right\} \cap \left\{ \widetilde{X}_0 \rho < -\delta_1 \right\}\right)$, using that 
$\psi = \rho+\rho_0$ near $\left\{\rho=-\rho_0\right\}$ (so $\psi \geq \delta_1 \rho_0$ on the former set), and that $\chi_1'(\rho) \leq -\delta_1$, we obtain that this can be bounded from below by:
\begin{equation}
\label{equation:bb}
\psi \cdot (\widetilde{X}\rho) \chi_1'(\rho) \chi_2(\widetilde{X}_0\rho) \geq \frac{\delta^2_1\psi}{2} \geq \frac{\delta_1^3 \rho_0}{2} > 0.
\end{equation}

We now deal with the second and third term. The strict convexity property of the level sets $\left\{\rho = c\right\}$ (for $c \in [-2\rho_0,0]$) with respect to $\widetilde{X}_0$ reads: $\widetilde{X}_0\rho = 0 \implies \widetilde{X}^2_0\rho < 0$. Since $\left\{\widetilde{X}_0\rho = 0\right\} \cap \left\{-2\rho_0 \leq \rho \leq 0\right\}$ is compact, we deduce that there exists $\delta_1 > 0$ small enough such that, on the set $\left\{|\widetilde{X}_0\rho| \leq 2\delta_1\right\}$, one has $\widetilde{X}_0^2\rho\leq-c< 0$ for some constant $c = c(\delta_1) > 0$. Using that $\mathrm{supp}(\chi_2'(\widetilde{X}_0\rho))$ has support in $\left\{|\widetilde{X}_0\rho| \leq 2 \delta_1\right\}$ and assuming $\|X-X_0\|_{C^0(\mc{M},T\mc{M})} \leq \delta$, we obtain the existence of some constant $C > 0$ (depending on $\delta_1$ but independent of $\delta > 0$) such that:
\[
\begin{split}
 \psi \cdot (\widetilde{X}_0^2\rho) \chi_1(\rho)\chi_2'(\widetilde{X}_0\rho)  + \psi \cdot \left((\widetilde{X}-\widetilde{X}_0)\widetilde{X}_0\rho\right) \chi_1(\rho) \chi_2'(\widetilde{X}_0\rho)  \geq (C\delta- c) \psi \chi_1(\rho) \chi_2'(\widetilde{X}_0\rho).
 \end{split}
\]
Taking $\delta \leq c/(2C)$ small enough (depending on $\delta_1 > 0$), we obtain that this last term is non-negative. 

Overall, we have thus proved (1) and (2), and (3) directly follows from (2) together with \eqref{equation:bb}, since we can take $\delta_1 > 0$ small enough so that $\left\{\rho \geq -2\delta_1\right\} \subset \mc{M}_e$. \end{proof}

\subsubsection{Piecing together the functions} The various sets appearing in the previous constructions and the functions $m_\pm, \chi_\pm$ can be seen in Figure \ref{figure:construction}. We now piece together the previous constructions and prove Proposition \ref{proposition:escape}.

\begin{proof}[Proof of Proposition \ref{proposition:escape}]
Define $m$ by \eqref{equation:m}, where $m_\pm$ and $\chi_\pm$ are provided by Lemmas \ref{lemma:m-plus} and \ref{lemma:chi-plus}, and the constant $\delta_1 > 0$ is chosen small enough so that both Lemmas \ref{lemma:m-plus}, \ref{lemma:chi-plus} hold.

Since $\chi_\pm$ have support outside of $\mc{M}$, $m_\pm=1$ near $(E_\pm^{X})^* \cap \pi^{-1}(\mc{M})$ and $m = m_- - m_+$ on $\pi^{-1}(\mc{M})$, we get that the points (1,2,3) are verified. The fact that $\supp(m) \subset \left\{\rho > -2\rho_0\right\}$ is also straightforward by Lemmas \ref{lemma:m-plus}, \ref{lemma:chi-plus}, which proves (4). Eventually, (5) is also immediate to verify.

We now show that (6) holds if we take $\eta > 0$ small enough. By Lemmas \ref{lemma:m-plus}, item (4), and \ref{lemma:chi-plus}, item (2), the condition $\X m \leq 0$ holds on $\pi^{-1}(\mc{M}_e)$. On the set $\left\{\rho \leq -\rho_0(1-\delta_1)\right\}$, we have $m_\pm = 0$ and thus by Lemma \ref{lemma:chi-plus}, the inequality $\X m \leq 0$ also holds. It remains to check the inequality on $\left\{\rho \geq -\rho_0(1-\delta_1)\right\} \cap (\mc{M}_e)^\complement$. But there, we have by Lemma \ref{lemma:chi-plus}, (3):
\[
\begin{split}
\X m & = \X m_- - \X m_+ + \eta^{-1}(\pi^* X\chi_--\pi^*X\chi_+) \\
& \leq \|m_-\|_{C^1}+\|m_+\|_{C^1} - \eta^{-1} \frac{\delta_1^3 \rho_0}{2} \leq 0,
\end{split}
\]
if $\eta > 0$ is chosen small enough.
\end{proof}

\subsection{Meromorphic extension of the resolvent}

\label{ssection:meromorphic}

We now study the meromorphic extension of the resolvent on anisotropic Sobolev spaces, and its dependence with respect to the vector field $X$. This is the main difference with \cite{Dyatlov-Guillarmou-16}. We will be particularly interested by the resolvent at $z=0$, namely $R_g$, for our application.

\subsubsection{Global resolvent on uniform anisotropic Sobolev spaces}

In the following, we assume that an arbitrary metric $h$ was chosen on $T\mc{N} \to \mc{N}$. This induces a metric $h^\sharp$ on $T^*\mc{N} \to \mc{N}$ and for $(y,\xi) \in T^*\mc{N}$, we will write $\langle \xi \rangle := (1+h^\sharp_y(\xi,\xi))^{1/2}$ (the $y$ is dropped from the Japanese bracket notation in order to avoid repetition). For $\varrho \in (1/2,1]$, we denote by $S_{\varrho}^k(T^*\mc{N})$ (resp. $\Psi^k_\varrho (\mc{N})$) the Fréchet space of symbols of order $k$, i.e. $a\in S^k(T^*\mc{N})$ if in local coordinates 
\[ \forall \alpha,\beta, \exists C>0, \quad |\pl_\xi^\alpha\pl_x^\beta a(y,\xi)|\leq C\cjg \xi\cjd^{k-\varrho|\alpha|+(1-\varrho)|\beta|}\]
 (resp. the space of pseudodifferential operators of order $k$ obtained by quantization of symbols in $S^k_\varrho(T^*\mc{N})$). We shall remove the $\varrho$ index from the notation when $\varrho=1$. Note that $k$ can be a real number, but also a variable \emph{order function}, see \cite[Appendix A]{Faure-Roy-Sjostrand-08} for further details. 

The function $m \in C^\infty(S^*\mc{N},[-1,1])$ constructed in \S\ref{ssection:escape} yields a smooth, $0$-homogeneous function $m \in C^\infty(T^*\mc{N} \setminus \left\{0\right\},[-1,1])$, still denoted by $m$ which decreases along all flow lines of $\X$, the Hamiltonian vector field induced by $X$ (and $X$ is close to $X_0$). We can always modify $m$ in a small neighborhood of the $0$-section in $T^*\mc{N}$ to obtain a new function — still denoted by the same letter $m$ to avoid unnecessary notation — such that $m \in C^\infty(T^*\mc{N},[-1,1])$ and $\X m (y,\xi) \leq 0$ for all $(y,\xi) \in T^*\mc{N}$ such that $\langle \xi \rangle > 1$. \\

Define a \emph{regularity pair} as a pair of indices $\mathbf{r} := (r_\bot,r_0)$, where $r_\bot > r_0 \geq 0$. Given such a regularity pair $\textbf{r}$, we introduce (for all $\eps>0$ small enough):
\begin{equation}
A_{\mathbf{r}} := \Op\left(\langle \xi \rangle^{(r_\bot m(y,\xi)-r_0)/2}\right)^* \Op\left(\langle \xi \rangle^{(r_\bot m(y,\xi)-r_0)/2}\right) \in \Psi_{1-\eps}^{r_\bot m - r_0}(\mc{N}).
\end{equation}
This is an elliptic and formally selfadjoint pseudodifferential operator belonging to an \emph{anisotropic class}, see \cite[Appendix A]{Faure-Roy-Sjostrand-08} for further details. As a consequence, up to a modification by a finite-rank formally selfadjoint smoothing operator, we can assume that $A_{\mathbf{r}}$ is invertible. 

\begin{definition}
\label{definition:anisotropic}
We define the \emph{scale of anisotropic Sobolev spaces} with regularity $\mathbf{r} := (r_\bot,r_0), r_\bot > r_0 \geq 0$, as:
\[
\mc{H}^{\mathbf{r}}_\pm(\mc{N}) := A_{\mathbf{r}}^{\mp 1}(L^2(\mc{N})), \qquad \|f\|_{\mc{H}^{\mathbf{r}}_\pm(\mc{N})} := \|A_{\mathbf{r}}^{\pm 1} f\|_{L^2(\mc{N})}.
\]
\end{definition}

We formulate important remarks:

\begin{remark}
\begin{enumerate}
\item The spaces $\mc{H}^{\mathbf{r}}_\pm(\mc{N})$ are Hilbert spaces, equipped with the scalar product
\[
\langle \cdot, \cdot \rangle_{\mc{H}^{\mathbf{r}}_\pm(\mc{N})} := \langle A_{\mathbf{r}}^{\pm 1} \cdot, A_{\mathbf{r}}^{\pm 1} \cdot\rangle_{L^2(\mc{N})}.
\]

\item This scale of spaces is \emph{independent} of the vector field $X$, as long as it is close enough to $X_0$ in the $C^2$-topology, since the escape function $m$ is independent of the vector field. This will be important when studying the regularity of the meromorphic extension of the resolvent $z \mapsto R^X_\pm(z)$ (given by  \eqref{equation:rpm}) with respect to the vector field $X$.

\item Distributions in $\mc{H}^{\mathbf{r}}_+(\mc{N})$ are microlocally in $H^{r_\bot-r_0}(\mc{N})$ near $(E_-^X)^*$, $H^{-r_0}(\mc{N})$ near $(E_0^X)^*$, $H^{-r_\bot-r_0}(\mc{N})$ near $(E_+^X)^*$ (in the sense that after application of an $A\in \Psi^0(\mc{N})$ with wavefront set in the discussed region they have the announced regularity). The choice of regularity is arbitrary here and we did not try to optimize it. The only crucial point is that distributions in $\mc{H}^{\mathbf{r}}_+(\mc{N})$ have positive Sobolev regularity near $(E_-^X)^*$ while they have negative Sobolev regularity near $(E_+^X)^*$. 
\end{enumerate}
\end{remark}

We let $q \in C^\infty(\mc{N}, [0,1])$ be a smooth cutoff function such that:
\begin{itemize}
\item $\mathrm{supp}(q)$ is contained in the complement of a small open neighborhood of $\mc{M}$,
\item $q = 1$ on the complement of some slightly larger open neighborhood of $\mc{M}$,
\item the closure of the set $\left\{q <1\right\}$ is strictly convex with respect to all the vector fields $X$ for $\|X-X_0\|_{C^2} \leq \delta$ small enough.
\end{itemize}
Given a regularity pair $\mathbf{r} := (r_\bot,r_0)$ and a constant $\omega > 0$, we define for $X$ close enough to $X_0$ and $\Re(z) \gg 0$ large enough:
\begin{equation}
\label{equation:rpm}
R_{\mp}^X(z) := -\int_0^{+\infty} e^{-tz} e^{-\omega \int_0^t (\varphi_{\mp s}^X)^*q \dd s} e^{\mp tX} \dd t,
\end{equation}
Although we do not indicate it in the notation, $R_\mp^X(z)$ \emph{does} depend on a choice of $\omega$.
This satisfies the identity on $C^\infty(\mc{N})$:
\[
(\mp X-z- \omega q)R_{\mp}^X(z) = \mathbbm{1}_{\mc{N}}.
\]
The constant $\omega > 0$ will be fixed later. \\

The aim of this section is to study the meromorphic extension of the resolvent $z \mapsto R^X_+(z)$, for $X$ close to $X_0$, in the anisotropic Sobolev spaces of Definition \ref{definition:anisotropic}, and the dependence with respect to the vector field $X$.

\begin{theorem}
\label{theorem:meromorphic}
There exists $C_\star, \delta_\star,\Lambda > 0$ such that the following holds. For all $\delta \leq \delta_\star$, for all regularity pairs $\mathbf{r} = (r_\bot,r_0)$, there exists a choice of constant $\omega := \omega(\mathbf{r}) > 0$ large enough such that for all smooth vector fields $X$ on $\mc{M}$ such that $\|X-X_0\|_{C^2(\mc{M},T\mc{M})} \leq \delta$, the family
 \[
z \mapsto R_-^X(z) = (-X-z- \omega(\mathbf{r}) q)^{-1} \in \mc{L}(\mc{H}^{\mathbf{r}}_+),
\]
initially defined for $\Re(z) \gg 1$ by \eqref{equation:rpm} and holomorphic for $\Re(z) \gg 1$ large enough, extends to a meromorphic family of operators on the half-space $\left\{\Re(z) > -\Lambda (r_\bot-r_0) + C_\star \delta \right\}$. The same holds for $R_+^X(z)$ on the space $\mc{H}^{\mathbf{r}}_-$. 

Moreover, if $z_0 \in \left\{\Re(z) > -\Lambda (r_\bot-(r_0+2)) + C_\star \delta \right\}$ is not a pole of $z \mapsto R^{X_0}(z)$, then there exists $\eps_0 > 0$ such that the map
\[
C^\infty(\mc{N},T\mc{N}) \times D(z_0,\eps_0) \ni (X,z) \mapsto R^{X}(z) \in \mc{L}(\mc{H}^{(r_\perp,r_0)}_+,\mc{H}^{(r_\perp,r_0+2)}_+),
\]
is $C^2$-regular\footnote{Even though we only need $C^2$, our proof actually shows it is $C^k$ for all $k\in \mathbb{N}$.} with respect to $X$ and holomorphic in $z$, where $D(z_0,\eps_0) \subset \C$ is the disk centred at $z_0$, of radius $\eps_0$.
\end{theorem}

As usual, the poles do not depend on the choices made in the construction of the spaces. The rest of \S\ref{ssection:meromorphic} is devoted to the proof of Theorem \ref{theorem:meromorphic}. Theorem \ref{theorem:meromorphic} obviously implies Theorem \ref{axiomAsmooth} stated in the introduction, since the resolvent on $\M$ can be expressed in terms of the resolvent on $\mc{N}$ and the restriction to $\M$ (as in Lemma \ref{lemma:res} below in the analogous case of geodesic vector fields).

\subsubsection{Parametrix construction} Denote by $\mu$ a smooth measure on $\mc{N}$ which restricts to the Liouville measure on $\mc{M}$. Note that $X_0$ is volume-preserving on $\mc{M}$ and, up to minor modifications, we can also assume that the extension of $X_0$ to $\mc{N}$ is volume-preserving on $\mc{M}_e$ (but not on $\mc{N}$, since $X_0$ vanishes on $\left\{\rho=-\rho_0\right\}$). In order to shorten notation, we will write $L^2(\mc{N}) := L^2(\mc{N},\mu)$.

For $T > 0$, consider a smooth cutoff function $\chi_T \in C^\infty_{\comp}(\R_+)$, depending smoothly on $T$, such that $\chi_T = 1$ on $[0,T]$, $-2 \leq \chi_T' \leq 0$, and $\chi_T = 0$ on $[T+1,\infty)$. For $\Re(z) \gg 1, \omega \geq 1$, the following identity holds on $C^\infty(\mc{N})$:
\begin{equation}
\label{equation:identite}
\begin{split}
- \int_0^{+\infty} \chi_T(t) e^{-tz}  &e^{- \int_0^t (\varphi^X_{-s})^*(\omega q)~ \dd s}  e^{-tX} \dd t ~~~(-X-z- \omega q) \\
& = \mathbbm{1} + \int_0^{+\infty} \chi_T'(t) e^{-tz} e^{- \int_0^t (\varphi^X_{-s})^*(  \omega q) ~\dd s} e^{-tX} \dd t.
\end{split}
\end{equation}
We now \emph{fix} once for all a regularity pair $\mathbf{r} := (r_\bot,r_0)$, and set $r := r_0 + r_\bot$. The constant $\omega \geq 1$ will be chosen to depend on $\mathbf{r}$ later. We conjugate the equality \eqref{equation:identite} by $A_{\mathbf{r}}$. We obtain:
\begin{equation}
\label{equation:identite2}
\begin{split}
&- A_{\mathbf{r}} \int_0^{+\infty} \chi_T(t) e^{-tz}  e^{- \int_0^t (\varphi^X_{-s})^*( \omega q)~ \dd s}  e^{-tX} A_{\mathbf{r}}^{-1} \dd t ~~~A_{\mathbf{r}}(-X-z- \omega q)A_{\mathbf{r}}^{-1} \\
& = \mathbbm{1} + \int_0^{+\infty} \chi_T'(t) e^{-tz} e^{-tX} \underbrace{e^{tX}A_{\mathbf{r}} e^{- \int_0^t (\varphi^X_{-s})^*(\omega q) ~\dd s} A_{\mathbf{r}}^{-1}e^{-tX}}_{:=B_1^X(t)}\underbrace{e^{tX}A_{\mathbf{r}} e^{-tX} A_{\mathbf{r}}^{-1}}_{:=B_2^X(t)}~~\dd t. 
\end{split}
\end{equation}

Since the second term on the right-hand side of \eqref{equation:identite2} is defined as an integral over time in the flow direction $e^{-tX}$, it is smoothing outside $\left\{p_X=0\right\}$. We let $\Omega' \Subset \Omega$ be two open nested conic neighborhoods of $\left\{p_{X_0}=0\right\}$ in $T^*\mc{N} \cap \left\{\rho > -\rho_0\right\}$. Note that by continuity, these are also conic neighborhoods of $\left\{p_{X}=0\right\}$ for all $X\sim X_0$. We let $e \in S^0(T^*\mc{N})$ be a symbol of order $0$ such that $e=0$ outside $\Omega$ and $e=1$ on $\Omega'$ and we set $E := \Op(e)$. We then decompose the second term on the right-hand side of \eqref{equation:identite2} as:
\begin{equation}
\label{equation:identite3}
\int_0^{+\infty} \chi_T'(t) e^{-tz} e^{-tX} B_1^X(t) B_2^X(t) ~\dd t = \int_0^{+\infty} \chi_T'(t)  e^{-tz} e^{-tX}  E B^X_1(t) B_2^X(t) ~\dd t + K^X_1(T,z),
\end{equation}
where
\[
K^X_1(T,z) := \int_0^{+\infty} \chi_T'(t) e^{-tz} e^{-tX} (\mathbbm{1}-E) B_1^X(t) B_2^X(t) ~\dd t,
\]
and $K^X_1(T,z) \in \Psi^{-\infty}(\mc{N})$. In order to prove that $K_1^X(T,z)$ is smoothing, we can remark that $K^X_1(T,z)=E'K^X_1(T,z)$ for some $E'\in \Psi^0(\mc{N})$
with microsupport that does not intersect a conic neighborhood of  $\{p_X=0\}$, and then show that $X^kK^X_1(T,z)\in \mc{L}(L^2)$ for all $k\in \N$, using that 
$X^ke^{-tX}=(-\pl_t)^ke^{-tX}$ and integrating by part in $t$, and finally use that $E'(C-X^2)^{-1}\in \Psi^{-2}(\mc{N})$ for some $C\gg 1$ since $C-X^2$ is elliptic on the microsupport  of $E'$. 
The dependence of $K^X_1(T,z)$ on its parameters is holomorphic in $z \in \C$ and smooth in the variables $T \in \R, X \in C^\infty(\mc{M},T\mc{M})$. \\

Below, we use the notation $\mc{L}(\mc{H})$ to denote continuous linear operators on a Hilbert space $\mc{H}$ and $\mc{K}(\mc{H})$ for compact operators.

\begin{proposition}
\label{proposition:crucial-bound}
There exist $C_\star,\delta_\star,\Lambda > 0$ such that the following holds. For all regularity pairs $\mathbf{r}$, there exist $C(\mathbf{r}),\omega(\mathbf{r}) > 0$ such that for all smooth vector fields $\|X-X_0\|_{C^2} \leq \delta$ with $\delta \leq \delta_\star$, for all $t \geq 0$, there exist (Fourier Integral) operators $M^X(t) \in \mc{L}(L^2(\mc{N}))$ and $S^X(t) \in \mc{K}(L^2(\mc{N}))$ such that
\[
e^{-tX} E B_1^X(t) B_2^X(t) = M^X(t) + S^X(t),
\]
and
\[
\|M^X(t)\|_{L^2(\mc{N})} \leq C(\mathbf{r}) e^{\left(-\Lambda(r_\bot-r_0) + C_\star \delta\right)t}.
\]
Moreover, the map
\[
\begin{split}
\R \times C^\infty(\mc{M},T\mc{M}) \ni (t,X) \mapsto (M^X(t),S^X(t)) \in \mc{L}(L^2(\mc{N})) \times \mc{K}(L^2(\mc{N}))
\end{split}
\]
is smooth.
\end{proposition}

The rest of this paragraph is devoted to the proof of Proposition \ref{proposition:crucial-bound}. It is split into several sub-lemmas. Given a regularity pair ${\bf{r}} = (r_\bot, r_0)$, in order to simplify notation we introduce:
\begin{equation}
\label{equation:mr}
m_{\mathbf{r}} := r_\bot m - r_0.
\end{equation} 
We start with:

\begin{lemma}
\label{lemma:order-zero}
For all $t \in \R, 1/2 < \varrho < 1$, $B^X_1(t), B^X_2(t) \in \Psi^0_{\varrho}(\mc{N})$ with principal symbols 
\[
\sigma_{B_1^X(t)}(y,\xi)=e^{- \omega\int_0^t (\varphi^X_{s})^*(q)(y) ~\dd s} , \quad  \sigma_{B_2^X(t)}(y,\xi)=\dfrac{\langle e^{t\X}(y,\xi) \rangle^{m_{\mathbf{r}}(e^{t\X}(y,\xi))}}{\langle \xi \rangle^{m_{\mathbf{r}}(y,\xi)}}.
\]
 \end{lemma}
 
 \begin{proof}
This follows directly from Egorov's Lemma, see \cite[Section \S 2.4.1]{Lefeuvre-thesis}.
 \end{proof}
 
In particular, Lemma \ref{lemma:order-zero} shows that the integrand $e^{-tX} B^X_1(t)B^X_2(t)$ on the right-hand side of \eqref{equation:identite2} is a Fourier Integral Operator (FIO). We let
\begin{equation}
\label{equation:compensation}
a^X(t)(y) := |\det d\varphi_{-t}^X(\varphi_t^X(y))|^{-1/2},
\end{equation}
where the Jacobian is defined with respect to the measure $\dd \mu$ on $\mc{N}$.

\begin{lemma}
\label{lemma:unitary}
For all $t \in \R$, $\|e^{-tX} (a^X(t))^{-1}\|_{\mc{L}(L^2(\mc{N}))} = 1$. Moreover, for all $y \in \mc{N}, t \in \R$:
\[
a^X(t)(y) \leq \exp\left(\int_0^t |\Div_\mu X|(\varphi_s^X(y)) \dd s\right).
\]
\end{lemma}

\begin{proof}
We have
\[
\int_{\mc{N}}|e^{-tX}((a^X(t))^{-1}f)|^2d\mu= \int_{\mc{N}}(a^X(t))^{-2}|f|^2 \, |\det d\varphi^X_t|d\mu=\|f\|_{L^2}^2.
\]
The estimate on $a^X(t)(y)$ follows directly from the fact that  $\Div_\mu X \circ \varphi_t = \pl_t(\log |\det d\varphi^X_t|)$.
\end{proof}

By Lemma \ref{lemma:order-zero}, the operator $a^X(t) E B^X_1(t) B^X_2(t)$ is a pseudodifferential operator of order $0$. By the Calderón-Vaillancourt Theorem \cite[Theorem 4.5]{Grigis-Sjostrand-94}, up to a compact remainder in $\mc{K}(L^2(\mc{N}))$, its norm on $L^2(\mc{N})$ is given by the $\limsup$ of its principal symbol as $|\xi| \to \infty$. We now bound the $\limsup$ of its principal symbol. 

 \begin{lemma}
 \label{lemma:bound-symbol}
There exists $\delta_\star, C_\star,\Lambda > 0$ such that the following holds. For all regularity pairs $\mathbf{r} := (r_\bot,r_0)$, there exists $C(\mathbf{r}), \omega(\mathbf{r}) > 0$ such that for all smooth vector fields $X$ with $\|X-X_0\|_{C^2(\mc{M},T\mc{M})} \leq \delta$, where $\delta \leq \delta_\star$, for all $t \geq 0$, 
 \[
 \limsup_{(y,\xi) \in T^*\mc{N}, |\xi|\to \infty}  \sigma_{a^X(t) E B_1^X(t)B_2^X(t)}(y,\xi) \leq C(\mathbf{r}) e^{\left(-\Lambda (r_\bot-r_0) +C_\star\delta\right)t}.
 \]
 \end{lemma}
  \begin{proof}
For $(y,\xi) \in T^*\mc{N}$, we have by Lemma \ref{lemma:order-zero}:
 \begin{equation}
 \label{equation:principal}
 \sigma_{a^X(t) E B^X_1(t)B^X_2(t)}(y,\xi) = e(y,\xi) \exp\left(\int_0^t \left(\frac{1}{2}\Div_\mu X- \omega q\right)(e^{sX}(y)) \dd s\right) \dfrac{\langle e^{t\X}(y,\xi) \rangle^{m_{\mathbf{r}}(e^{t\X}(y,\xi))}}{\langle \xi \rangle^{m_{\mathbf{r}}(y,\xi)}}.
 \end{equation}
Modulo the term $e(y,\xi) \leq 1$ which we can neglect, this is a \emph{cocycle} over the flow of $\X$ as it satisfies the relation
 \begin{equation}
 \label{equation:cocycle}
 \sigma_{B^X_1(t')B^X_2(t')}(e^{t\X}(y,\xi))\sigma_{B^X_1(t)B^X_2(t)}(y,\xi) = \sigma_{B^X_1(t'+t)B^X_2(t'+t)}(y,\xi),
 \end{equation}
 for all $t,t' \in \R$.
 
First, we need the following:
\begin{lemma}
 \label{lemma:bound-symbol1}
For all regularity pairs $\mathbf{r}=(r_\bot,r_0)$, there exist constants $C(\mathbf{r}),\omega(\mathbf{r}) > 0$ such that for all $(y,\xi) \in T^*\mc{N}$, $\omega > \omega(\mathbf{r})$, and for all $t \geq 0$:
\[
\left\{e^{s\X}(y,\xi) ~|~ s \in [0,t]\right\} \subset \pi^{-1}\left(\left\{q=1\right\}\right) \implies \limsup_{(y,\xi) \in T^*\mc{N}, |\xi| \to \infty}  \sigma_{a^X(t) E B^X_1(t)B^X_2(t)}(y,\xi)  \leq C(\mathbf{r}) e^{-rt}, 
\]
where $r:=r_\bot+r_0$.
\end{lemma}

\begin{proof}
Define $\nu := \sup_{\|X-X_0\|_{C^2} \leq \delta} \|\Div_\mu X\|_{L^\infty(\mc{N})}$. We have, if $q(\varphi_s(x))=1$ for $s\in [0,t]$,
\[
\begin{split}
\sigma_{a^X(t) E B_1(t)B_2(t)}(y,\xi) & \leq e^{\nu t} e^{-\omega t} \dfrac{\langle e^{t\X}(y,\xi) \rangle^{m_{\mathbf{r}}(e^{t\X}(y,\xi))}}{\langle \xi \rangle^{m_{\mathbf{r}}(y,\xi)}} \\
& = e^{(\nu- \omega) t} \langle e^{t\X}(y,\xi)\rangle^{m_{\mathbf{r}}(e^{t\X}(y,\xi))-m_{\mathbf{r}}(y,\xi)} \left(\dfrac{\langle e^{t\X}(y,\xi)\rangle}{\langle \xi \rangle}\right)^{m_{\mathbf{r}}(y,\xi)}.
\end{split}
\]
By construction, $m_{\mathbf{r}}$ is non-increasing along the flow lines of $\X$ outside a neighborhood of the $0$-section in $T^*\mc{N}$, see Proposition \ref{proposition:escape}, item (6). This implies that
\[
\limsup_{(y,\xi) \in T^*\mc{N}, |\xi| \to \infty} ~~ \langle e^{t\X}(y,\xi)\rangle^{m_{\mathbf{r}}(e^{t\X}(y,\xi))-m_{\mathbf{r}}(y,\xi)} \leq 1.
\]
Moreover, there exist a uniform exponent $\lambda > 0$ and $C > 0$ (depending only on $X_0$) such that for all $X \sim X_0$, for all $t \geq 0, (y,\xi) \in T^*\mc{N}$, one has:
\begin{equation}
\label{equation:uniform-bound}
\langle e^{t\X}(y,\xi) \rangle \leq Ce^{\lambda t}\langle \xi \rangle.
\end{equation}
Using \eqref{equation:uniform-bound} and taking the $\limsup$ as $|\xi| \to \infty$, we then obtain:
\[
\limsup_{(y,\xi) \in T^*\mc{N}, |\xi| \to \infty}  \sigma_{a^X(t) E B^X_1(t)B^X_2(t)}(y,\xi) \leq C(\mathbf{r}) e^{(\nu- \omega+r\lambda) t}.
\]
Taking $\omega(\mathbf{r}) := \nu + r + r\lambda$, we obtain the announced result.
\end{proof}

From now on, given a regularity pair $\mathbf{r}$, the constant $\omega$ in \eqref{equation:identite2} will always be taken to be fixed, equal to $\omega := \omega(\mathbf{r}) > 0$ provided by Lemma \ref{lemma:bound-symbol1}. Next we need the following:
\begin{lemma}
 \label{lemma:bound-symbol2}
There exists $C_\star,\Lambda_1 > 0$ such that the following holds. For all regularity pairs $\mathbf{r}$, there exists a constant $C(\mathbf{r}) > 0$ such that for all $X$ with $\|X-X_0\|_{C^2} \leq \delta$, $(y,\xi) \in T^*\mc{N}$, for all $t \geq 0$:
\[
(y,\xi), e^{t \X}(y,\xi) \in T^*\mc{M}_e \implies \limsup_{|\xi| \to \infty} \sigma_{a^X(t) E B^X_1(t)B^X_2(t)}(y,\xi) \leq C(\mathbf{r}) e^{\left(-\Lambda_1 (r_\bot-r_0) + C_\star \delta\right) t}.
\]
\end{lemma}

\begin{proof}
We start with a preliminary observation: there exists a constant $C_\star > 0$ such that if $\|X-X_0\|_{C^2(\mc{M},T\mc{M})} \leq \delta$ and $y, \varphi_t^{X}(y) \in \mc{M}_e$, then 
\begin{equation}
\label{equation:div}
a^X(t)(y) \leq e^{C_\star \delta t}.
\end{equation}
This simply follows from the fact that $X_0$ is volume preserving on $\mc{M}_e$ (that is $a^{X_0}(t)=1$).

We now consider the sets $U_\pm$ given by Lemma \ref{lemma:expansion}. These sets can always be constructed so that $U_\pm \subset \left\{m=\pm 1\right\}$. We also consider the sets $V_\pm$ given by Lemma \ref{lemma:contraction}. Denote by $T > 0$ the time provided by Lemma \ref{lemma:time}. If $t \leq T$, namely if the time is uniformly bounded, then the claim is immediate as $a^X(t) E B^X_1(t)B^X_2(t)$ is of order $0$ by Lemma \ref{lemma:order-zero} and depends continuously on time. If $t \geq T$, and $(y,\xi), e^{t \X}(y,\xi) \in T^*\mc{M}_e \cap \WF(E)$, then the flow line $\left\{e^{s\X}(y,\xi) ~|~ s \in [0,t]\right\}$ passes at least a time $t-T$ in $U_+ \sqcup U_-$. We can thus introduce $0 \leq s_0 < s_1 \leq t$ such that for all $s \in [0,s_0], e^{s\X}(y,\xi) \in U_-$ and for all $s \in [s_1,t], e^{s\X}(y,\xi) \in U_+$ and we have $s_0 + (t-s_1) \geq t-T$. Hence, using the cocycle relation \eqref{equation:cocycle} and that $\sigma_E \in [0, 1]$:
\begin{equation}
\label{equation:split}
\begin{split}
\sigma_{a^X(t) EB_1^X(t)B_2^X(t)}(y,\xi) \leq&  \sigma_{a^X(t-s_1) B_1^X(t-s_1)B_2^X(t-s_1)}\left(e^{s_1\X}(y,\xi)\right) \\
& \hspace{2cm}  \cdot \sigma_{a^X(s_1-s_0) B_1^X(s_1-s_0)B_2^X(s_1-s_0)}\left(e^{s_0 \X}(y,\xi)\right) \\
& \hspace{2cm} \cdot \sigma_{a^X(s_0) B_1^X(s_0)B_2^X(s_0)}(y,\xi).
\end{split}
\end{equation}
Note that it suffices to bound the terms on the right hand side of \eqref{equation:split} on $\WF(E)$, that is, on a conic neighbourhood of $\cup_{X \sim X_0} \left\{p_X = 0\right\}$, since otherwise $\sigma_E = 0$ and the symbol on the left hand side vanishes. 

Since $s_1-s_0 \leq T$ (independent of $t$) and $\sigma_{B_1^X(t)B_2^X(t)} \in \Psi^0_{\varrho}(\mc{N})$ for all $t \geq 0$ by Lemma \ref{lemma:order-zero}, we get that the middle term in \eqref{equation:split} is bounded uniformly by some constant, that is,
\begin{equation}
\label{equation:uzero}
\sigma_{a^X(s_1-s_0) B_1(s_1-s_0)B_2(s_1-s_0)}\left(e^{s_0 \X}(y,\xi)\right) \leq C(\mathbf{r})
\end{equation}
for some $C(\mathbf{r}) > 0$ which is independent of the point $(y,\xi) \in T^*\mc{N}$ and of the time $t$. As to the third factor in \eqref{equation:split}, we have, using that $m_{\mathbf{r}}=r_\bot-r_0$ on $U_-$, that $q$ vanishes in $\mc{M}$, and \eqref{equation:div}
\begin{equation}
\label{equation:contr}
\begin{split}
 \sigma_{a^X(s_0)B_1(s_0)B_2(s_0)}(y,\xi)  & \leq e^{C_\star \delta s_0} e^{-\int_0^{s_0}   \omega(\mathbf{r}) q(e^{sX}(y)) \dd s} \dfrac{\langle e^{s_0 \X}(y,\xi) \rangle^{m_{\mathbf{r}}(e^{s_0\X}(y,\xi))}}{\langle \xi \rangle^{m_{\mathbf{r}}(y,\xi)}} \\
 &  \leq C(\mathbf{r}) e^{C_\star \delta s_0}  \left(\dfrac{\langle e^{s_0 \X}(y,\xi) \rangle}{\langle \xi \rangle}\right)^{r_\bot-r_0}.
 \end{split}
\end{equation}
Using the uniform contraction rate on $U_-$ of Lemma \ref{lemma:expansion}, we get that $|e^{s_0 \X}(y,\xi)| \leq C e^{-\lambda s_0} |\xi|$, for some uniform constants $C, \lambda > 0$ depending only on $X_0$. Taking the $\limsup$ as $|\xi| \to \infty$ in \eqref{equation:contr}, we thus obtain
\begin{equation}
\label{equation:umoins}
\limsup_{|\xi| \to \infty}  \sigma_{a^X(s_0) B_1(s_0)B_2(s_0)}(y,\xi) \leq C(\mathbf{r}) e^{C_\star \delta s_0} e^{-\lambda  s_0 (r_\bot-r_0)}.
\end{equation}
Similarly, using the expansion rate on $U_+$ of Lemma \ref{lemma:expansion} and that $m_{\mathbf{r}} = -r_\bot-r_0$ on $U_+$, the first term in \eqref{equation:split} can be bounded by
\begin{equation}
\label{equation:uplus}
\limsup_{|\xi| \to \infty} \sigma_{a^X(t-s_1) B_1(t-s_1)B_2(t-s_1)}\left(e^{s_1\X}(y,\xi)\right) \leq C(\mathbf{r}) e^{C_\star \delta (t-s_1)} e^{-\lambda (t-s_1) (r_\bot+r_0)}.
\end{equation}
Taking $\Lambda_1 := \lambda$, and combining \eqref{equation:uzero}, \eqref{equation:umoins}, \eqref{equation:uplus} in \eqref{equation:split} completes the proof.
\end{proof}

We can now end the proof of Lemma \ref{lemma:bound-symbol}. Given $(y,\xi) \in T^*\mc{N}$, the flowline of $(y,\xi)$ under $e^{t\X}$ can be schematically described by one of the five following possibilities:
\begin{align}
& \left\{q=1\right\}, \label{equation:trivial1} \\
& \mc{M}_e\label{equation:stayinMe},\\
 &\left\{q = 1\right\} \to \left\{0 < q < 1\right\} \to \left\{q=1\right\}, \label{equation:trivial2}\\
&\left\{q = 1\right\} \to \left\{0 < q < 1\right\} \to \mc{M}_e, \label{equation:dur1}\\
& \mc{M}_e \to  \left\{0 < q < 1\right\} \to \left\{q = 1\right\}, \label{equation:dur2} \\
&\left\{q = 1\right\} \to \left\{0 < q < 1\right\} \to  \mc{M}_e \to  \left\{0 < q < 1\right\} \to \left\{q = 1\right\}.\label{equation:dur3}
\end{align}
Note that, for any flow line, there is a maximum time, bounded by some uniform constant $T_\star > 0$ spent in the region $\left\{0 < q < 1\right\}$. As a consequence, if the flowline of $(y,\xi)$ falls in one of the cases \eqref{equation:trivial1}, \eqref{equation:trivial2}, we get, using the cocycle relation \eqref{equation:cocycle} and Lemma \ref{lemma:bound-symbol1}:
\[
\limsup_{|\xi| \to \infty} \sigma_{a^X(t)EB_1(t)B_2(t)}(y,\xi) \leq C(\mathbf{r}) e^{-r t}.
\]
As to  \eqref{equation:stayinMe}, \eqref{equation:dur1}, \eqref{equation:dur2}, the bound will be obtained similarly to the bound for \eqref{equation:dur3}, which we now study.

So we assume that the flowline $\gamma$ of $(y,\xi)$ under $e^{t\X}$ passes successively through the five sets of \eqref{equation:dur3}. Define the times $s_{0}, s_1 \geq 0$ such that
\[
\begin{split}
\forall s \in [0,s_0], ~~~\varphi_s^X(y)& \in \left\{q=1\right\},  \forall s \in [s_0,s_1], ~~~\varphi_s^X(y) \in \left\{q < 1\right\} \cup \mc{M}_e, \\
& \forall s \in [s_1,t], ~~~\varphi_s^X(y) \in \left\{q=1\right\}.
\end{split}
\]
Combining the cocycle relation \eqref{equation:cocycle} and Lemmas \ref{lemma:bound-symbol1}, \ref{lemma:bound-symbol2}, we get on $\WF(E)$:
\[
\begin{split}
& \limsup_{|\xi| \to \infty} \sigma_{a^X(t)E B_1(t)B_2(t)}(y,\xi) \\
&  \leq \limsup_{|\xi| \to \infty}  \sigma_{a^X(t-s_1) B_1(t-s_1)B_2(t-s_1)}\left(e^{s_1\X}(y,\xi)\right)  \\
& \cdot \limsup_{|\xi| \to \infty}  \sigma_{a^X(s_1-s_0) B_1(s_1-s_0)B_2(s_1-s_0)} \left(e^{s_0 \X}(y,\xi)\right)   \cdot \limsup_{|\xi| \to \infty}  \sigma_{a^X(s_0)B_1(s_0)B_2(s_0)}(y,\xi) \\
& \leq C_r e^{-r(t-s_1)} \cdot C_r e^{(-(r_\bot-r_0) \Lambda_1 + C_\star \delta) (s_1-s_0)} \cdot C_r e^{-r s_0}  \leq C_r e^{(-(r_\bot-r_0) \Lambda + C_\star \delta)t},
\end{split}
\]
by taking $\Lambda := \min(1,\Lambda_1)$. This concludes the proof.
\end{proof}

We now complete the proof of Proposition \ref{proposition:crucial-bound}.

\begin{proof}[Proof of Proposition \ref{proposition:crucial-bound}]
Write
\[
e^{-tX} E B_1(t) B_2(t) = e^{-tX} (a^X(t))^{-1} a^X(t) E B_1(t) B_2(t).
\]
By Lemma \ref{lemma:unitary}, $e^{-tX} (a^X(t))^{-1} \in \mc{L}(L^2(\mc{N}))$ is unitary. By Lemma \ref{lemma:bound-symbol}, $a^X(t) E B_1(t) B_2(t)$ is a pseudodifferential operator of order $0$ such that
\[
\limsup_{(y,\xi) \in T^*\mc{N},|\xi| \to \infty} \sigma_{a^X(t) E B_1(t) B_2(t)}(y,\xi) \leq C(\mathbf{r}) e^{\left(-(r_\bot-r_0)\Lambda + C_\star \delta \right)t}.
\]
By the Calderón-Vaillancourt Theorem \cite[Theorem 4.5]{Grigis-Sjostrand-94} for pseudodifferential operators, we can thus write
\[
a^X(t) E B_1(t) B_2(t) = M^X_0(t) + S^X_0(t),
\]
where $M^X_0(t)$ is a pseudodifferential operator of order $0$ and $S^X_0(t)$ is smoothing and
\[
\|M^X_0(t)\|_{\mc{L}(L^2(\mc{N}))} \leq 2 C(\mathbf{r}) e^{\left(-(r_\bot-r_0)\Lambda + C_\star \delta \right)t}.
\]
Moreover, it is straightforward to check that these operators can be constructed so that they depend smoothly on the parameters $t \in \R$ and $X \in C^\infty(\mc{M},T\mc{M})$ as $a^X(t),B_1(t),B_2(t)$ depend in an explicit (and smooth) fashion on $X$, and the decomposition in the Calderón-Vaillancourt Theorem depends smoothly on the operator. As a consequence, setting $M^X(t) := e^{-tX} (a^X(t))^{-1} M^X_0(t)$ and $S^X(t) := e^{-tX} (a^X(t))^{-1} S^X_0(t)$, we have
\[
e^{-tX} E B_1(t) B_2(t) = M^X(t) + S^X(t),
\]
and this concludes the proof.
\end{proof}

\subsubsection{Meromorphic extension on the closed manifold}

We now prove Theorem \ref{theorem:meromorphic}. 
\begin{proof}[Proof of Theorem \ref{theorem:meromorphic}]
\emph{Step 1: meromorphic extension.} Fix ${\mathbf{r}}=(r_\bot,r_0)$ with $r_\bot > r_0$, and consider $z \in \C$ such that $\Re(z) > -\Lambda (r_\bot-r_0) + C_\star \delta$. By Proposition \ref{proposition:crucial-bound}, we can consider a time $T > 0$ large enough, depending on ${\mathbf{r}}$, so that: 
\begin{equation}
\label{equation:condition-time}
\forall t \geq T, ~~~~  e^{-\Re(z)t} \|M^X(t)\|_{\mc{L}(L^2(\mc{N}))} < 1/6.
\end{equation}

Using \eqref{equation:identite2} and \eqref{equation:identite3}, we thus obtain:
\[
\int_0^{+\infty} \chi_T'(t) e^{-tz} e^{-tX} B^X_1(t) B^X_2(t) ~\dd t = B^X(z) + K^X(z),
\]
where
\[
\begin{split}
B^X(z) := \int_0^{+\infty} \chi_T'(t) e^{-tz} M^X(t)  \dd t,
\end{split}
\]
and $K^X(z) \in \Psi^{-\infty}(\mc{N})$ is the remainder. It is immediate to check that both $B^X(z)$ and $K^X(z)$ depends holomorphically on $z$ and smoothly on $X \in C^\infty(\mc{M},T\mc{M})$ as operators in $\mc{L}(L^2(\mc{N}))$.

Using that $\|\chi'_T\|_{L^\infty}\leq 2$, we get:
\begin{equation}
\label{equation:key-bound}
\begin{split}
\|B^X(z)\|_{\mc{L}(L^2(\mc{N}))} & \leq 2 \int_T^{T+1} e^{-\Re(z) t}\|M^X(t)\|_{\mc{L}(L^2(\mc{N}))}  ~\dd t \leq 1/3 < 1.
\end{split}
\end{equation}
The equality \eqref{equation:identite2} then reads:
\begin{equation}\label{eq:identite3}
\begin{split}
- A_{\mathbf{r}} \int_0^{+\infty} \chi_T(t) e^{-tz}  e^{- \int_0^t (\varphi^X_{-s})^*(\omega q)~ \dd s}   e^{-tX} A_{\mathbf{r}}^{-1} \dd t ~~~\underbrace{A_{\mathbf{r}}(-X-z-\omega q)A_{\mathbf{r}}^{-1}}_{=: -P^X - z} \\
 = \mathbbm{1} + B^X(z) + K^X(z),
\end{split}
\end{equation}
and $\mathbbm{1}+B^X(z)$ is invertible while $K^X(z)$ is compact. Moreover, for $\Re(z) \gg 1$, $\mathbbm{1} + B^X(z) + K^X(z)$ is invertible on $\mc{L}(L^2(\mc{N}))$ since the $L^2$-norm of $B^X(z) + K^X(z)$ is exponentially decaying as $\Re(z) \to +\infty$. By the Fredholm analytic theorem \cite[Theorem D.4]{Zworski-12}, we deduce that
\[
z \mapsto (\mathbbm{1} + B^X(z) + K^X(z))^{-1} \in \mc{L}(L^2(\mc{N}))
\]
is a meromorphic family of operators on $\left\{\Re(z) > -\Lambda (r_\bot-r_0) + C_\star \delta \right\}$. Equivalently, 
\[
z \mapsto -X-z-\omega(\mathbf{r}) q,
\]
is a holomorphic family of Fredholm operators\footnote{Note that this is an unbounded family of operators. Since Fredholm operators are continuous by definition, one has to consider the operators on their domain $\mc{D}(\mc{H}^{\mathbf{r}}_+) := \left\{ f \in \mc{H}^{\mathbf{r}}_+ ~|~ Xf \in \mc{H}^{\mathbf{r}}_+\right\}$.} of index $0$ on the anisotropic space $\mc{H}^r_+(\mc{N})$ that is invertible for $\Re(z) \gg 1$. Thus
\[
z \mapsto R_-^X(z) = (-X-z-\omega(\mathbf{r}) q)^{-1} \in \mc{L}(\mc{H}^{\mathbf{r}}_+)
\]
is a meromorphic family of operators on $\left\{\Re(z) > -\Lambda (r_\bot-r_0) + C_\star \delta \right\}$. This proves the first part of the theorem; we next study the dependence in $X$ and $z$.

\medskip
\emph{Step 2: continuity of resonances.} Assume $z_0$ is not a pole of $z \mapsto R^{X_0}(z)$ and furthermore that it does not have any poles in the closed disk $D(z_0, \varepsilon_0) \subset \mathbb{C}$ (since the resolvent is meromorphic, such $\varepsilon_0 > 0$ exists). We first show that for $X$ sufficiently close to $X_0$ in $C^N$ for some $N$ large enough, the map $z \mapsto R^X(z)$ does not have any poles in $D(z_0, \varepsilon_0)$. Let $z \in D(z_0, \varepsilon_0)$; we will use the identity \eqref{eq:identite3}. We first claim that we may pick the cutoff function $\chi$ suitably and $T$ sufficiently large such that 
\[\ker(\mathbbm{1} + B^X(z) + K^X(z))|_{L^2} = 0.\]
Note that, as we will see below, this kernel could be non-zero even if $z$ is not a resonance of $-X - q\omega$; we will show that generically this does not happen. We will argue by assuming that there is non-zero $u \in L^2(\mc{N})$ such that $(\mathbbm{1} + B^X(z) + K^X(z))u = 0$. Since $K^X(z) \in \Psi^{-\infty}(\mc{N})$, we get
\[(\mathbbm{1} + B^X(z)) u \in C^\infty(\mc{N}) \subset \mc{D}({L^2}) = \{f \in L^2(\mc{N}) \mid Xf \in L^2(\mc{N})\},\] 
and since $\mathbbm{1} + B^X(z)$ is invertible on $\mc{D}(L^2)$ (and on $L^2(\mc{N})$, by construction), we conclude that $u \in \mc{D}(L^2)$. Since $P^X + z$ commutes with $\mathbbm{1} + B^X(z) + K^X(z)$, we have that $P^X + z$ acts on $\ker(\mathbbm{1} + B^X(z) + K^X(z))|_{L^2}$, which is a finite dimensional space by the Fredholm property shown above. Therefore, we can pick $u$ such that $(P^X + z + \lambda) u = 0$ for some $\lambda \in \mathbb{C}$; by assumption, we have $\lambda \neq 0$. Write $u = A_{\mathbf{r}} v$ for some $v \in \mc{H}^{\mathbf{r}}_+$. This implies
\[e^{-tX} v  = e^{(z + \lambda)t} e^{\int_0^t (\varphi_{-s}^X)^*(q\omega)\, ds} v, \quad \forall t \in \mathbb{R},\]
and hence
\begin{align*}
	0 = (\mathbbm{1} + B^X(z) + Q^X(z))u &= -A_{\mathbf{r}} \Big(\mathbbm{1} + \int_0^{+\infty} \chi_T'(t) e^{-tz} e^{- \int_0^t (\varphi^X_{-s})^*(q \omega ) ~\dd s} e^{-tX} \dd t\Big) v\\ 
	&= -\Big(1 +\underbrace{\int_T^{T + 1} \chi_T'(t) e^{\lambda t}\,dt}_{F(\chi_T, \lambda) := }\Big) u.
\end{align*}
If $\Re(\lambda) \leq 0$, the integral in the last equality can be bounded by $\|\chi_T'\|_{C^0} e^{T \Re(\lambda)}$; then
\begin{equation}\label{eq:loc1}
	\|\chi_T'\|_{C^0} e^{T \Re(\lambda)} < 1 \iff \Re(\lambda) < -\frac{1}{T}\log(\|\chi_T'\|_{C^0}).
\end{equation}
Moreover, integrating by parts once we have
\[\int_T^{T + 1} \chi_T'(t) e^{\lambda t}\,dt = -\frac{1}{\lambda} \int_T^{T + 1} \chi''_T(t) e^{\lambda t}\, dt\]
which is in absolute value bounded by $\frac{1}{|\lambda|} \|\chi_T''\|_{C^0} e^{(T + 1)|\Re(\lambda)|}$. Then
\begin{equation}\label{eq:loc2}
	\frac{1}{|\lambda|} \|\chi_T''\|_{C^0} e^{(T + 1)|\Re(\lambda)|} < 1 \iff |\Re(\lambda)| < \frac{\log |\lambda| - \log \|\chi_T''\|_{C^0}}{T + 1}.
\end{equation}
Using \eqref{eq:loc1} and taking $T$ large enough (changing $\chi_T$ in such a way that $\chi_T|_{[T, T+1]}$ is the same as before after a translation), we conclude $1 + F(\chi_T, \lambda)$ has no zeroes (in $\lambda$) in $\{\Re(\la) > -\kappa\}$ where $\kappa = \kappa(T) > 0$ can be chosen arbitrarily small; we conclude that $z + \lambda$ is a resonance of $-X - q\omega$. Using additionally \eqref{eq:loc2}, we conclude that $z + \lambda$ belongs to a finite set of resonances $\mc{S} \subset \mathbb{C}$ of $-X - q\omega$ (in the regions defined by \eqref{eq:loc1} and \eqref{eq:loc2}; note that there are no resonances with sufficiently large real part). Observe that the set $\mc{S}$ depends only on $T$, $\|\chi_T'\|_{C^0}$, and $\|\chi''_T\|_{C^0}$. Enumerate elements of the set $\mc{S} - z$ by $\lambda_1, \dotsc, \lambda_k$ for some $k \geq 0$.  

We now perturb $\chi_T$ by considering $\chi_T + s \eta_T$, where $\eta_T \in C_c^\infty((T, T + 1))$ is a smooth cutoff function and $s \in \mathbb{R}$ is small in absolute value. Assume $F(\chi_T, \lambda) = -1$ and $\Re(e^{i\Im(\lambda) t})$ to be positive on an interval $(T_1, T_2) \subset (T, T + 1)$ (we argue similarly if it is negative), where $\lambda \in \mc{S} - z$. Taking $\eta \neq 0$ to be non-negative and supported on $(T_1, T_2)$, there is an $s > 0$ small enough such that 
\[1 + F(\chi_T + s\eta, \lambda) = - \lambda s \int_T^{T + 1} \eta(t) e^{\lambda t}\, dt \neq 0.\] 
Arguing inductively, we ensure that $F(\widetilde{\chi}_T, \lambda_i) \neq -1$ for $i = 1, \dotsc, k$ for some new cutoff function $\widetilde{\chi}_T$ (satisfying all the previously set out conditions of $\chi_T$). We conclude that $\ker(\mathbbm{1} + B^X(z) + K^X(z))|_{L^2} = \{0\}$ with these new choices of $T$ and $\chi_T$, proving the claim.

As previously explained, since $B^X(z') + K^X(z')$ depend continuously on $X$ and $z'$ in $\mc{L}(L^2)$, there is an $\varepsilon(z) > 0$ small enough such that for $\|X - X_0\|_{C^N} < \varepsilon(z)$ and $|z - z'| < \varepsilon(z)$, we have $\mathbbm{1} + B^X(z) + K^X(z)$ invertible on $L^2$ (since it has empty kernel and is Fredholm of index zero). This implies that there are no resonances in $D(z, \varepsilon(z))$ for $z \in D(z_0, \varepsilon_0)$. By compactness of $D(z_0, \varepsilon_0)$, we conclude that there is an $\varepsilon > 0$ such that there are no resonances in $D(z_0, \varepsilon_0)$ for $\|X - X_0\|_{C^N} < \varepsilon$, proving the originally sought claim.\footnote{A different proof of this step can be found in \cite{Bonthonneau-20}.}

\medskip
\emph{Step 3: smoothness of the resolvent.} Now, using the following resolvent identity valid for $z \in D(z_0, \varepsilon_0)$ and $X$ close to $X_0$ in $C^N$
\begin{equation*}
	R_-^X(z)-R_-^{X'}(z) = R_-^{X'}(z)(X-X')R_-^{X}(z),
\end{equation*}
we obtain that $X \mapsto R^X_-(z)$ is twice differentiable in $X$, uniformly in $z \in D(z_0,\eps_0)$, with
 \begin{align}
 & \label{equation:close-diff1}
\partial_X (R^X_-(z)).Y = R^X_-(z) Y R^X_-(z), \\
& \label{equation:close-diff2} \partial^2_X(R^X_-(z)).(Y,Y') = R^X_-(z) Y' R^X_-(z)  Y R^X_-(z) + R^X_-(z) Y R^X_-(z)  Y' R^X_-(z).
\end{align}
where $Y,Y' \in C^\infty(\mc{N},T\mc{N})$.

Using the first part of Theorem \ref{theorem:meromorphic}, namely the boundedness of $R_-^X(z)$ on the spaces $\mc{H}^{\mathbf{r}}_+$ for $X$ close to $X_0$ in $C^2$-norm, we deduce that the first derivative \eqref{equation:close-diff1} is bounded as a map
\[
\mc{H}^{(r_\bot,r_0)}_+ \overset{R_-^{X}(z)}{\longrightarrow} \mc{H}^{(r_\bot,r_0)}_- \overset{Y}{\longrightarrow} \mc{H}^{(r_\bot,r_0+1)}_-  \overset{R_-^{X}(z)}{\longrightarrow} \mc{H}^{(r_\bot,r_0+1)}_-,
\]
and similarly the second derivative \eqref{equation:close-diff2} is bounded as a map $\mc{H}^{(r_\bot,r_0)}_- \to \mc{H}^{(r_\bot,r_0+2)}_-$, and this holds for all $X$ close enough to $X_0$ in the $C^N$-topology, with $N \gg 1$ large enough, and for all $z \in D(z_0,\eps_0)$. Moreover, the dependence on $z$ in \eqref{equation:close-diff1} and \eqref{equation:close-diff2} is holomorphic. This completes the proof of Theorem \ref{theorem:meromorphic}.
\end{proof}

\subsection{Smoothness of the scattering map with respect to the metric} 

\label{ssection:final}

The goal of this paragraph is to prove Proposition \ref{regularityR}. We start with the following:

\begin{lemma}
\label{lemma:res}
If $R_g$ and $R_{g_e}$ are the resolvents defined in \eqref{defofRg} for $(M,g)$ and $(M_e,g_e)$, we have for 
$X=\psi \widetilde{X}_g$ defined in \S \ref{sssection:extension} that, for all $z \in \mathbb{C}$
\[
R_g(z) = \mathbf{1}_{\mc{M}} R_+^{X}(z) \mathbf{1}_{\mc{M}}, \quad R_{g_e}(z) = \mathbf{1}_{\mc{M}_e} R_+^{X}(z) \mathbf{1}_{\mc{M}_e},
\]
when acting respectively on $C_{\comp}^\infty(\mc{M}^\circ)$ and $C_{\comp}^\infty(\mc{M}_e^\circ)$.
\end{lemma}

\begin{proof}
This is an obvious consequence of the following fact: for $f\in C_{\comp}^\infty(\mc{M}^\circ)$, writing $u_z=(R_g(z)f)|_{\mc{M}}$, if $\Re(z)\gg 1$, we have
\[
u_z(y)=-\int_{0}^{\tau_g(y)}e^{-zt}f(\varphi^g_t(y))dt
\]
and similarly for $R_{g_e}(z)$. Indeed, if $y\in \mc{M}$, the flow line $\gamma:=\cup_{t\geq 0}\varphi^g_t(y)$ 
is contained in $\{\rho>-\rho_0\}$ and the convexity of $\mc{M}$ implies that $\gamma\cap \mc{M}=\cup_{t\in [0,\tau_g(y))}\varphi^g_t(y)$.
\end{proof}

We can now complete the proof of Proposition \ref{regularityR}.

\begin{proof}[Proof of Proposition \ref{regularityR}]
Let $\omega \in C^\infty(\partial_+ \mc{M})$. Observe that by Lemmas \ref{SchwartzkernelSg} and \ref{lemma:res}:
\[
\chi \mc{S}_g(\omega) = \chi \left. \left[R_{g_e}(\tilde\chi \omega \delta_{\partial_+\mc{M}})\right]\right|_{\partial_-\mc{M}},
\]
where $\tilde \chi$ is some smooth cutoff function equal to $1$ everywhere except in a neighborhood of $\partial_0 \mc{M}$ and where 
$\tilde \chi \omega \delta_{\partial_+\mc{M}} \in \mc{D}'(\mc{N})$ denotes the distribution defined by:
\[
\langle \tilde \chi \omega  \delta_{\partial_+\mc{M}}, \varphi \rangle := \int_{\partial_+\mc{M}} \tilde \chi \omega \varphi ~\dd \mu_{\partial}.
\]

Let $u := \tilde\chi \omega \delta_{\partial_+\mc{M}}$. Since $\partial_+ \mc{M}$ is of codimension $1$, $u \in H^{-1/2-\eps}(\mc{N})$ for all $\eps > 0$. Let $N^*\pl_+\mc{M}\subset T_{\pl_+\mc{M}}^*\mc{N}$ be the conormal of $\pl_+\mc{M}$ in $\mc{N}$ (i.e. $N^*\pl_+\mc{M}(T\pl_+\mc{M})=0$). 
By a standard argument of distribution theory, the wavefront set of $u$ satisfies ${\rm WF}(u)\subset N^*\pl_+\mc{M}$.

The escape function $m$ provided by Proposition \ref{proposition:escape} can be constructed so that, over $\mc{M}$, it has only support in a small conic neighborhood of $(E_-^{X_0})^*$ and $(E_+^{X_0})^*$. In particular, this construction can be achieved so that
\begin{equation}
N^*\pl_+\mc{M} \cap \supp(m) = \emptyset.
\end{equation}
Indeed, a covector $V^*\in T_{\pl_+\mc{M}}^*\mc{N}$ so that $V^*\in (E_+^{X_0})^*$ must satisfy $V^*(X_0)=0$ and $V^*(W)$ for all $W\in T\pl_+\mc{M}$, but since $X_0$ is transverse to $\pl_-\mc{M}$, one gets $V^*=0$.
 We now take a regularity pair $\mathbf{r} := (r_\bot,r_0)$ with $1/2 < r_0 < 1$, $r_0 + 2 < r_\bot < 3$, and a small $\delta > 0$, such that $-\Lambda(r_\bot-(r_0+2)) +C_\star \delta< 0$. By the previous discussion, $u \in \mc{H}_-^{\mathbf{r}}(\mc{N})$, i.e. since $\mc{H}_-^{\mathbf{r}}(\mc{N})$ is microlocally equivalent to $H^{-r_0}$ near $N^*\pl_+ \mc{M}$. Denote by $\theta\in C_{\rm comp}^\infty(\mc{M}_e^\circ)$ a cutoff function equal to $1$ near $\mc{M}$.
We claim that the map
\[
 C^\infty(M,\otimes^2_S T^*M) \ni g \mapsto  \theta R_{g_e}\theta \in \mc{L}\left(\mc{H}^{(r_\bot,r_0)}_-(\mc{N}), \mc{H}^{(r_\bot,r_0+2)}_-(\mc{N})\right)
 \]
is $C^2$ for $g$ close to $g_{0}$.  
Indeed, similarly to the proof of Theorem \ref{theorem:meromorphic} (alternatively we could simply use Theorem 1.10 along with the fact that $g \mapsto X_g$ is smooth; we give a direct argument instead), we can use the resolvent identity (recall $X=\psi \widetilde{X}_{g}$ and $X_0=\psi \widetilde{X}_{g_0}$)
\[\theta R_{g_e}\theta -\theta R_{g_{0e}}\theta =\theta R_+^{X}(0)(X_0-X)R_+^{X_0}(0)\theta \] 
to deduce that $g\mapsto \theta R_{g_e}\theta$ is differentiable twice with
 \begin{align}
& \label{equation:diff1} \partial_g \theta R_{g_e}\theta = -\theta R_+^{X}(0) (\partial_g X) R_+^{X}(0)\theta , \\
& \label{equation:diff2} \partial^2_g  \theta R_{g_e}\theta = 2 \theta R_+^{X}(0) (\partial_g X) R_+^{X}(0) (\partial_g X) R_+^{X}(0)\theta - \theta R_+^{X}(0)(\partial^2_g X)
R_+^{X}(0)\theta .
\end{align}
The first derivative \eqref{equation:diff1} is bounded as a map
\[
\mc{H}^{(r_\bot,r_0)}_- \overset{R_+^{X}(0)}{\longrightarrow} \mc{H}^{(r_\bot,r_0)}_- \overset{ \partial_g X}{\longrightarrow} \mc{H}^{(r_\bot,r_0+1)}_-  \overset{R_+^{X}(0)}{\longrightarrow} \mc{H}^{(r_\bot,r_0+1)}_-,
\]
and similarly the second derivative \eqref{equation:diff2} is bounded as a map $\mc{H}^{(r_\bot,r_0)}_- \to \mc{H}^{(r_\bot,r_0+2)}_-$, and this holds for all $g$ close enough to $g_0$ in the $C^N$-topology, with $N \gg 1$ large enough.

As a consequence, 
 \[
 C^\infty(M,\otimes^2_S T^*M) \ni g \mapsto  \theta R_{g_e}\theta u=  \theta R_{g_e}u\in \mc{H}^{(r_\bot,r_0+2)}_-(\mc{N})
 \]
 is $C^2$-regular for $g$ close to $g_0$. Note that, as $r_\bot + r_0 + 2 < 6$
 \[
 \mc{H}^{(r_\bot,r_0+2)}_-(\mc{N}) \hookrightarrow H^{-6}(\mc{N}).
 \]
 Moreover, it satisfies $X_{g_e} \theta R_{g_e} u = 0$ near $\partial_-\mc{M}$ so $\WF(\theta R_{g_e}u) \subset \left\{p_{X_{g_e}}=0\right\}$. The restriction 
 $\chi \left.[\theta R_{g_e}u]\right|_{\partial_-\mc{M}}=\chi \left.[R_{g_e}u]\right|_{\partial_-\mc{M}} \in H^{-6}(\partial_-\mc{M})$ is therefore well-defined and depends in a $C^2$-fashion on the metric $g \in C^N(M,\otimes^2_S T^*M)$, proving the first part of Proposition \ref{regularityR}.
 
Using \eqref{equation:diff1} and \eqref{equation:diff2}, writing $g=g_{0}+h$ with 
$\|h\|_{C^N}\leq \delta$ for $\delta>0$ small and $N$ chosen large,
we have by Taylor expansion, for $u= \tilde{\chi} \omega \delta_{\partial_+ \mc{M}}$ as above:
\begin{equation}\label{TaylorexpRg}
\begin{split} 
\theta R_{g_e}u = \theta R_{g_{0e}}u - \theta R_+^{X_0}(0)((\partial_g X)|_{g=g_0}.h) R_+^{X_0}(0)u +  \int_0^1(1-t)\pl_g^2(\theta R_{g_{0e}+th}u).(h,h)dt.
\end{split}
\end{equation}
Let $Y_g(h):=\pl_gX(h)\in C^{\infty}(\mc{N},T\mc{N})$ for any metric $g$ smooth close to $g_0$ in $C^N(M, \otimes_S^2 T^*M)$. For all $k\geq 1$, one has  
$\|Y_g(h)\|_{C^k(\mc{N},T\mc{N})}\leq C_k\|h\|_{C^{k+1}}$ for some $C_k> 0$ depending uniformly on $\|g\|_{C^{k+1}}$. 
Let $Z_g(h,h)=\pl^2_gX(h,h)\in C^{\infty}(\mc{N},T\mc{N})$. One has $\|Z_g(h,h)\|_{C^k(\mc{N},T\mc{N})}\leq C_k\|h\|^2_{C^{k+2}}$ for some $C_k>0$ depending uniformly on $\|g\|_{C^{k+2}}$.  Then the remainder term in \eqref{TaylorexpRg} satisfies, for $g_e(t)$ the extension of $g(t)=g_0+th$ (with $t\in [0,1]$) and $X(t)=\psi \widetilde{X}_{g(t)}$,
\[ \begin{split}
\pl_g^2(\theta R_{g_e(t)}u)(h,h)=  2 \theta R_+^{X(t)}(0) Y_{g(t)}(h) R_+^{X(t)}(0) Y_{g(t)}(h)  R_+^{X(t)}(0)u  \\ 
 \hspace{5cm} - \theta R_+^{X(t)}(0) Z_{g(t)}(h,h)R_+^{X(t)}(0)u.
\end{split}\]
By the analysis above, for $\delta>0$ small and $N>0$ large enough, there exists a constant $C>0$ such that for $h=g(1) - g_0$ such that $\|h\|_{C^N}\leq \delta$:
\[
\begin{split}
& \sup_{t\in [0,1]} \|R_+^{X(t)}u\|_{\mc{H}^{(r_\bot,r_0+j)}_-(\mc{N})}\leq C,  \forall j \in \left\{0,1,2\right\},\\
&\sup_{t\in [0,1]} \|Y_{g(t)}(h)\|_{\mc{H}^{(r_\bot,r_0 + j)}_-\to \mc{H}^{(r_\bot,r_0+1 + j)}_-}\leq C\|h\|_{C^N}, \forall j \in \left\{0,1\right\},\\
&\sup_{t\in [0,1]}\|Z_{g(t)}(h,h)\|_{\mc{H}^{(r_\bot,r_0)}_-\to \mc{H}^{(r_\bot,r_0+2)}_-}\leq C\|h\|^2_{C^N}.
\end{split}
\] 
Combining the last inequalities with \eqref{TaylorexpRg}, this shows \eqref{Taylororder2S} by applying the restriction to $\pl_- \mc{M}$ on the left of \eqref{TaylorexpRg}. Note that, in turn, this gives an expression of $\pl_g\mc{S}_g|_{g= g_0}$ in terms of $R_+^{X_0}(0)$ and $\pl_gX|_{g=g_0}$. This concludes the proof.
\end{proof}

\bibliographystyle{alpha}
\bibliography{Biblio}

\begin{thebibliography}{{Lef}19a}

\bibitem[BI10]{Burago-Ivanov-10}
Dmitri Burago and Sergei Ivanov.
\newblock Boundary rigidity and filling volume minimality of metrics close to a
  flat one.
\newblock {\em Ann. of Math. (2)}, 171(2):1183--1211, 2010.

\bibitem[BK85]{Burns-Katok-85}
Keith Burns and Anatole Katok.
\newblock Manifolds with nonpositive curvature.
\newblock {\em Ergodic Theory Dynam. Systems}, 5(2):307--317, 1985.

\bibitem[Bon20]{Bonthonneau-20}
Yannick~Guedes Bonthonneau.
\newblock Perturbation of {R}uelle resonances and {F}aure-{S}j\"{o}strand
  anisotropic space.
\newblock {\em Rev. Un. Mat. Argentina}, 61(1):63--72, 2020.

\bibitem[CDS00]{Croke-Dairbekov-Sharafutdinov-00}
Christopher~B. Croke, Nurlan~S. Dairbekov, and Vladimir~A. Sharafutdinov.
\newblock Local boundary rigidity of a compact {R}iemannian manifold with
  curvature bounded above.
\newblock {\em Trans. Amer. Math. Soc.}, 352(9):3937--3956, 2000.

\bibitem[CH16]{Croke-Herreros-16}
Christopher~B. Croke and Pilar Herreros.
\newblock Lens rigidity with trapped geodesics in two dimensions.
\newblock {\em Asian J. Math.}, 20(1):47--57, 2016.

\bibitem[CL21]{Cekic-Lefeuvre-21-2}
Mihajlo {Ceki{\'c}} and Thibault {Lefeuvre}.
\newblock {Generic injectivity of the X-ray transform}.
\newblock {\em arXiv e-prints}, page arXiv:2107.05119, July 2021.

\bibitem[Cro90]{Croke-90}
Christopher~B. Croke.
\newblock Rigidity for surfaces of nonpositive curvature.
\newblock {\em Comment. Math. Helv.}, 65(1):150--169, 1990.

\bibitem[Cro91]{Croke-91}
Christopher~B. Croke.
\newblock Rigidity and the distance between boundary points.
\newblock {\em J. Differential Geom.}, 33(2):445--464, 1991.

\bibitem[Cro04]{Croke-04}
Christopher~B. Croke.
\newblock Rigidity theorems in {R}iemannian geometry.
\newblock In {\em Geometric methods in inverse problems and {PDE} control},
  volume 137 of {\em IMA Vol. Math. Appl.}, pages 47--72. Springer, New York,
  2004.

\bibitem[DG16]{Dyatlov-Guillarmou-16}
Semyon Dyatlov and Colin Guillarmou.
\newblock Pollicott-{R}uelle resonances for open systems.
\newblock {\em Ann. Henri Poincar\'{e}}, 17(11):3089--3146, 2016.

\bibitem[DGRS20]{Dang-Guillarmou-Riviere-Shen-20}
Nguyen~Viet Dang, Colin Guillarmou, Gabriel Rivi\`ere, and Shu Shen.
\newblock The {F}ried conjecture in small dimensions.
\newblock {\em Invent. Math.}, 220(2):525--579, 2020.

\bibitem[DS10]{Dairbekov-Sharafutdinov-10}
Nurlan Dairbekov and Vladimir Sharafutdinov.
\newblock Conformal {K}illing symmetric tensor fields on {R}iemannian
  manifolds.
\newblock {\em Mat. Tr.}, 13(1):85--145, 2010.

\bibitem[EL23]{Erchenko-Lefeuvre-23}
Alena {Erchenko} and Thibault {Lefeuvre}.
\newblock {Marked boundary rigidity for surfaces of Anosov type}.
\newblock {\em arXiv e-prints}, page arXiv:2305.06893, May 2023.

\bibitem[FRS08]{Faure-Roy-Sjostrand-08}
Fr\'{e}d\'{e}ric Faure, Nicolas Roy, and Johannes Sj\"{o}strand.
\newblock Semi-classical approach for {A}nosov diffeomorphisms and {R}uelle
  resonances.
\newblock {\em Open Math. J.}, 1:35--81, 2008.

\bibitem[GGJ22]{Bonthonneau-Guillarmou-Jezequel-22}
Yannick {Guedes Bonthonneau}, Colin {Guillarmou}, and Malo {J{\'e}z{\'e}quel}.
\newblock {Scattering rigidity for analytic metrics}.
\newblock {\em arXiv e-prints}, page arXiv:2201.02100, January 2022.

\bibitem[GL19]{Guillarmou-Lefeuvre-18}
Colin Guillarmou and Thibault Lefeuvre.
\newblock The marked length spectrum of {A}nosov manifolds.
\newblock {\em Ann. of Math. (2)}, 190(1):321--344, 2019.

\bibitem[GL21]{Gouezel-Lefeuvre-21}
S\'{e}bastien Gou\"{e}zel and Thibault Lefeuvre.
\newblock Classical and microlocal analysis of the x-ray transform on {A}nosov
  manifolds.
\newblock {\em Anal. PDE}, 14(1):301--322, 2021.

\bibitem[GM18]{Guillarmou-Mazzucchelli-18}
Colin Guillarmou and Marco Mazzucchelli.
\newblock Marked boundary rigidity for surfaces.
\newblock {\em Ergodic Theory Dynam. Systems}, 38(4):1459--1478, 2018.

\bibitem[Gro83]{Gromov-83}
Mikhael Gromov.
\newblock Filling {R}iemannian manifolds.
\newblock {\em J. Differential Geom.}, 18(1):1--147, 1983.

\bibitem[GS94]{Grigis-Sjostrand-94}
Alain Grigis and Johannes Sj\"{o}strand.
\newblock {\em Microlocal analysis for differential operators}, volume 196 of
  {\em London Mathematical Society Lecture Note Series}.
\newblock Cambridge University Press, Cambridge, 1994.
\newblock An introduction.

\bibitem[Gui17a]{Guillarmou-17-1}
Colin Guillarmou.
\newblock Invariant distributions and {X}-ray transform for {A}nosov flows.
\newblock {\em J. Differential Geom.}, 105(2):177--208, 2017.

\bibitem[Gui17b]{Guillarmou-17-2}
Colin Guillarmou.
\newblock Lens rigidity for manifolds with hyperbolic trapped sets.
\newblock {\em J. Amer. Math. Soc.}, 30(2):561--599, 2017.

\bibitem[HMS16]{Heil-Moroianu-Semmelmann-16}
Konstantin Heil, Andrei Moroianu, and Uwe Semmelmann.
\newblock Killing and conformal {K}illing tensors.
\newblock {\em J. Geom. Phys.}, 106:383--400, 2016.

\bibitem[Kli95]{Klingenberg-95}
Wilhelm P.~A. Klingenberg.
\newblock {\em Riemannian geometry}, volume~1 of {\em De Gruyter Studies in
  Mathematics}.
\newblock Walter de Gruyter \& Co., Berlin, second edition, 1995.

\bibitem[{Lef}19a]{Lefeuvre-19-2}
Thibault {Lefeuvre}.
\newblock {Local marked boundary rigidity under hyperbolic trapping
  assumptions}.
\newblock {\em Journal of Geometric Analysis}, 30:448--465, 2019.

\bibitem[Lef19b]{Lefeuvre-thesis}
Thibault Lefeuvre.
\newblock {\em On the rigidity of Riemannian manifolds}.
\newblock 2019.
\newblock PhD thesis, Available online:
  \url{https://thibaultlefeuvre.files.wordpress.com/2019/12/main.pdf}.

\bibitem[Lef19c]{Lefeuvre-19-1}
Thibault Lefeuvre.
\newblock On the s-injectivity of the x-ray transform on manifolds with
  hyperbolic trapped set.
\newblock {\em Nonlinearity}, 32(4):1275--1295, 2019.

\bibitem[Mic82]{Michel-81}
Ren\'{e} Michel.
\newblock Sur la rigidit\'{e} impos\'{e}e par la longueur des
  g\'{e}od\'{e}siques.
\newblock {\em Invent. Math.}, 65(1):71--83, 1981/82.

\bibitem[NS15]{Noakes-Stoyanov-15}
Lyle Noakes and Luchezar Stoyanov.
\newblock Rigidity of scattering lengths and travelling times for disjoint
  unions of strictly convex bodies.
\newblock {\em Proc. Amer. Math. Soc.}, 143(9):3879--3893, 2015.

\bibitem[Ota90a]{Otal-90}
Jean-Pierre Otal.
\newblock Le spectre marqu\'{e} des longueurs des surfaces \`a courbure
  n\'{e}gative.
\newblock {\em Ann. of Math. (2)}, 131(1):151--162, 1990.

\bibitem[Ota90b]{Otal-90-2}
Jean-Pierre Otal.
\newblock Sur les longueurs des g\'{e}od\'{e}siques d'une m\'{e}trique \`a
  courbure n\'{e}gative dans le disque.
\newblock {\em Comment. Math. Helv.}, 65(2):334--347, 1990.

\bibitem[PSU]{Paternain-Salo-Uhlmann-book}
Gabriel~P. Paternain, Mikko Salo, and Gunther Uhlmann.
\newblock {\em Geometric inverse problems, with emphasis on two dimensions}.
\newblock Cambridge University Press.

\bibitem[PSU14]{Paternain-Salo-Uhlmann-14-1}
Gabriel~P. Paternain, Mikko Salo, and Gunther Uhlmann.
\newblock Tensor tomography: progress and challenges.
\newblock {\em Chin. Ann. Math. Ser. B}, 35(3):399--428, 2014.

\bibitem[PU05]{Pestov-Uhlmann-05}
Leonid Pestov and Gunther Uhlmann.
\newblock Two dimensional compact simple {R}iemannian manifolds are boundary
  distance rigid.
\newblock {\em Ann. of Math. (2)}, 161(2):1093--1110, 2005.

\bibitem[Rob80]{Robinson-80}
Clark Robinson.
\newblock Structural stability on manifolds with boundary.
\newblock {\em J. Diff. Eq.}, 37:1--11, 1980.

\bibitem[Rud87]{Rudin-87}
Walter Rudin.
\newblock {\em Real and complex analysis}.
\newblock McGraw-Hill Book Co., New York, third edition, 1987.

\bibitem[Sha94]{Sharafutdinov-94}
Vladimir Sharafutdinov.
\newblock {\em Integral geometry of tensor fields}.
\newblock Inverse and Ill-posed Problems Series. VSP, Utrecht, 1994.

\bibitem[SKL23]{Desimoi-Kaloshin-Leguil}
Jacopo~De Simoi, Vadim Kaloshin, and Martin Leguil.
\newblock Marked length spectral determination of analytic chaotic billiards
  with axial symmetry.
\newblock {\em Invent. math.}, 2023.

\bibitem[SU04]{Stefanov-Uhlmann-04}
Plamen Stefanov and Gunther Uhlmann.
\newblock Stability estimates for the {X}-ray transform of tensor fields and
  boundary rigidity.
\newblock {\em Duke Math. J.}, 123(3):445--467, 2004.

\bibitem[SU09]{Stefanov-Uhlmann-09}
Plamen Stefanov and Gunther Uhlmann.
\newblock Local lens rigidity with incomplete data for a class of non-simple
  {R}iemannian manifolds.
\newblock {\em J. Differential Geom.}, 82(2):383--409, 2009.

\bibitem[SUV21]{Stefanov-Uhlmann-Vasy-21}
Plamen Stefanov, Gunther Uhlmann, and Andr\'{a}s Vasy.
\newblock Local and global boundary rigidity and the geodesic {X}-ray transform
  in the normal gauge.
\newblock {\em Ann. of Math. (2)}, 194(1):1--95, 2021.

\bibitem[Tay11]{Taylor-11}
Michael~E. Taylor.
\newblock {\em Partial differential equations {I}. {B}asic theory}, volume 115
  of {\em Applied Mathematical Sciences}.
\newblock Springer, New York, second edition, 2011.

\bibitem[Var09]{Vargo-09}
James Vargo.
\newblock A proof of lens rigidity in the category of analytic metrics.
\newblock {\em Math. Res. Lett.}, 16(6):1057--1069, 2009.

\bibitem[Vig80]{Vigneras-80}
Marie-France Vign\'{e}ras.
\newblock Vari\'{e}t\'{e}s riemanniennes isospectrales et non isom\'{e}triques.
\newblock {\em Ann. of Math. (2)}, 112(1):21--32, 1980.

\bibitem[Zwo12]{Zworski-12}
Maciej Zworski.
\newblock {\em Semiclassical analysis}, volume 138 of {\em Graduate Studies in
  Mathematics}.
\newblock American Mathematical Society, Providence, RI, 2012.

\end{thebibliography}

\end{document}